\documentclass[12pt, oneside]{amsart}

\usepackage{amssymb}
\usepackage{amsfonts}
\usepackage{amstext}
\usepackage{quiver}

\usepackage{soul}

\usepackage{amsmath, amsthm}
\usepackage{mathrsfs}

\usepackage{bbm}

\usepackage{stmaryrd}

\usepackage{tensor}

\usepackage{geometry}
\usepackage{marginnote}

\usepackage{pdfsync}

\usepackage{hyperref}

\usepackage{cleveref}

\usepackage[title,titletoc]{appendix}

\usepackage{changepage} 

\usepackage{tikz} 
\usetikzlibrary{arrows,decorations.pathmorphing,backgrounds,positioning,fit,petri,shapes.misc,decorations.markings,shapes,shapes.multipart,calc,cd}

\usepackage{mathtools}

\newcommand{\spar}
[1]{\medskip\refstepcounter{subsection}\noindent\arabic{section}.\arabic{subsection}.\label{#1}}

\tikzset{
	plain trans/.style = {circle, minimum size = 14, inner sep=2, draw, font=\scriptsize}}
\tikzset{
	plain inv trans/.style = {circle, minimum size = 18, inner sep=1, draw, font=\tiny}}
\tikzset{
	plain dual trans/.style = {circle, minimum size = 14, inner sep=2, draw, font=\scriptsize}}
\tikzset{
	plain inv dual trans/.style = {circle, minimum size = 18, inner sep=1, draw, font=\tiny}}

\tikzset{trans/.style={circle, minimum size = 14, inner sep=1, draw, font=\scriptsize,
	path picture={%
	\filldraw[anchor=center,black]
	($(path picture bounding box.south east)!0.15!(path picture bounding box.north east)$) to [out=-135, in=0]	(path picture bounding box.south) to [out=180, in=-45]
	($(path picture bounding box.south west)!0.15!(path picture bounding box.north west)$) to cycle;
      }}}
      
\tikzset{inv trans/.style={circle, minimum size = 14, inner sep=1, draw, font=\tiny,
	path picture={%
	\filldraw[anchor=center,black]
	($(path picture bounding box.south east)!0.15!(path picture bounding box.north east)$) to [out=-135, in=0]	(path picture bounding box.south) to [out=180, in=-45]
	($(path picture bounding box.south west)!0.15!(path picture bounding box.north west)$) to cycle;
      }}}      

\tikzset{dual trans/.style={circle, minimum size = 14, inner sep=1, draw, font=\scriptsize,
	path picture={%
	\filldraw[anchor=center,black]
	($(path picture bounding box.north east)!0.15!(path picture bounding box.south east)$) to [out=135, in=0]	(path picture bounding box.north) to [out=180, in=45]
	($(path picture bounding box.north west)!0.15!(path picture bounding box.south west)$) to cycle;
      }}}      

\tikzset{inv dual trans/.style={circle, minimum size = 14, inner sep=1, draw, font=\tiny,
	path picture={%
	\filldraw[anchor=center,black]
	($(path picture bounding box.north east)!0.15!(path picture bounding box.south east)$) to [out=135, in=0]	(path picture bounding box.north) to [out=180, in=45]
	($(path picture bounding box.north west)!0.15!(path picture bounding box.south west)$) to cycle;
      }}}      

\tikzset{
	func/.style = {font=\scriptsize},
	in func/.style = {font=\scriptsize, above, text opacity=0},
	in func visible/.style = {font=\scriptsize, above},
	out func/.style = {font=\scriptsize, below, text opacity=0},	
	out func visible/.style = {font=\scriptsize, below},	
	mid func/.style = {font=\scriptsize, inner sep=1, fill=white} 
}

\tikzset{
	cat/.style = {font=\scriptsize}}
\tikzset{
	frame/.style = {double,rounded corners}}

\tikzset{
	paste/.style = {dashed,opacity=0.5}}

\tikzset{
    dir/.style={},
    dir end/.style={},    
    dual/.style={},
    eq/.style={inner sep=0pt},
    plain/.style={} 
}

\newcommand{\Q}{\mathbb{Q}}

\newcommand{\m}{\mathfrak{m}}
\renewcommand{\a}{\mathfrak{a}}

\newcommand{\Ddual}{\mathfrak{D}}
\newcommand{\Hom}{\operatorname{Hom}}
\newcommand{\Map}{\operatorname{Map}}

\newcommand{\Fun}{\operatorname{Fun}}
\newcommand{\End}{\operatorname{End}}

\newcommand{\iso}{\overset{\sim}{\longrightarrow}}
\renewcommand{\=}{\cong}
\newcommand{\Id}{\operatorname{Id}}

\newcommand{\invlim}{\operatorname{\underleftarrow{lim}}}
\newcommand{\colim}{\operatorname{\underrightarrow{colim}}}

\newcommand{\Set}{\operatorname{Set}}

\newcommand{\xto}{\xrightarrow}

\newcommand{\set}[1]{\left\{{#1}\right\}}

\newcommand{\al}{\alpha}

\newcommand{\M}{\mathcal{M}}
\newcommand{\A}{\mathcal{A}}

\newcommand{\E}{\mathcal{E}}

\newcommand{\C}{\mathcal{C}}
\newcommand{\D}{\mathcal{D}}

\newcommand{\N}{\mathcal{N}}

\newcommand{\Xx}{\mathcal{X}}
\newcommand{\Yy}{\mathcal{Y}}

\newcommand{\LM}{\mathcal{LM}}
\newcommand{\Alg}{\operatorname{Alg}}
\newcommand{\CAlg}{\operatorname{CAlg}}
\newcommand{\LMod}{\operatorname{LMod}}
\newcommand{\Ass}{\operatorname{Assoc^\otimes}}
\newcommand{\Bmod}{\operatorname{BMod}}
\newcommand{\BMod}[2]{\tensor*[_{#1}]{\Bmod}{_{#2}} }
\newcommand{\BLMod}{\operatorname{BLMod}}
\renewcommand{\Bar}{\operatorname{Bar}}

\newcommand{\cosk}{\operatorname{cosk}}
\newcommand{\PSh}{\operatorname{PSh}}
\newcommand{\Dec}{\operatorname{Dec}}

\mathchardef\mhyphen="2D 

\newcommand{\Ss}{\mathcal{S}}

\newcommand{\adj}{\rightleftarrows}

\newcommand{\op}{{^\mathrm{op}} }
\newcommand{\Psh}{\operatorname{Psh}}

\newcommand{\Tens}{\operatorname{Tens}}
\newcommand{\one}{\mathbbm{1}}
\newcommand{\Mon}{{\mathrm{Mon}}}

\newcommand{\Cat}{\mathcal{C}\mathrm{at}}

\newcommand{\ol}[1]{\overline{#1}}
\newcommand{\inv}{^{-1}}

\newcommand{\Fin}{\mathbf{Fin}}

\newtheorem{thm}{Theorem}[section]
\numberwithin{thm}{subsection}

\newtheorem{prop}[thm]{Proposition}
\newtheorem{cor}[thm]{Corollary}
\newtheorem{lem}[thm]{Lemma}
\newtheorem{mydef}[thm]{Definition}

\theoremstyle{definition}

\newtheorem{rem}[thm]{Remark}
\newtheorem{ex}[thm]{Example}
\newtheorem{notation}[thm]{Notation}
\newtheorem{construction}[thm]{Construction}
\newtheorem{terminology}[thm]{Terminology}

\newcommand{\cof}{\hookrightarrow}

\newcommand{\fib}{\twoheadrightarrow}

\newcommand{\pbcorner}{\arrow[dr, phantom, "\ulcorner" description, very near start]}

\newcommand{\POp}{\mathcal{P}\mathrm{Op}_\infty}

\newcommand{\Oo}{\mathcal{O}}

\newcommand\noloc{
  \nobreak
  \mspace{6mu plus 1mu}
  {:}
  \nonscript\mkern-\thinmuskip
  \mathpunct{}
  \mspace{2mu}
}

\newcommand{\opnm}{\operatorname}
\newcommand{\cl}{\mathcal}

\newcommand{\Cc}{\cl{C}}
\newcommand{\TwArr}{\opnm{TwArr}}

\newcommand{\la}{\lambda}
\newcommand{\ep}{\epsilon}
\newcommand{\mar}{\marginpar}

\newcommand{\ka}{\kappa}

\newcommand{\from}{\leftarrow}

\newcommand{\coAlg}{\opnm{coAlg}}

\newcommand{\DK}{\Ddual_\mathrm{Ksz}}
\newcommand{\DKpr}{\Ddual'_\mathrm{Ksz}}

\newcommand{\foot}{\footnote}

\newcommand{\Assoc}{\mathrm{Assoc}^\otimes}

\newcommand{\new}{\newcommand}

\new{\BM}{\mathcal{BM}}

\new{\angled}[1]{\langle #1 \rangle}

\new{\mM}{\mathfrak{m}}

\new{\nato}{{\vec \nabla}}

\new{\be}{\beta}

\new{\Aa}{\mathcal{A}}

\new{\Dd}{\mathcal{D}}

\new{\fF}{\mathfrak{f}}
\new{\aA}{\mathfrak{a}}

\new{\MAss}{\mathrm{MAss}}

\new{\Mm}{\M}

\new{\Ee}{\mathcal{E}}
\new{\Ff}{\mathcal{F}}
\new{\Gg}{\mathcal{G}}

\new{\lan}{\langle}
\new{\ran}{\rangle}
\new{\LAlg}{\opnm{LAlg}}

\new{\comment}[1]{***\footnote{Temp: #1}\mar{see foot}}

\new{\oo}{$\infty$}

\new{\Mul}{\opnm{Mul}}

\new{\qq}{\langle}
\new{\pp}{\rangle}
\new{\Si}{\Sigma}
\new{\Finst}{\Fin_*}
\new{\ArDopconv}{\opnm{Ar}(\Delta)_{\rm{convex}}^\op}
\new{\Cut}{\opnm{Cut}}
\new{\Cleave}{\opnm{Cleave}}
\new{\Prebar}{\opnm{Prebar}}

\title[Relative tensor products]{Relative tensor products and Koszul duality in monoidal $\infty$-categories}

\author{Ishai Dan-Cohen, Asaf Horev}
\thanks{\textbf{Acknowledgement of support:} This research was supported by ISF grants 726/17 and 621/21 to I. Dan-Cohen and by ERC grant ERC-2017-STG 75908 to Dan Petersen. \textbf{Political disclaimer:} Both authors ask not to be considered responsible for the actions of any government that does not fully embrace the principles of democracy. }
\date{\today}

\begin{document}
\maketitle

\begin{abstract}
This semi-expository work covers central aspects of the theory of relative tensor products as developed in \cite{HA}, as well as their application to Koszul duality for algebras in monoidal $\infty$-categories. Part of our goal is to expand on the rather condensed account of loc. cit. Along the way, we generalize various aspects of the theory. For instance, given a monoidal $\infty$-category $\Cc$, an $\infty$-category $\Mm$ which is left-tensored over $\Cc$, and an algebra $A$ in $\Cc$, we construct an action of $A$-$A$-bimodules $N$ in $\Cc$ on left $A$-modules $M \in \Mm$ by an \emph{external relative tensor product} $N \otimes_A M$. (Up until now, even the special (``internal'') case $\Cc = \Mm$ appears to have escaped the literature.) 

As an application, we generalize the Koszul duality of loc. cit. to include modules. Our straightforward approach requires that we at this point assume certain compatibilities between tensor products and limits; these assumptions have recently been shown to be unnecessary in \cite[\S3]{BrantnerCamposNuiten}.
\end{abstract}


\setcounter{tocdepth}{2}
\tableofcontents

\section{Introduction}
\label{SectionIntro}


Let $\C^\otimes$ be a monoidal $\infty$-category and let $A$ be an algebra object in $\C^\otimes$. 
Koszul duality for algebras in monoidal $\infty$-categories aims to set up an adjunction 
\[
\tag{A}
\Alg^{\opnm{aug}}(\C) \overset{\Ddual_\mathrm{Ksz}} 
{\underset{\Ddual'_\mathrm{Ksz}}  \rightleftarrows}
\opnm{coAlg}^{\opnm{aug}}(\C)
\]
between the category of augmented algebras in $\C$ and augmented coalgebras in $\C$ which generalizes the classical bar-cobar adjunction used, for instance, in delooping and recognition theorems in classical topology~\cite{MayLoops}, and developed in the setting of dga's, for instance, by Loday-Vallette in~\cite{LodayVallette}.
\footnote{More recent related works which should be compared to the present one include ~\cite{FrancisGaitsgory, calaque2020moduli, pridham2018tannaka, PositsTwoKinds, BrantnerCamposNuiten}}
\footnote{Let $A \in \Alg^{\opnm{aug}}(\C)$. We warn the reader that some authors, who restrict attention to situations where a coalgebra can be turned into an algebra simply by taking linear duals, use the term ``Koszul'' to refer to the linear dual of our $\Ddual_\mathrm{Ksz}(A)$.}
In particular, if $A$ is an augmented algebra object in $\C$, then the object of $\C$ underlying $\DK(A)$ should be given by the relative tensor product $\one \otimes_A \one$.


The adjunction above can be upgraded to include modules. This generalizes the classical constructions which, indeed, include path-spaces (and not only loop-spaces) in their scope. It also relates to the Koszul duality for modules of Positselski \cite{PositsTwoKinds}, but in the general abstract setting of monoidal $\infty$-categories. For example, for a reasonably nice monoidal \oo-category $\C^\otimes$, we obtain adjoint functors
\footnote{In fact, we may work more generally with an $\LM$-monoidal category; see  \cref{BarCobarAdjunction-XXXversion}  for the precise statement.}
\[
\tag{M}
\LMod^{\opnm{aug}}(\C)
\overset{\DK} 
{\underset{\DKpr}  \rightleftarrows}
\opnm{coLMod}^{\opnm{aug}}(\C).
\]
Here $\LMod^{\opnm{aug}}(\C)$ is the $\infty$-category of pairs $(A,X)$ consisting of an augmented algebra $A \to \one$ and an $A$-module $X$,
and
$\opnm{coLMod}^{\opnm{aug}}(\C) = \LMod^{\opnm{aug}}(\C^\op)^\op$; 
see \cref{cons:C_1} for the precise definition.
Moreover, 
\[
\DK(A,X)
\]
witnesses a left coaction of the augmented coalgebra $\DK(A) = \one \otimes_A \one$ on the object $\one \otimes_A X \in \C$.

We learned how to construct the adjunction (A) from section 5.2.2 of Lurie \cite{HA}. Lurie's construction (and hence ours), revolves around an augmented algebra 
\[
M = (A \to \one)
\]
in $\C$ (in which $\one$ is temporarily assumed to be terminal
\footnote{More precisely, we work with an $\LM$-monoidal category $\Cc^\otimes$ in which $\one$ is (temporarily) assumed to be terminal in $\Cc_\aA$.}),
regarded as an algebra object of the twisted arrow category $\M = \TwArr(\C)$. Heavy use is made of (at least certain aspects of) a monoidal structure on the category $\BMod{M}{M}(\M)$ of $M$-$M$-bimodules in the twisted arrow category. In our approach, for the upgrade to modules, we also need a left action of this category on the category $\LMod_M(\M)$ of left $M$-modules in the twisted arrow category. Many of the pages below are devoted to the construction of this expected action. 

As we learned from the work of Brantner--Campos--Nuiten \cite[\S3]{BrantnerCamposNuiten}, it is in fact possible to dodge the construction of this left-action, and to derive the Koszul duality for modules (M) from the Koszul duality for algebras (A). Moreover, in their indirect approach, it's possible to avoid certain technical assumptions concerning the compatibility of tensor products with limits, which we have been unable to avoid in our direct approach. Nevertheless, we hope that our foundational work, and our straightforward approach to the case of modules, helps to paint a coherent and conceptually satisfying picture, even if one that becomes unavailable (or undesirable) beyond a certain level of generality.\footnote{An earlier draft of this paper contained an erratum pertaining to an inaccuracy in Lemma 5.2.2.29 of \cite{HA}. It's precisely here that our need to impose extra assumptions first arises; this need becomes stronger, in fact, when we switch to working with the $\LM$-operad. Since the appearance of \cite{BrantnerCamposNuiten}, however, we've decided to deemphasize this point.}

Let $\Cc$ be a monoidal \oo-category and let $\Ee$ be an \oo-category left-tensored over $\Cc$. Let $A$ be an algebra object in $\Cc$, let $E$ be a left $A$-module in $\Ee$, let $B$ be a coalgebra object in $\Cc$ and let $F$ be a left $B$-comodule in $\Ee$. A \emph{twisted homomorphism from $(A,E)$ to $(B,F)$} consists of a morphism $A \to B$ in $\Cc$, and a morphism $E \to F$ in $\Ee$ which, together, intertwine the algebra and module structure of $(A,E)$ with the coalgebra and comodule structure of $(B,F)$ up to coherent higher homotopy (see \cref{260727A} for a precise definition). The Koszul dual $\DK(A,E)$ may be characterized as the target of the universal twisted homomorphism from the pair $(A,E)$ to a pair consisting of a coalgebra in $\Cc$ and a comodule in $\Ee$ as above, provided such an object exists. Our main application to Koszul duality may be stated informally as follows. We assume for simplicity that the unit object of $\Cc$ is both initial and terminal. We must also assume that $\Cc$, $\Ee$ admit certain limits and colimits, and that tensor products preserve certain limits and colimits.
 
\begin{thm}
Every pair $(A,E)$ as above admits a Koszul dual $\DK(A,E)$. The assignment $(A,E) \mapsto \DK(A,E)$ may be upgraded to a functor $\DK$. Moreover, the functor $\DK$ admits a right adjoint $\DK'$ which may be characterized in similar (dual) terms. 
\end{thm}

\noindent
See \cref{BarCobarAdjunction-XXXversion} for the precise (and somewhat more general) statement. 

Our paper includes two overviews. 
One, which begins presently, explains the structure of the paper in broad outline: what are the main results, and how they fit together. 
The second, in \cref{SectionOverview}, provides an informal overview of core aspects of Lurie's construction and of ours; this section functions as a second introduction and is not logically needed for what follows. In section \ref{infop} we review the foundations of the theory of \oo-operads. 

In \cref{sec:Morita_double_cat} we refine Lurie's treatment of the relative tensor product and its associativity structure. By a \emph{multimodule} we mean a sequence of algebras $A_0, \dots, A_n$ and for each $ 0 \le i < n$ an $A_i$-$A_{i+1}$-bimodule $M_{i,i+1}$ (\cref{29.1}). Recall that when $\C^\otimes$ is a monoidal category in which $\otimes$ respects realizations, there is an associated double $\infty$-category 
\[
\Bmod(\C)^\circledast \to \Delta^\op
\]
whose 
fiber over $[n]$ is the $\infty$-category of
multimodules as above. We check that this construction can be refined in the following way: dropping the blanket assumption on compatibility with realizations, one may restrict attention to the class $\A$ of algebras in $\C$ which are \emph{bar compatible}, in the sense that for any $A \in \A$ and $A$-bimodules $X,Y$, the geometric realization of the bar construction $\Bar_A(X,Y)_\bullet$ exists and is compatible with $\otimes$.\footnote{As we explain below, the need for this arrises from the fact that such blanket assumptions are not inherited by twisted arrow categories.} 
In this way we obtain a double $\infty$-category 
\[
\Bmod_\A(\C)^\circledast \to \Delta^\op
\]
whose higher vertical morphisms are \emph{$\A$-multimodules}, by which we mean that the algebras $A_i$ belong to $\A$. The main result of this section, which amounts to a mild refinement of~\cite[prop. 4.4.3.9]{HA}, is our \cref{link1}.

Given $A \in \A$, we use $\Bmod_\A(\C)^\circledast$ in \cref{sec:BLMod} to endow $\BMod{A}{A}(\C)$ with a monoidal structure given by relative tensor product over $A$. We also construct  an $\LM$-monoidal category
\[
\BLMod_A(\C)^\otimes \to \LM^\otimes
\]
which witnesses a left action of the monoidal $\infty$-category $\BMod{A}{A}(\C)^\otimes$ of $A$-$A$-bimodules on the category $\LMod_A(\C)$ of left $A$-modules, the action being given on the level of individual objects by the relative tensor product. \footnote{Actually, our results in this section are somewhat more general: given an arbitrary double $\infty$-category and objects $A$ and $B$, we construct an action of the monoidal $\infty$-category of horizontal endomorphisms of $A$ on the $\infty$-category of horizontal morphisms from $B$ to $A$.} 

\Cref{sec:External} concerns a variant of the results of sections \ref{sec:Morita_double_cat} and \ref{sec:BLMod} in which $\Cc^\otimes$, instead of being monoidal, is $\LM$-monoidal, and $A \in \Alg(\Cc_\a)$. 
Under suitable assumptions on $A$, we construct a left-action of the monoidal $\infty$-category $\BMod{A}{A}(\Cc_\a)$ of $A$-$A$-bimodules in $\Cc_\a$ on the $\infty$-category $\LMod_A(\Cc)$ of left $A$-modules in $\Cc_\mM$. 
Strictly speaking, this construction generalizes the analogous construction from the previous two sections, and it is this generalized version that is needed in the sequel.
Indeed, even if one is interested in applying Koszul duality for modules to a monoidal category
\[
\Cc^\otimes \to \Ass,
\]
if the unit object $\one \in \Cc$ is not terminal, the methods of sections \ref{SectionKoszLM} and \ref{sec:Twisted_arrow_bar_construction} necessitate that we work with an associated $\LM$-monoidal category which witnesses $\Cc$ as left-tensored over the monoidal slice category $\Cc^\otimes_{/\one}$.
\footnote{Thus, in some sense sections \ref{sec:Morita_double_cat} and \ref{sec:BLMod} treat a special case before moving on to the generality we need in section \ref{sec:External}. 
This may not be logically efficient, but we feel that what we lose in logical efficiency we gain in clarity of exposition. Additionally, this gives us an opportunity to discuss the closely related construction of the action of horizontal endomorphisms on horizontal morphisms in the general setting of a double $\infty$-category.} 

In \cref{sec:LM_monoidal_pairings} we review $\LM$-monoidal categories and $\LM$-monoidal pairings in preparation for \cref{SectionKoszLM}.
\Cref{LM52227}, which takes place in the setting of general $\LM$-monoidal pairings and generalizes Proposition 5.2.2.27 of~\cite{HA}, may reasonably be regarded as the main theorem on Koszul duality for modules, but in an abstract, axiomatic form. As hinted above, \cref{sec:Overview} is devoted to an overview of its proof, preceded by an overview of the proof of the analogous Theorem 5.2.2.27 of~\cite{HA}. 

In \cref{sec:Twisted_arrow_bar_construction} we specialize from the generality of $\LM$-monoidal pairings to the case of the twisted arrow category. 
We show that if $\C^\otimes$ is an $\LM$-monoidal $\infty$-category and $M$ is an algebra in the twisted arrow category of $\Cc_\aA$ which lies over algebras $A$ in $\C_\aA$ and $B$ in $\C^\op_\aA$ with $B$ trivial, then $\TwArr(\C)$ \emph{admits realizations of $M$-bar constructions $\otimes$-compatibly} (see \cref{def:admits_A_bar}).

We arrive at our $\LM$-monoidal version of~\cite[thm. 5.2.2.17]{HA} in our \cref{BarCobarAdjunction}. 
This completes the construction of the $\LM$-monoidal bar-cobar adjunction under the assumption that the unit object of $\C$ is both initial and terminal. 
It remains to dispense with this assumption and, more significantly, to compare the result with the relative tensor product $\one \otimes_A X$ mentioned above. 
The final result is summarized in \cref{BarCobarAdjunction-XXXversion}. 

Let us indicate how \cref{BarCobarAdjunction-XXXversion} may be applied, for instance, to the study of commutative algebras in motives. This is meant only to point in the direction of future applications --- more significant applications will require more work. Let $Z$ be a Noetherian scheme and let $\D^\otimes = \mathrm{DA}^{\mathrm{\acute{e}t}}(Z,\Q)^\otimes$ denote the monoidal derived $\infty$-category of \'etale motives over $Z$ with $\Q$ coefficients; see, for instance, Ayoub \cite[\S2.1.1]{ayoub2014algebre} for a model-categorical construction. The $\infty$-category $\D^\otimes$ is presentably symmetric monoidal; hence $ \C^\otimes = \CAlg \D$ with its coCartesian symmetric monoidal structure admits geometric realizations of bar constructions $\otimes$-compatibly, and \cref{BarCobarAdjunction-XXXversion} applies to produce a functor 
\[
\Ddual_\mathrm{Ksz} =    
\Ddual_{\LMod(\la)}^\op \colon 
\LMod^{aug}(\C) 
\to \opnm{coLMod}^{aug}(\C).
\]
Let $A \in \C$ be a commutative algebra in $\D$ and let
\[
x,y: A \rightrightarrows \one
\]
be two augmentations. In view of the equivalence
\[
\Alg \C \simeq \C,
\]
$(A,x)$ may be lifted to an object of $\Alg^{aug}(\C)$. We may then view $\one$ as a left module via $y$ --- we write $_y\one$ for this structure. Together, we obtain an object
\[
(A \xto{x} \one, {_y\one})
\]
of $\LMod^{aug}(\C)$. Applying $\Ddual_\mathrm{Ksz}$, we obtain an augmented  coalgebra structure on $\one \otimes_{x,A,x} \one$ together with a coaction on $\one \otimes_{x,A,y} \one$ which makes the latter into a torsoric comodule in $\C = \CAlg \D$. 

Suppose $A$ is the dual algebra of the coalgebra associated to a smooth $Z$-scheme $X$, and $x,y$ are the augmentations associated to $Z$-points of $X$.
Suppose, moreover, that all motives in sight belong to a triangulated subcategory possessing a \emph{motivic} t-structure (alternatively, replace motives with any realization). Then 
\[
\pi_1^\mathrm{un}(X,x) = \opnm{Spec} H^0(\one \otimes_{x,A,x} \one)
\]
is the unipotent fundamental group of $X$ at $x$ and 
\[
\pi_1^\mathrm{un}(X,x,y) = \opnm{Spec} H^0(\one \otimes_{x,A,y} \one)
\]
is the torsor of unipotent homotopy classes of paths from $y$ to $x$; see, for instance \cite{DCShRatMot}, \cite{iwanari2017motivic}. \foot{In this setting, following Deligne \cite{deligne1989groupe}, ``$\opnm{Spec}$'' just means that we view the given object as an object of the opposite Tannakian category.} Thus 
\[
\Ddual_\mathrm{Ksz}(A \xto{x} \one, {_y\one})
\]
lifts the torsor structure of $\pi_1^\mathrm{un}(X,y,x)$ to the level of highly structured algebras in motives.


\subsection*{Acknowledgements}
We would like to thank Jonathan Beardsley, Lukas Brantner, Rune Haugseng, Marc Hoyois, Thomas Nikolaus, Joost Nuiten, Piotr Pstragowski, Dan Petersen, Jon Pridham, Tomer Schlank, and Lior Yanovski for helpful conversations and email exchanges concerning the subject of this paper. We wish to express our heartfelt gratitude to the referee for a very careful reading and many helpful comments.

\section{Overview}
\label{sec:Overview}


\label{SectionOverview}

We begin with an overview of the proof of Theorem 5.2.2.27 of Lurie \cite{HA}, followed by an overview of our proof of our analogous \cref{LM52227}.

\subsection{Review: Koszul duality and bimodules}
\label{sec:review_duality_bimodules}
Let us first explain the relationship between the coalgebra structure of the Koszul dual and the monoidal structure on bimodules. Let $\C^\otimes$ be a symmetric monoidal $\infty$-category, and assume that $\C$ admits geometric realizations of simplicial objects and that the tensor product functor
\[
\otimes \colon \C \times \C \to \C
\]
preserves geometric realizations of simplicial objects separately in each variable. 
Under these assumptions, bimodules in $\C$ admit relative tensor products ~\cite[sec. 4.4.3]{HA}. It is then easy to describe the underlying object of $\DK(A)$; it is given by the relative tensor product 
$\one \otimes_A \one$.
To see the coalgebra structure, consider
\[
\one \otimes_A \one \simeq \one \otimes_A A \otimes_A \one
\]
as the image of $A$ under the functor 
\begin{equation}\label{eq:free_obj_on_A_bimodule}
\one \otimes_A - \otimes_A \one \colon \BMod{A}{A}(\C)^\op \to \C^\op.
\end{equation}
Being the unit of a monoidal category, $A \in \BMod{A}{A}(\C)^\op$ admits an essentially unique algebra structure.
If the functor~\eqref{eq:free_obj_on_A_bimodule} is lax monoidal then the induced functor 
\[
\one \otimes_A - \otimes_A \one \colon \Alg \left( \BMod{A}{A}(\C)^\op \right) \to \Alg \left( \C^\op \right),
\]
equips the Koszul dual $\one \otimes_A A \otimes_A \one \simeq \one \otimes_A \one$ with a coalgebra structure. Let us point out two salient features of the constructions that follow.
\begin{enumerate}
\item The lax monoidality is additional \emph{structure} on the functor~\eqref{eq:free_obj_on_A_bimodule}, which depends on the choice of augmentation $ A \to \one$. The constructions of~\cite[sec. 5.2.2]{HA} may be thought of in part as the construction of this extra structure.
\item It is not immediately clear how to turn the above construction into a functor 
\[
\DK \colon \Alg^\mathrm{aug}(\C)^\op \to \Alg^\mathrm{aug}(\C^\op),
\]
since different algebras require us to examine different bimodule categories. The theory of monoidal pairings of ~\cite[sec. 5.2.1]{HA} helps to solve this problem.
\end{enumerate}

\subsection{Review: functorial characterization of Koszul duality for algebras}

Let $\C^\otimes$ be a monoidal $\infty$-category. 
Lurie begins by formulating a functorial characterization of the coalgebra $\DK(A)$ associated to an algebra $A$ in $\C$, which may be summarized as follows. After possibly replacing $\C$ by $\C_{/\one}$ we may assume the unit object is terminal. 
There's a right fibration 
\[
\la = \la_\Cc: \TwArr(\C) \to \C \times \C^\op
\]
which classifies the Hom-functor $\la_\Cc\inv(A,B) = \Hom_\C(A,B)$, so an object of the \emph{twisted arrow category} $\TwArr(\C)$ over $(A,B)\in \C \times \C^\op$ is a map $A \to B$ in $\C$. The right fibration $\la_\Cc$ is an example of a \textit{pairing}, i.e. a right fibration of the form
\[
\M \to \C \times \D,
\]
and we refer to $\la_\Cc$ as the \emph{twisted arrow pairing of $\Cc$}. The monoidal structure on $\C$ endows the twisted arrow category $\TwArr(\C)$ with a monoidal structure, and $\la_\Cc$ may be upgraded to a pairing between categories of algebras
\[
\tag{*}
\Alg(\la): \Alg \TwArr(\C)
\to
\Alg( \C) \times \Alg(\C^\op).
\]
The pairing $\Alg(\la)$ induces a functor
\[
\hat \Ddual_{\Alg(\la)}: \Alg(\C)^\op \to \PSh \Alg(\C^\op)
\]
to the category of presheaves, and the \emph{Koszul dual}
$\DK(A)$ of $A$ may be defined as \emph{the} object of $\Alg(\C^\op)$ representing $\hat \Ddual_{\Alg(\la)}(A)$ (unique up to essentially unique equivalence), provided that such an object exists. 

As such, $\DK(A)$ comes equipped with an object $M^\mathrm{univ} \in \Alg \TwArr(\C)$ lying over the pair $\big( A, \DK(A) \big)$. Unwinding definitions, we find that $M^\mathrm{univ}$ defines a morphism $A \to \DK(A)$ which intertwines the algebra structure of $A$ with the coalgebra structure of $\DK(A)$. Thus, for instance, $M^\mathrm{univ}$ includes the data of a coherently commuting diagram
\begin{equation*}
\begin{tikzcd}
A \otimes A  \ar[r] \ar[d] & \DK(A) \otimes \DK(A)
\\
A  \ar[r] & \DK(A). \ar[u]
\end{tikzcd}
\end{equation*}

\begin{mydef}
\label{260826A}
Let $\Cc^\otimes$ be a monoidal \oo-category, let $\Alg(\la)$ be the associated paring between categories of algebras as in (*), let $A \in \Alg(\Cc)$ be an algebra in $\Cc$ and let $B \in \Alg(\Cc^\op)$ be a coalgebra in $\Cc$. A \emph{twisted homomorphism from $A$ to $B$} is an object of $\Alg \TwArr(\Cc)$ lying over $(A,B)$. 
\footnote{N.B. Similar terminology that conflicts with ours appears elsewhere in the literature.}
\end{mydef}

With this terminology, we may characterize $\DK(A)$ as the target of the \textit{universal twisted homomorphism from the algebra $A$ to a coalgebra}.  

If $\one \in \C$ is initial, we may reverse direction. 
The pairing $\Alg(\la)$ also induces a functor 
\[
\PSh \Alg(\C) \from \Alg(\C^\op)^\op: \hat \Ddual'_{\Alg(\la)},
\]
and the Koszul dual $\DKpr(B)$ of $B$ may be defined as the object of $\Alg(\C)$ representing $\hat \Ddual'_{\Alg(\la)}(B)$, provided that such an object exists. 

If $\one \in \C$ is both initial and terminal, and if Koszul duals in both directions are representable, then it follows formally that $\DK^\op$ is left adjoint to $\DKpr$. 

\subsection{Review: construction of Koszul duality for algebras} \label{sec:review_duality_algebras}

We assume $\one \in \C$ to be terminal. 
In rough outline, the construction may be seen to proceed in three main steps.
\begin{enumerate}
\item
We upgrade the (essentially unique) augmentation of $A$ to a twisted homomorphism $M: A \to \one$. 
\item
We construct an $M$-$M$-bimodule in the twisted arrow category $M^\mathrm{univ}$ which is in a suitable sense universal over the $A$-$A$-bimodule $A$. 
\item
We promote $M^\mathrm{univ}$ to an associative algebra in the twisted arrow category.
\end{enumerate} 

It's convenient to carry out these steps in a somewhat more general, hence more flexible setting. A pairing of $\infty$-categories
\[
\la: \M \to \C \times \D
\]
is \emph{left representable} if the associated functor 
\[
\hat \Ddual_\la : \C^\op \to \PSh \D
\]
factors through the Yoneda embedding of $\D$. For example, the pairng $\la$ associated to the twisted arrow category above is representable --- indeed, in this case $\hat \Ddual_\la$ \textit{is} the Yoneda embedding of $\C^\op$. There's a natural notion of \emph{monoidal pairing} of monoidal $\infty$-categories
\[
\la^\otimes: \M^\otimes \to \C^\otimes \times_{\Ass} \D^\otimes
\]
\cite[Definition 5.2.2.20]{HA}. A monoidal pairing $\la^\otimes$ gives rise to a pairing
\[
\Alg(\la): \Alg(\M) \to \Alg(\C) \times \Alg(\D)
\]
of categories of algebras. Theorem 5.2.2.27 of~\cite{HA} states that under certain assumptions on $\la^\otimes$, $\Alg(\la)$ is left representable, hence gives rise to a functor
\[
\Ddual_{\Alg(\la)}: (\Alg \C)^\op \to \Alg \D.
\]

One assumption on $\la^\otimes$ that will intervene in our summary is that the underlying pairing $\la$ is left representable, hence gives rise to a functor 
\[
\Ddual_{\la} : \C^\op \to \D.
\]
Another assumption ensures the existence of an essentially unique algebra $M \in \Alg(\M)$ lifting $(A, \one_\D)$. 
There is then an induced monoidal pairing
\[
\la_M \colon {\BMod{M}{M}(\M)}^\otimes \to 
{\BMod{A}{A}(\C)}^\otimes \times_{\Ass} \D^\otimes
\]
of categories of bimodules. If $F = A \otimes E \otimes A$ is the free $A$-$A$-bimodule generated by $E \in  \C$, then its $\la_M$-dual $\hat \Ddual_{\la_M}(F) \in \Psh  \D$ with respect to the pairing $\la_M$ is represented by
\[
\Ddual_{\la_M}(F) = \Ddual_{\la} (E).
\]
However, while $A$ is the unit object of ${\BMod{A}{A}}$, it is \emph{not free} as $A$-$A$-bimodule. 
Using free resolutions, one constructs an $M$-$M$-bimodule $M^\mathrm{univ}$ lying over a pair $(A,B)$ which witnesses the fact that the presheaf $\hat \Ddual_{\la_M}(A) \in \PSh(\D)$ is represented by an object $B \in \D$ as in step (2) above. 
Since $A$ has the structure of a coalgebra in $A$-$A$-bimodules, and since $\hat \Ddual_{\la_M}$ is in a natural way lax-monoidal, $B =  \Ddual_{\la_M}(A) \in \D$ inherits an algebra structure, as witnessed by an algebra structure on $M^\mathrm{univ}$. 
Dispensing with the spent bimodule structures, we obtain an algebra object $M^\mathrm{univ}$ of $\M^\otimes$ witnessing the fact that $B$ represents $\Ddual_{\Alg(\la)}(A)$.

\subsection{Review: comparison with the relative tensor product}
\label{comprev}
We now return to the case of the twisted arrow category, and relate the functor $\Ddual_{\Alg(\la)}$ to the description given in \cref{sec:review_duality_bimodules}.
Suppose $\D = \C^\op$, $\M = \TwArr(\C)$ is the twisted arrow category, and $M \in \Alg(\M)$ witnesses the augmentation $A\to \one$. 
Then there's a natural augmentation $M \to \one$.
Pullback of objects of $\M$ (regarded as $\one$-$\one$-bimodules) along the augmentation, gives rise to a \textit{morphism $f$ of pairings}
\begin{equation*}
\begin{tikzcd}
\TwArr(\C) \ar[d, "\la"] \ar[r] & 
{\BMod{M}{M}}\big( \TwArr(\C)  \big) \ar[d] 
\\
\C \times \C^\op \ar[r] & 
{\BMod{A}{A}}(\Cc) \times \C^\op.
\end{tikzcd}
\end{equation*}
A certain compatibility with induced duality functors ($f$ is a ``right representable morphism of right representable pairings'') leads to a commuting square  
\begin{equation*}
\begin{tikzcd}
\C \ar[d] & \C \ar[l,"\Ddual'_\la"] \ar[d, "="] \\
{\BMod{A}{A}}(\C) & \C \ar[l,"\Ddual'_{\la_M}"]
\end{tikzcd}
\end{equation*}
in which the left vertical arrow is given by pullback along the augmentation. Since $\Ddual'_\la$ is the identity functor, this shows that $\Ddual'_{\la_M}$ is given by pullback of bimodules along the augmentation. Hence its left adjoint $\Ddual_{\la_M}$ is given by 
\[
\Ddual_{\la_M}(X) = \one \otimes_A X \otimes_A \one.
\]
In particular,  $\DK(A) = \one \otimes_A \one$.

\subsection{\texorpdfstring{$\LM$}{LM}-variant: functorial characterization of Koszul duals}
Recall that the $\infty$-operad $\LM^\otimes$ has two colors $\a$ and $\m$, that (by definition) an $\LM$-monoidal category 
\[
\C^\otimes \to \LM^\otimes
\]
witnesses the category $\C_\m$ as left-tensored over the monoidal category $\C_\a^\otimes$, and that an $\LM$-algebra $X$, that is a morphism of $\infty$-operads
\[
X : \LM^\otimes \to \C^\otimes 
\]
over $\LM^\otimes$, models a pair $(X_\a, X_\m)$ consisting of an algebra object $X_\a$ in $\C_\a$ and an $X_\a$-module $X_\m$ in $\C_\m$. 
Given $X \in \LMod(\C)$ we say that $X$ \emph{endows $X_\m$ with the structure of a left $X_\a$-module}.

Suppose given an \emph{$\LM$-monoidal pairing}
\begin{align}
\tag{*} \label{ast:LM_monoidal_pairing}
\la^\otimes \colon \M^\otimes \to \C^\otimes 
\times_{\LM^\otimes} \D^\otimes 
\end{align}
where the unit object of $\D_\a$ is initial. 
We assume the underlying pairings $\la_\a$, $\la_\m$ give rise to duality functors
\[
\Ddual_{\la_\a}: \C_\a^\op \to \D_\a, 
\quad
\Ddual_{\la_\m}: \C_\m^\op \to \D_\m.
\]
The pairing \eqref{ast:LM_monoidal_pairing} may be upgraded to a pairing of left-module categories
\[
\LMod(\la) \colon \LMod(\M) \to \LMod(\C) \times \LMod(\D)
\]
which subsequently gives rise to a functor 
\[
\hat \Ddual_{\LMod(\la)} \colon \LMod(\C)^\op \to \PSh \LMod \D
\]
to the category of presheaves on $\LMod \D$. The functor $\hat \Ddual_{\LMod(\la)}$ has the expected compatibility with the functor 
\[
\hat \Ddual_{\Alg(\la_\a)}: \Alg(\C_\a)^\op \to \PSh \Alg(\D_\a^\op)
\]
associated to the pairing
\[
\la_\a: \M_\a^\otimes \to \C_\a^\otimes 
\times_{\Ass} \D_\a^\otimes
\]
so that if $X = (X_\a, X_\m)$ is a left-module object of $\C^\otimes$ and $\hat \Ddual_{\LMod(\la)}(X) = (Y_\a, Y_\m)$ is its candidate Koszul dual, then 
\[
\hat \Ddual_{\Alg(\la_\a)}(X_\a) = Y_\a.
\]
Mirroring \cite[5.2.2.27]{HA}, our theorem \ref{LM52227} states that under certain assumptions, the functor $\hat \Ddual_{\LMod(\la)}$ factors through a functor
\[
\Ddual_{\LMod(\la)} \colon \LMod(\C)^\op \to  \LMod \D.
\]

We may associate to our $\LM$-monoidal category $\Cc^\otimes$ an $\LM$-monoidal category $\TwArr(\Cc)^\otimes$, the \emph{$\LM$-monoidal twisted arrow category}, witnessing $\TwArr(\Cc_\mM)$ as left-tensored under $\TwArr(\Cc_\aA)^\otimes$, as well as an $\LM$-monoidal pairing of \oo-categories
\[
\la^\otimes:
\TwArr(\Cc)^\otimes \to \Cc^\otimes \times_{\LM^\otimes} (\Cc^\op)^\otimes
\]
(\cref{lmtw}). By \cite[Remark 5.2.2.26]{HA} there's an associated pairing
\[
\tag{*}
\LMod(\la):
\LMod(\TwArr \Cc) \to \LMod(\Cc) \times \LMod(\Cc^\op).
\]

The pairing $\LMod(\la)$ induces a functor 
\[
\hat \Ddual_{\LMod(\la)}: \LMod(\Cc)^\op
\to
\PSh \LMod(\Cc^\op)
\]
to the category of prehseaves, and the \emph{Koszul dual} $\DK(X)$ of $X$ may be defined as \emph{the} object of $\LMod(\Cc^\op)$ representing $\hat \Ddual_{\LMod(\la)}(X)$, provided such an object exists. 

In this case, there's an associated object 
$
M^\mathrm{univ} \in \LMod \TwArr(\Cc)
$
lying over the pair $\big(X, \DK(X)\big)$. Unwinding definitions, we find that $M^\mathrm{univ}$ defines a morphism
\[
X_\aA \to \DK(X)_\aA
\]
which intertwines the algebra structure of $X_\aA$ with the coalgebra structure of $\DK(X)_\aA$, as well as a morphism
\[
X_\mM \to \DK(X)_\mM
\]
which intertwines the module structure of $X_\mM$ with the comodule structure of $\DK(X)_\mM$. This motivates the following definition.

\begin{mydef}
\label{260727A}
Let $\Cc^\otimes$ be an $\LM$-monoidal \oo-category, and let $\LMod(\la)$ be the associated pairing between categories of $\LM$-algebras as in (*). Consider objects $X \in \LMod(\Cc)$ and $Y \in \LMod(\Cc^\op)$. A \emph{twisted homomorphism from $X$ to $Y$} is an object of $\LMod \TwArr(\Cc)$ lying over $(X,Y)$.  
\end{mydef}

In terms of \cref{260727A}, $\DK(X)$ may be characterized as the target of the universal twisted homomorphism from the $\LM$-algebra $X$ in $\Cc^\otimes$ to an $\LM$-coalgebra in $\Cc^\otimes$. 

\subsection{\texorpdfstring{$\LM$}{LM}-variant: construction of Koszul duality functors}
In order to prove theorem \ref{LM52227}, we adapt the techniques of \cref{sec:review_duality_algebras} to the $\LM$-monoidal setting.
Let
\[
\M^\otimes \to \LM^\otimes
\]
be an $\LM$-monoidal category, and let $M$ be an algebra object of the monoidal category $\M_\a$.
In \cref{def:BLMod_A} we construct (under certain assumptions) an $\LM$-monoidal category $\BLMod_{M}(\M)^\otimes$  witnessing a left action of the category of bimodules 
\[
\BLMod_{M}(\M)_\a^\otimes = 
{_{M} \Bmod_{M}}(\M_\a)^\otimes
\]
on the category of left modules
\[
\BLMod_{M}(\M)_\m = \LMod_{M}(\M_\m).
\]
We show that an $\LM$-algebra 
\[
Z = (Z_\a, Z_\m) \in \LMod \BLMod_{M}(\M)
\]
is nothing more than an $M$-algebra $Z_\a$ and a $Z_\a$-module $Z_\m$.\footnote{
So, while the construction of the $\LM$-monoidal category $\BLMod_{M}(\M)^\otimes$ is somewhat intricate, its category of $\LM$-algebras is predetermined. 
}
We refer to such an object as an \emph{$\LM$-algebra over $M$}. 
The characters above each play a role analogous to a character in the previous play as in the following table.

\bigskip
\begin{center}
\begin{tabular}{ c| c  }
$\Ass$ & $\LM^\otimes$  \\ 
\hline
${\BMod{M}{M}}(\M)$ & 
${\BMod{M}{M}}(\M_\a)$,  ${\LMod_M}(\M_\m)$ 
\\ %
${\BMod{M}{M}}(\M)^\otimes$ &
$\BLMod_M(\M)^\otimes$
\\ %
${\Alg\left({\BMod{M}{M}}(\M)\right)}$ &
$\LMod \left( \BLMod_M(\M) \right)$
\end{tabular}
\end{center} 
\bigskip

Returning to the $\LM$-monoidal pairing above, fix a left-module object $X = (X_\a, X_\m)$ of the $\LM$-monoidal category $\C^\otimes$. 
To construct the left-module object $\Ddual_{\LMod(\la)}(X)$ of $\LMod \D$ (and hence prove theorem \ref{LM52227}), we first lift $X_\a$ to an algebra object $M$ of $\M_\a$ lying above the pair of algebras $(X_\a, \one_{\D_\a})$ just as in the monoidal case. 
We then construct an $\LM$-monoidal pairing 
\[
\la^\otimes_{M}: \BLMod_{M}(\M)^\otimes
\to 
\BLMod_{X_\a}(\C)^\otimes
\times_{\LM^\otimes} \D^\otimes.
\]
The $\LM$-monoidal pairing $\la^\otimes_{M}$
plays the role of the monoidal pairing of bimodule categories in the monoidal case. 
In fact, it includes a monoidal pairing
\[
{_{M}\Bmod_{M}}(\M_\a)^\otimes
\xto{\la_{M, \a}^\otimes}
{_{X_\a}\Bmod_{X_\a}}(\C_\a)^\otimes
\times \D_\a^\otimes,
\]
as well as a pairing of $\infty$-categories of left modules
\[
\LMod_{M}(\M_\m) 
\xto{\la_{M, \m}}
\LMod_{X_\a}(\C_\m) \times \D_\m,
\]
and enables us to construct a pairing of $\infty$-categories
\[
\la_{M}: \LMod \BLMod_{M}(\M) \to
\LMod \BLMod_{X_\a}(\C)
\times
\LMod \D
\]
of left modules \textit{over}
\[
M \mapsto (X_\a, \one_{\D_\a}).
\]

Using free resolutions as above, we construct an $M$-$M$-bimodule $Z_\a$ lying over a pair
\[
(X_\a,Y_\a) \in 
{_{X_\a}\Bmod_{X_\a}}(\C_\a) \times \D_\a
\]
witnessing the fact that $Y_\a \in \D_\a$ represents the presheaf $\hat \Ddual_{\la_{M, \a}}(X_\a) \in \PSh(\D_\a)$. 
Similarly, if $F = X_\a \otimes E$ is the free left $X_\a$-module generated by $E \in \C_\m$, then its dual $\hat \Ddual_{\la_M, \m}(F) \in \PSh \D_\m$ with respect to the pairing $\la_{M, \m}$ is represented by the object
\[
\Ddual_{\la_M, \m}(F) = \Ddual_{\la_\m}(E).
\]
Applying this observation to a free resolution, one obtains a left $M$-module $Z_\m$ lying over a pair
\[
(X_\m,Y_\m) \in 
{\LMod_{X_\a}}(\C_\m) \times \D_\m
\]
witnessing the fact that $Y_\m \in \D_\m$ represents the presheaf $\hat \Ddual_{\la_{M, \m}}(X_\m) \in \PSh(\D_\m)$. 
Since $X_\a$ is the unit object of ${_{X_\a}\Bmod_{X_\a}}(\C_\a)$ and since $\hat \Ddual_{\la_{M}}$ is lax-$\LM$-monoidal, $Y = (Y_\a, Y_\m)$ inherits a left module structure, as witnessed by a left module structure on $Z = (Z_\a, Z_\m)$. 
Dispensing with the spent $M$-algebra structure of $Z_\a$, we obtain a left module $Z$ in $\M^\otimes$ witnessing the fact that $Y$ represents $\Ddual_{\LMod(\la)}(X)$.

This completes our overview of the proof of \cref{LM52227}. Comparison with the bar construction in the case of the twisted arrow category is achieved similarly to \S \ref{comprev} above, and we refer the reader to the proof of \cref{BarCobarAdjunction} for details. 

\section{Review of $\infty$-operads}
\label{infop}



This section too is preliminary and not logically necessary for what follows. On the recommendation of the referee, we recall salient features of the foundations of $\infty$-operads, since these play a central role in all that follows. This may help to orient some readers, but will not replace the need to refer frequently to Chapters 2 and 3 of \textit{Higher Algebra}. Moreover, we've attempted to design our presentation so as to complement, rather than replace, the introductory material in Higher Algebra. Thus, the reader should be warned that we do not reproduce in detail essential points that are covered in sufficient detail already in loc. cit.

\spar{}
\label{21205A}
A colored operad $\Oo$ in sets consists of a family of \emph{colors} (or \emph{objects}), for every finite family $\{X_i\}_{i\in I}$ of objects and every object $Y$ a set $\Mul(\{X_i\}_{i\in I}, Y)$ of \emph{morphisms}, plus special elements
\[
id_X\in \Mul(\{X\}, X)
\]
plus composition laws
\[
\tag{*}
\prod_{j \in J}
\Mul(\{X_i\}_{i \in I_j}, Y_j) \times
\Mul(\{Y_j\}_{j \in J}, Z)
\to
\Mul(\{X_i\}_{i\in I}, Z).
\]
In addition to the evident families of objects, (*) also depends on the choice of a map $I \to J$ of index sets, and $I_j$ denotes the fiber of $I$ above $j\in J$. These are required to be unital and associative in an unsurprising way --- see~\cite[Definition 2.1.1.1]{HA} for detail. 

If $\Oo$, $\Oo'$ are colored operads, a \emph{map $F: \Oo \to \Oo'$ of colored operads} consists of a map of colors, plus maps
\[
\Mul(\{X_i\}_{i\in I}, Y) \to 
\Mul(\{FX_i\}_{i\in I}, FY)
\]
which are compatible with unit objects and compositions.

\spar{}
\label{21205B}
For $n \in \mathbb{N}$ a natural number, we set
$
\langle n \rangle^\circ = \{1, \dots, n\}.
$
If $\Oo$ has only one color $\one$ then there's a unique family of objects $\one^I$ associated to any index set $I$ and we define
\[
\Oo_n = \Mul(\one^{\qq n \pp^\circ}, \one).
\]
The composition laws
\[
\prod_{j \in \qq n \pp^\circ}
\Mul (\{\one\}, \one) \times
\Mul(\one^{\qq n \pp^\circ}, \one)
\to
\Mul(\one^{\qq n \pp^\circ}, \one)
\]
associated to the bijections $\qq n \pp^\circ \to \qq n \pp^\circ$, applied to the identity elements $id \in \Mul(\{\one\}, \one)$, provide an action of the symmetric group $\Si_n$ on $\Oo_n$.

\spar{}
\label{21205C}
We let $\Finst$ denote the category of finite pointed sets. We often restrict attention to the equivalent full subcategory of pointed sets of the form
\[
\langle n \rangle = \{\ast, 1, \dots, n\}
\quad (n \ge 0)
\]
in which $\ast$ denotes the distinguished element. 
Suppose $\Oo$ is a colored operad in sets. Following \cite[Construction 2.1.1.7]{HA}, we may associate to $\Oo$ a functor between ordinary  categories
\[
p:\Oo^\otimes \to \Finst
\]
as follows. An object of $\Oo^\otimes$ is a finite sequence of colors (we may equivalently restrict attention to sequences indexed by the sets $\qq n \pp^\circ$). A morphism 
\[
(X_1, \dots, X_m) \xto{(\al, \phi)}
(Y_1, \dots, Y_n)
\]
consists of a morphism $\al: \qq m \pp \to \qq n \pp$ in $\Finst$ plus a 
\[
\phi_j \in \Mul(\{X_i\}_{i \in \al\inv(j)}, Y_j)
\]
for each $j \in \qq n \pp^\circ$. The composition law and the functor $p$ are evident. The data
\[
p:\Oo^\otimes \to \Finst
\]
then forms an \emph{$\infty$-operad} in the sense of \cite[Definition 2.1.1.10]{HA}. Here and below, we identify an ordinary (1-)category with the associated $\infty$-category. 

The general definition of an \oo-operad may not be essential on first reading. It is, however, important, at least psychologically, to see how the above construction may be reversed in order to retrieve the colored operad $\Oo$ from the associated \oo-operad $\Oo^\otimes \to \Finst$. See~\cite[Construction 2.1.1.7]{HA}.

\spar{}
If $p: \Oo^\otimes \to \Finst$ is an \oo-operad, then the fiber $\Oo^\otimes_{\langle 1 \rangle}$ of $p$ over the object $\langle 1 \rangle \in \Finst$ is called the \emph{underlying \oo-category} and is frequently denoted by $\Oo$. 

In practice, the superscript $\otimes$ often becomes cumbersome and is frequently dropped (just as one identifies a group with its \textit{underlying set}). This generally does \textit{not} mean return to a primordial \textit{colored operad} as in paragraph \ref{21205A}. Indeed, this last notion will not intervene beyond the present review.

\spar{}
\label{21205Ca}
Nevertheless, let $\Oo$ be a colored operad in sets as in paragraph \ref{21205A}. If $\C$ is a symmetric monoidal category, an \emph{$\Oo$-algebra in $\Cc$} consists of the following: for each color $X \in \Oo$, a set $V_X$, plus maps of sets
\[
\Mul(\{X_i\}_{i\in I}, Y) \times \prod_{i\in I} V_{X_i}
\to V_Y
\]
which are required to be compatible with identity elements and composition laws in an unsurprising way. 

The notion of ``$\Oo$-algebra'' may be reworded in terms of maps of colored operads. Indeed, the symmetric monoidal category $\Cc$ itself may be upgraded to a colored operad whose colors are the objects of $\Cc$ by defining
\[
\Mul_\Cc\big(\{X_i\}_{i\in I}, Y\big):=
\Hom_\Cc \big(
\bigotimes_i X_i, Y
\big).
\]
An $\Oo$-algebra in $\Cc$ is then the same as a map of operads $\Oo \to \Cc$. 

On the other hand, when $\Oo$ has only one color, the notion of ``$\Oo$-algebra'' may also be phrased in terms of the \emph{endomorphism operad} $\Ee nd(V)$ defined by $\Ee nd_n(V) = \Hom(V^{\otimes n}, V)$: in this case, an $\Oo$-algebra corresponds to a morphism of operads $\Oo \to \Ee nd(V)$. Throughout the body of the article, we will replace the above notion with a notion of algebra over an $\infty$-operad. Nevertheless, this temporary definition is useful in gaining a first acquaintance with some of the operads we will encounter. 

\spar{}
\textbf{Example: the commutative operad.}
\label{21205E}
We set $\mathrm{Comm}_n = \{*\}$ for all $n \ge 0$. Thus, a $\mathrm{Comm}$-algebra in sets consists of a set $V$ equipped with one map
\[
\mu_n: V^{\otimes n} \to V
\]
for each $n$. The map $\mu_1: V \to V$ is forced to be the identity. Moreover, $\mu_n$ for $n >2$ is determined by $\mu_i$ for $i <n$. For instance, if $n = i+j$ is any partition of $n$ into two parts, then $\mu_n$ is equal to the composite
\[
V^{\otimes n} \simeq
V^{\otimes i} \otimes V^{\otimes j} 
\xto{\mu_i \otimes \mu_j} 
V \otimes V
\xto{\mu_2}
V.
\]
In this way, the $\mathrm{Comm}$-algebra structure is uniquely determined by $\mu_2$ and $\mu_0$. In turn, these are forced to determine the structure of an associative commutative unital monoid. For the commutativity, note that the map
\[
\mathrm{Comm}_2 \to \Ee nd_2(V)
\]
is forced to be compatible with the action of $\Si_2$, and that the universal family of binary operations
\[
\Ee nd_2(V) \otimes V^{\otimes 2}
\to V
\]
coequalizes the two maps
\[
\Ee nd_2(V) \otimes V^{\otimes 2}
\rightrightarrows
\Ee nd_2(V) \otimes V^{\otimes 2}
\]
induced by the transposition. A similar discussion applies with the category of sets replaced by the symmetric monoidal category of modules over a commutative ring $R$ and yields $R$-algebras in the usual sense.

These examples provide some intuition for the operad $\mathrm{Comm}$; in turn, this helps to motivate the definition of the $\infty$-operad $\mathrm{Comm}^\otimes$ via the construction of \S\ref{21205C}. Since $\mathrm{Comm}$ has only one color, the objects of $\mathrm{Comm}^\otimes$ may be identified with the objects of $\Finst$. Moreover, a morphism 
\[
(\al, \phi):(\one, \dots, \one) \to 
(\one, \dots, \one)
\]
is uniquely determined by $\al$ with no restrictions imposed. So $q: \mathrm{Comm}^\otimes \to \Finst$ is an \textit{isomorphism} of categories. 

\spar{}
\label{21205G}
A morphism
$
g:\langle m \rangle \to \langle n \rangle
$
is said to be \emph{inert} if for each $i \in \langle n \rangle^\circ$, $g^{-1}(i)$ has exactly one element.
Let
\[
p:\Oo^\otimes \to \Finst
\]
be an $\infty$-operad. A morphism $f$ in $\Oo^\otimes$ is said to be \emph{inert} if $p(f)$ is inert and $f$ is $p$-coCartesian. A morphism
\[
\Oo^\otimes \to \Oo’^\otimes
\]
of simplicial sets over $\Finst$ is an $\infty$-\textit{operad map} if it carries inert morphisms to inert morphisms. These span a full $\infty$-subcategory $\Alg_\Oo(\Oo’)$ of the $\infty$-category of all functors. An object of $\Alg_\Oo(\Oo’)$ is called an \emph{$\Oo$-algebra in $\Oo'$} ~\cite[def. 2.1.2.7]{HA}.

\spar{}
Let $\Oo^\otimes \to \Finst$ be an $\infty$-operad. An \emph{$\Oo$-monoidal category} consists of a coCartesian fibration of $\infty$-operads
\[
\Cc^\otimes \to \Oo^\otimes
\]
in the sense of \cite[Definition 2.1.2.13]{HA}.

\spar{}
\textbf{Example.}
A morphism
$
g:\langle m \rangle \to \langle n \rangle
$
in $\Finst$ is said to be \emph{active} if $g^{-1}(\ast) = \ast$. A \emph{symmetric monoidal category} \cite[Definition 2.0.0.1]{HA} is a coCartesian fibration
\[
q : \Cc^\otimes \to \Finst
\]
with \emph{underlying category} $\Cc := \Cc^\otimes_{\qq 1 \pp}$ such that for each $n$, the functor
\[
\Cc^\otimes_{\qq n \pp} \to \Cc^n
\]
induced by the inert maps
\[
\rho^i : \qq n \pp \to \qq 1 \pp
\quad
(i = 1, \dots, n) 
\]
is an equivalence. The active maps
\[
\qq 2 \pp \to \qq 1 \pp
\from \qq 0 \pp
\]
determine functors
\[
\Cc^2 \to \Cc \from \Cc^0 \simeq \Delta^0
\]
which define the \emph{unit} and the \emph{tensor product}. 

If $\Cc^\otimes \to \Finst$ is a symmetric monoidal category, then a \emph{commutative algebra} in $\C$ is an \oo-operad map $\mathrm{Comm}^\otimes \to \C^\otimes$. The \emph{\oo-category of commutative algebras in $\C$} is defined by 
\[
\CAlg(\Cc) := \Alg_{\mathrm{Comm}^\otimes}(\Cc).
\]

\spar{}
\textbf{Remark.}
If one regards the sets $\Oo_n$ of an operad in sets as discrete topological spaces, then operads in sets become part of the larger and more flexible category of topological operads. The latter is equipped with a natural model structure which allows for the possibility of replacing $\Oo$ by a ``cofibrant replacement'' $\Oo^c$. Suppose the symmetric monoidal model category $\Cc$ is enriched over nice topological spaces. Then a suitable notion of \emph{$\Oo^c$-algebras in $\Cc$} provides a well behaved model category of ``strict models for $\Oo$-algebras in $\Cc$''. Moreover, instead of applying the small object argument, in concrete examples $\Oo^c$ may often be realized by a concrete and charming model.

However, if one is willing to make do with highly structured algebras in highly structured categories as we do here, then one can dispense with this part of the theory.

\spar{}
\label{Review_of_Assoc}
The foundations of \oo-operads as presented in \textit{Higher Algebra} make symmetric monoidal categories and commutative algebras particularly natural. However, below, (merely) monoidal categories, associative algebras, and modules of various sorts, will play a more central role. We now rapidly recall the relevant \oo-operads. 

The associative operad $\mathrm{Assoc}$ is defined in \cite[Definition 4.1.1.1]{HA}. We recall that there is only one color $\a$ and for every finite set $I$, $\Mul(\{\a\}_{i \in I}, \a)$ is the set of linear orderings on $I$. Thus, the objects of the associated $\infty$-operad $\Assoc$ are the same as those of $\Finst$ and a morphism
\[
\langle m \rangle \to \langle n \rangle
\]
in $\Assoc$ consists of a pair $(\al, \{\succ_i\}_{1 \le i \le n})$ where $\al:\langle m \rangle \to \langle n \rangle$ is a map of pointed finite sets and $\succ_i$ is a linear ordering on $\al^{-1}(i)$. This is spelled out in complete detail in \cite[Remark 4.1.1.4]{HA}. If
\[
q: \C^\otimes \to \Assoc
\]
is a fibration of $\infty$-operads, then we define an \emph{associative algebra in $\C^\otimes$} to be an $\infty$-operad section of $q$, and we define the $\infty$-category of such by
\[
\Alg(\C) := \Alg_{\mathrm{Assoc}}(\C).
\]
This is Definition 4.1.1.6 of \cite{HA}.

The \emph{left-module $\infty$-operad} 
$\LM^\otimes \to \Fin_*$ (which is just the $\infty$-operad associated to the \emph{left-module colored operad} encountered elsewhere in the literature)
is constructed in ~\cite[def. 4.2.1.7]{HA}. 
We denote the two objects of the underlying category $\LM^\otimes_{\angled{1}}$ by $\a,\m$. 
An $\LM$-monoidal category $\C^\otimes$ is a coCartesian fibration
\[
\C^\otimes \to \LM^\otimes
\]
of $\infty$-operads (see~\cite[def. 2.1.2.13]{HA}).
Following~\cite[sec. 4.2.1]{HA} we denote: 
\begin{align*}
\C^\otimes_\a := \C^\otimes \times_{\LM^\otimes} \Ass , 
\quad \C_\a := \C^\otimes \times_{\LM^\otimes} \set{\a}, 
\quad \C_\m := \C^\otimes \times_{\LM^\otimes} \set{\m}. 
\end{align*}
Note that 
$\C^\otimes_\a \to \Ass$ defines a monoidal structure on the $\infty$-category $\C_\a$ (see~\cite[rem. 4.2.1.10]{HA}).
Following~\cite[def. 4.2.1.19]{HA} we say that $\C^\otimes \to \LM^\otimes$ \emph{exhibits $\C_\m$ as left tensored over $\C_\a$}.

We denote the $\infty$-category of $\LM$-algebras in $\C$ by $\LMod(\C) = \Alg_{\LM}(\C)$.
Explicitly, an object $X\in \LMod(\C)$ is a map of $\infty$-operads 
\[
\LM^\otimes \to \C^\otimes
\]
over $\LM^\otimes$.
Precomposition with the inclusion $\Ass \to \LM^\otimes$ defines a functor
\[
\LMod(\C) \to \Alg(\C_\a).
\]
Given an algebra $A\in \Alg(\C_\a)$ we define the $\infty$-category $\LMod_A(\C_\m)$
of \emph{$A$-modules in $\C_\m$} as the fiber of
\[
\LMod(\C) \to \Alg(\C_\a)
\]
over $A$.
We call an object of $\LMod_A(\C_\m)$ a left \emph{$A$-module} in $\C_\m$. Given an $\LM$-algebra $X\in \LMod(\C)$, we denote by $X_\a \in \Alg(\C_\a)$ its image under
\[
\LMod(\C) \to \Alg(\C_\a)
\]
(so $X$ is a left $X_\a$-module) and we denote by $X_\m \in \C_\m$ the object corresponding to 
\[
\set{\m} \to \LM^\otimes \xto{X} \C.
\]
As a matter of notation, we allow ourselves to drop the subscript $\mM$ in `$\LMod_A(\C_\m)$'.

We will make particularly detailed use of the \emph{bimodule $\infty$-operad} $\BM$. The reader who is familiar with the \emph{usual bimodule colored operad} will be relieved to hear that this is indeed just the \oo-operad associated to the latter.

\begin{notation}
Let $\BM^\otimes$ be the $\infty$-operad of~\cite[def. 4.3.1.6]{HA}.
We can describe $\BM^\otimes$ as the (nerve of) the following ordinary (1-)category (see~\cite[ex. 4.4.1.7]{HA}):
\begin{enumerate}
\item An object of $\BM^\otimes$ is a tuple $(\angled{n}, c_-,c_+)$ where $\angled{n} \in \Fin_*$ is a finite pointed set and 
\[
c_-,c_+ \colon \angled{n}^\circ \rightrightarrows [1]
\]
are maps of sets satisfying 
\[
c_-(i) \leq c_+(i) \quad \text{for }  i \in \angled{n}^\circ.
\]
\item Let $(\angled{n}, c_-,c_+)$ and $(\angled{n'}, c'_-,c'_+)$ be objects of $\BM^\otimes$.
A morphism from $(\angled{n}, c_-,c_+)$ to $(\angled{n'}, c'_-,c'_+)$ consists of a morphism $\al \colon \angled{n} \to \angled{n'}$ in $\Ass$ satisfying the following conditions:
for each $j\in \angled{n'}^\circ$ with fiber 
\[
\al^{-1}(j) = \set{ i_0 \prec \cdots \prec i_m}:
\]
\begin{align*}
&c_-(i_0) = c'_-(j), \\ 
&c_+(i_m) = c'_+(j), \text{and} \\
&c_-(i_0) \leq c_+(i_0) = c_-(i_1) \leq c_+(i_1) = c_-(i_2) \leq \cdots \leq c_+(i_{m-1}) = c_-(i_m) \leq c_+(i_m).
\end{align*}
\end{enumerate}
The $\infty$-operad $\BM^\otimes$ has three colors, given by the objects $(\angled{1},c_-,c_+) \in \BM^\otimes_{<1>}$ with 
\[
c_0=c_1=0, \quad 
c_-=0, c_+=1 , \quad
c_-=c_+=1. 
\]
These are denoted variously by $\a_- = \a_0, \m$ and $\a_+ = \a_1$, respectively.
\end{notation}

\begin{mydef}
A $\BM$-monoidal $\infty$-category is by definition a coCartesian fibration of \oo-operads $\C^\otimes \to \BM^\otimes$. 
We let $\C_- = \C_0, \C_\m, \C_+= \C_1$ denote the fibers over $\a_-, \m, \a_+$ respectively, so that $\C^\otimes$ exhibits $\C_\m$ as bitensored over $\C_-$ and $\C_+$ (\cite[def. 4.3.1.17]{HA}).
\end{mydef}

\spar{}
Throughout the text, we make constant use of the techniques of \emph{operadic (co)limits} and \emph{operadic Kan extensions} developed in \cite[Sections 3.1.1, 3.1.2]{HA}. Let us note here that in practice (or at least in much of what follows), in view of \cite[3.1.1.15, 3.1.1.16]{HA}, operadic colimits may frequently be computed like so. Let
\[
q: \C^\otimes \to \Oo^\otimes
\]
be a coCartesian fibration of \oo-operads. Suppose given a diagram
\[
Y: K \to \C^\otimes_\mathrm{act} \subset \C^\otimes
\]
whose arrows correspond to active morphisms (see e.g. \cite[2.2.4.3]{HA} for the \emph{active subcategory}) and a commuting square
\[
\begin{tikzcd}
K \ar[r,"Y"] \ar[d] & \C^\otimes \ar[d,"q"]
\\
K^\rhd \ar[r, "X"] & \Oo^\otimes.
\end{tikzcd}
\]
We denote the cone point of $K^\rhd$ by $\infty$ and for $i \in K$ we write $X_i := X(i)$. For each edge $\ep:i \to \infty$ in $K^\rhd$ with source $i \in K$, we may lift the morphism
\[
X(\ep): X_i \to X_\infty
\]
in $\Oo^\otimes$ to a $q$-coCartesian edge $\kappa_i$ in $\C^\otimes$ with source $Y_i:=Y(i)$. Doing so amounts to applying various operations (``tensor products'') to the components of $Y_i$ as indicated by the morphism $X(\ep)$. As $i$ varies, the targets $Y'_i$ of the morphisms $\kappa_i$ in $\C^\otimes$ assemble to a coherently commuting diagram
\[
Y':K \to \Cc^\otimes_{X_\infty}
\]
in the fiber of $\C^\otimes$ above the object $X_\infty \in \Oo^\otimes$. If the ordinary colimit of $Y'$ is in a suitable sense compatible with the operations of $\Oo$ (condition (2) of Proposition 3.1.1.16 of \cite{HA}), then the latter is equal to the operadic $q$-colimit of $Y$.

\spar{}
This completes our review of the foundations of \oo-operads. Beginning immediately, however, we will have to rely on the more general theory of \textit{generalized \oo-operads}. We refer the reader to section 2.3.2 of Higher Algebra \cite{HA}. See, in particular, Definition 2.3.2.10 of loc. cit. for the notion of a \emph{family of \oo-operads}.

\section{Restricted double \texorpdfstring{$\infty$}{infinity}-categories of bimodules in a monoidal \texorpdfstring{$\infty$}{infinity}-category}
\label{sec:Morita_double_cat}


In~\cite[lem. 4.4.3.9]{HA} Lurie constructs an $\infty$-categorical version of the double Morita category of algebras and bimodules in a monoidal $\infty$-category $\C$, under the assumption that tensor products in $\C$ are compatible with geometric realizations. In this section we weaken the compatibility assumption:
we only assume compatibility with a restricted collection of geometric realizations, prescribed by a class of algebras $\A$. This serves as an opportunity to recall the material of \cite[sec. 4.4.3]{HA} while making occasional adjustments along the way. 
The results of this section will be used in \cref{sec:BLMod} to package the relative tensor product of an $A$-$A$-bimodule and a left $A$-module into an $\LM$-monoidal $\infty$-category.

\subsection{Multimodules}
\label{FF44}
The combinatorics of the associativity of relative tensor products up to coherent higher homotopies are encoded by the $\Delta^\op$-family of $\infty$-operads
\[
\tag{*}
\Tens^\otimes \to \Delta^\op \times \Fin_*
\]
of \cite[Definition 4.4.1.1]{HA}. We recall the precise definition:

\begin{mydef}
\label{def:Tens}
Let $\Ass$ be the $\infty$-operad of~\cite[def. 4.1.1.3]{HA} (see also \S\ref{Review_of_Assoc}). The 1-category $\Tens^\otimes$ is defined as follows:
\begin{enumerate}
\item An object of $\Tens^\otimes$ is a tuple $(\angled{n},[k],c_-,c_+)$ where $\angled{n} \in \Fin_*, [k]\in \Delta^\op$ and the maps
\(
c_-, c_+ \colon \angled{n}^\circ \to [k]
\)
satisfy the following condition:
\[
\mbox{for all }i \in \angled{n}^\circ: \quad c_-(i) \leq c_+(i) \leq c_-(i)+1.
\]
In other words, either $c_+(i) =c_-(i)$
or $c_+(i)=c_-(i)+1$.
\item Let $(\angled{n},[k],c_-,c_+)$ and $(\angled{n'},[k'],c'_-,c'_+)$ be two objects of $\Tens^\otimes$.
A morphism from $(\angled{n},[k],c_-,c_+)$ to $(\angled{n'},[k'],c'_-,c'_+)$ is a tuple $(\al,\la)$ where 
\[
\al \colon \angled{n} \to \angled{n'} \in \Ass,
\quad \la \colon [k'] \to [k] \in \Delta
\]
satisfy the following condition.
For every $j\in \angled{n'}^\circ$ with totally ordered fiber $\al^{-1}(j)= \set{i_0 \prec i_1 \prec \cdots \prec i_m}$ the following equalities hold:
\begin{align*}
& c_-(i_0) = \la(c'_-(j)), \quad c_+(i_m) = \la(c'_+(j)), \\
& c_-(i_0) \leq c_+(i_0) = c_-(i_1) \leq c_+(i_1) = \cdots = c_-(i_m) \leq c_+(i_m).
\end{align*}
\end{enumerate}
The functor (*) is given by
\[
(\angled{n},[k],c_-,c_+) \mapsto
([k], \angled{n}) \in \Delta^\op \times \Finst.
\]
\end{mydef}

Thus, for each $k \ge 1$, the fiber $\Tens^\otimes_{[k]}$ of $\Tens^\otimes$ over $[k] \in \Delta^\op$ is an \oo-operad with colors
\[
\a_0 , \a_1 , \a_2 , \cdots , \a_k
\quad \text{and} \quad
\m_{0,1} , \m_{1,2} , \cdots , \m_{k-1, k}
\quad
\in 
\Tens^\otimes_{[k],\langle 1 \rangle}
\]
\cite[Notation 4.4.1.9]{HA}. (Variants: We write $\aA_i = \mM_{i,i}$. We add a superscript `$k$' when working with variable $k$.) We have
\[
\Tens^\otimes_{[0]} = \Assoc
\quad \text{and} \quad
\Tens^\otimes_{[1]} = \BM^\otimes.
\]
Turning to general $[k] \in \Delta^\op$, if
\[
M \colon \Tens^\otimes_{[k]} \to \Cc^\otimes
\]
is a $\Tens^\otimes_{[k]}$-algebra over $\Ass$ in a monoidal $\infty$-category $\C^\otimes$, then associated to each color $\a_i$ is an algebra $A_i$ in $\C$, and associated to each color $\m_{i,i+1}$ is an $A_i$-$A_{i+1}$-bimodule in $\C$. We refer to $\Tens^\otimes$ as the \emph{Morita $\infty$-operad family}.

If $\al: S \to \Delta^\op$ is a map of simplicial sets, we denote the pullback of (*) along 
\[
\al \times id:
S\times \Fin_* 
\to
\Delta^\op \times \Fin_* 
\]
by $\Tens^\otimes_S$ or by $\Tens^\otimes_\al$ if we wish to emphasize the dependence on the map $\al$. 
In particular, if $\al:s \to s'$ is a morphism in $\Delta^\op$, we write $\Tens^\otimes_\al$ for the pullback along the associated morphism
\[
\al: \Delta^1 \to \Delta^\op.
\]

\begin{terminology}[$\A$-multimodules] \label{29.1}
There's an evident map of simplicial sets
\[
\Tens^\otimes \to \Delta^\op \times \Ass
\]
factoring \ref{FF44}(*). We refer to a $\Tens^{\otimes}_{[n]}$-algebra
\[
M \colon \Tens^\otimes_{[n]} \to \C^\otimes
\]
over $\Ass$ in a monoidal category $\C^\otimes$ as a \emph{multimodule}. We'll use the notation
\[
M= 
{_{A_0}{M^{0,1}}_{A_{1}} M^{1,2}_{A_2}}
\cdots
{_{A_{n-1}} M^{n-1,n}_{A_n}}
\]
which displays the constituents of $M$ typographically  with the algebras as subscripts and the bimodules between them placed on the line. 
An \emph{$(A_0,\ldots,A_n)$-multimodule} (or a \emph{multimodule over $(A_0, \dots, A_n)$}), is simply a multimodule whose algebras are $(A_0,\ldots,A_n)$.
If we wish to specify the number $n$ we say \emph{$n$-module} instead of \emph{multimodule}.
Thus, for instance, a 2-module looks like so:
\[
_{A_0}M_{A_1}N_{A_2}.
\]

Let $\A \subset \Alg(C)$ be a class of algebras closed under equivalences. 
By an \emph{$\A$-multimodule}, we mean a multimodule 
\[
M= 
_{A_0}{M^{0,1}}_{A_{1}} M^{1,2}_{A_2}
\cdots
_{A_{n-1}} M^{n-1,n}_{A_n}
\]
whose algebras $A_0, \dots, A_n$ belong to $\A$.
\end{terminology}

\begin{rem}
We let $\Tens_{[k]} = \Tens^\otimes_{[k]\langle 1 \rangle}$ denote the category of colors of the \oo-operad $\Tens_{[k]}^\otimes$. Denote the color $\aA_i$ also by $\mM_{i,i}$. Then under the equivalence 
\[
\Tens^\otimes_{[k]\langle n \rangle}
\simeq
\Tens^n_{[k]},  
\]
the object $\big( \angled{n}, [k], c_-,c_+ \big) \in \Tens^\otimes_{[k]\langle n \rangle}$ corresponds to the family of colors
\[
( \mM^k_{c_-(1), c_+(1)}, \mM^k_{c_-(2), c_+(2)}, 
\dots,  \mM^k_{c_-(n), c_+(n)}).
\]
\end{rem}

\begin{mydef}
\label{bbb}
Let $\succ: \Delta^1 \to \Delta^\op$ denote the map of simplicial sets associated to the morphism 
\[
\succ^\op:[1] \simeq \set{0,2} \cof [2]
\]
in $\Delta$, and let $\Tens^\otimes_{\succ}$ denote the pullback of $\Tens^\otimes$ along $\succ$. Let $\Cc^\otimes \to \Ass$ be a monoidal \oo-category. A \emph{bilinear map of bimodules in $\Cc$} is a map of generalized \oo-operads over $\Ass$
\[
F: \Tens^\otimes_{\succ} \to \Cc^\otimes.
\]
(See~\cite[def. 2.3.2.2]{HA} for the notion of a map of generalized \oo-operads.)
\end{mydef}

\begin{ex}
\label{LabelChaos}
Suppose given a bilinear map of bimodules as above. Then $F|_{\Tens^\otimes_{[2]}}$ is a $2$-module ${_{A_0} M^{0,1} _{A_1} M^{1,2} _{A_2}}$ in $\Cc$, and $F|_{\Tens^\otimes_{[1]}}$ is a $1$-module ${_{B_0} {N^{0,1}} _{B_1}}$ in $\Cc$. The bilinear map of bimodules $F$ then includes the data of algebra maps
\[
f:A_0 \to B_0
\quad \text{and} \quad
g: A_2 \to B_1,
\]
a map
\[
b: M^{0,1} \otimes M^{1,2} \to N^{0,1},
\]
as well as coherently commuting squares
\[
\tag{*}
\begin{tikzcd}
A_0 \otimes M^{0,1} \otimes M^{1,2}
\ar[r] \ar[d]
&
B_0 \otimes N^{0,1}
\ar[d]
\\
M^{0,1} \otimes M^{1,2}
\ar[r]
&
N^{0,1}
\end{tikzcd}
\]
{\footnotesize
\[
\begin{tikzcd}[column sep=1.5cm]
M^{0,1} \otimes A_1 \otimes M^{1,2}
\ar[r, "{\mathrm{mult}_{A_1, M^{1,2}}}"] 
\ar[d,  "{\mathrm{mult}_{M^{0,1}, A_1}}"]
&
M^{0,1} \otimes M^{1,2}
\ar[d]
&
M^{0,1} \otimes M^{1,2} \otimes A_{2}
\ar[r] \ar[d]
&
M^{0,1} \otimes M^{1,2}
\ar[d]
\\
M^{0,1} \otimes M^{1,2}
\ar[r]
&
N^{0,1}
&
N^{0,1} \otimes B_2
\ar[r]
&
N^{0,1}
\end{tikzcd}
\]}
in $\C$, and so we can reasonably say that \emph{$F$ witnesses $b$ as an $(f,g)$-bilinear map}.

As some amount of dexterity in working with bilinear maps is needed throughout much of the material that follows, the reader may wish to take the time to work this out. We include our own constructions of $b$ and of (*) for the reader's convenience. The morphism from $\aA_0 \in \Tens_{[2]}$ to $\aA_0 \in \Tens_{[1]}$ given by
\[\begin{tikzcd}
	1 & {\langle1\rangle^\circ} & {[2]} \\
	1 & {\langle1\rangle^\circ} & {[1]}
	\arrow["\in"{marking, allow upside down}, draw=none, from=1-1, to=1-2]
	\arrow[maps to, from=1-1, to=2-1]
	\arrow["0"', shift right, from=1-2, to=1-3]
	\arrow["0", shift left, from=1-2, to=1-3]
	\arrow[from=1-2, to=2-2]
	\arrow["\in"{marking, allow upside down}, draw=none, from=2-1, to=2-2]
	\arrow["0", shift left, from=2-2, to=2-3]
	\arrow["0"', shift right, from=2-2, to=2-3]
	\arrow["{\succ^\mathrm{op}}"', from=2-3, to=1-3]
\end{tikzcd}\]
gives rise to a homomorphism
$
f: A_0 \to B_0
$ of associative algebras in $\Cc$.
The morphism from $\aA_2 \in \Tens_{[2]}$ to $\aA_1 \in \Tens_{[1]}$ given by
\[\begin{tikzcd}
	1 & {\langle1\rangle^\circ} & {[2]} \\
	1 & {\langle1\rangle^\circ} & {[1]}
	\arrow["\in"{marking, allow upside down}, draw=none, from=1-1, to=1-2]
	\arrow[maps to, from=1-1, to=2-1]
	\arrow["2"', shift right, from=1-2, to=1-3]
	\arrow["2", shift left, from=1-2, to=1-3]
	\arrow[from=1-2, to=2-2]
	\arrow["\in"{marking, allow upside down}, draw=none, from=2-1, to=2-2]
	\arrow["1", shift left, from=2-2, to=2-3]
	\arrow["1"', shift right, from=2-2, to=2-3]
	\arrow["{\succ^\mathrm{op}}"', from=2-3, to=1-3]
\end{tikzcd}\]
gives rise to a homomorphism
$
g: A_2 \to B_1. 
$
The morphism from $(\mM_{0,1}, \mM_{1,2}) \in \Tens^\otimes_{[2]\angled{2}} \simeq \Tens^2_{[2]}$ to $\mM_{0,1} \in \Tens^\otimes_{[1]\angled{1}} = \Tens_{[1]}$ given by
\[\begin{tikzcd}
	& 2 & {1,2} \\
	& 1 & {0,1} \\
	{1,2} & {\langle 2 \rangle^\circ} & {[2]} \\
	1 & {\langle1\rangle^\circ} & {[1]} \\
	& 1 & {0,1}
	\arrow[maps to, from=1-2, to=1-3]
	\arrow[maps to, from=2-2, to=2-3]
	\arrow["\in"{marking, allow upside down}, draw=none, from=2-2, to=3-2]
	\arrow["\in"{marking, allow upside down}, draw=none, from=2-3, to=3-3]
	\arrow["\in"{marking, allow upside down}, draw=none, from=3-1, to=3-2]
	\arrow[maps to, from=3-1, to=4-1]
	\arrow[shift right, from=3-2, to=3-3]
	\arrow[shift left, from=3-2, to=3-3]
	\arrow["\alpha"', from=3-2, to=4-2]
	\arrow["\in"{marking, allow upside down}, draw=none, from=4-1, to=4-2]
	\arrow[shift left, from=4-2, to=4-3]
	\arrow[shift right, from=4-2, to=4-3]
	\arrow["{\succ^\mathrm{op}}"', from=4-3, to=3-3]
	\arrow["\in"{marking, allow upside down}, draw=none, from=5-2, to=4-2]
	\arrow[from=5-2, to=5-3]
	\arrow["\in"{marking, allow upside down}, draw=none, from=5-3, to=4-3]
\end{tikzcd}\]
gives rise to a morphism 
$(M^{0,1}, M^{1,2}) \to N^{0,1}$
in $\Cc^\otimes$ over the morphism $\al$ in $\Ass$. The choice of a coCartesian edge
$(M^{0,1}, M^{1,2}) \to M^{0,1} \otimes M^{1,2}$ over $\al$ gives rise to a morphism $b$ as above.

We turn to the construction of the coherently commuting square (*). Consider the commuting square
\[
\begin{tikzcd}
\big(
\langle 3\rangle, [2], c_\pm^{nw} 
\big)
:=
(\aA_0^2, \mM_{0,1}^2, \mM_{1,2}^2)
\ar[r, "{(\alpha^n, \succ)}"] \ar[d, "{(\al^w, id_{[2]})}"']
&
\big(
\langle 2\rangle, [1], c_\pm^{ne} 
\big)
:=
(\aA_0^1, \mM_{0,1}^1)
\ar[d,"{(\al^e, id_{[1]})}"]
\\
\big(
\langle 2\rangle, [2], c_\pm^{sw} 
\big)
:=
(\mM_{0,1}^2, \mM_{1,2}^2)
\ar[r,"{(\al^s, \succ)}"']
&
\big(
\langle 1\rangle, [1], c_\pm^{se} 
\big)
:=
(\mM_{0,1}^1)
\end{tikzcd}
\]
in $\Tens^\otimes_\succ$, in which $\al^n$ is given by
\[
\al^n(1) = 1, \quad \al^n(2) = \al^n(3) = 2,
\quad 2 \prec 3,
\]
$\al^w$ is given by
\[
\al^w(1) = \al^w(2) = 1,
\quad
\al^w(3) = 2, 
\quad
1 \prec 2,
\]
and
\[
\al^e = \al^s \quad \text{are given by} \quad 1,2 \mapsto 1,
\quad
1 \prec 2.
\]
Applying the functor $F$, we obtain a coherently commuting square
\[
\tag{$\rm{sq}^\otimes$}
\begin{tikzcd}
(A_0, M^{0,1}, M^{1,2})
\ar[r] \ar[d]
&
(B_0, N^{0,1})
\ar[d]
\\
(M^{0,1}, M^{1,2})
\ar[r]
&
N^{0,1}
\end{tikzcd}
\]
in $\Cc^\otimes$. Using the active morphisms $\langle n \rangle \to \langle 1 \rangle$ given by $n \succ n-1 \succ \cdots \succ 1$, its image in $\Ass$ may be extended to a diagram $\boxempty^\rhd \to \Ass$ ($\boxempty = \Delta^1 \times \Delta^1$). Using the natural map $\boxempty \times \Delta^1 \to \boxempty^\rhd$ which sends $\boxempty \times \{1\}$ to the cone point, we obtain a commuting square (solid arrows)
\[\begin{tikzcd}
\Delta^0 \ar[r] \ar[d]
&
(\Cc^\otimes)^\boxempty
\ar[d,"q^\boxempty"]
\\
\Delta^1
\ar[r] \ar[ur, dashed]
&
(\Ass)^\boxempty
\end{tikzcd}\]
of simplicial sets. Since $q^\boxempty$ is a coCartesian fibration, there exists a diagonal dashed arrow as shown. Taking the image of $1 \in \Delta^1$ we obtain the the commuting square (*).
\end{ex}

\begin{terminology} [Relative tensor product of $\A$-multimodules] \label{27.2}
Let 
\[
q \colon \C^\otimes \to \Ass
\]
be a monoidal $\infty$-category, and let $\A \subset \Alg(\C)$ be a class of algebras closed under equivalences.
We say that $\C$ \emph{admits operadic left Kan extensions from $\A$-multimodules} if for every edge $\al: s\to s'$ in $\Delta^\op$ and every $\A$-multimodule $F_0 \in \Alg_{\Tens_s/\Ass}(\C)$ 
there exists an operadic $q$-left Kan extension
\[
F \in \Alg_{\Tens_\al/\Ass}(\C) 
\]
of $F_0$, as in the following diagram:
\begin{equation*}
\begin{tikzcd}
\Tens^\otimes_s \ar[d] \ar[r,"F_0"] & \C^\otimes \ar[d] \\
\Tens^\otimes_\al \ar[r] \ar[ur, dashed, "F"] & \Ass. 
\end{tikzcd}
\end{equation*}
We also say that \emph{$\C$ admits relative tensor products of $\A$-multimodules}; this generalizes~\cite[def. 4.4.2.3]{HA}.
\end{terminology}  

\subsection{The Morita pre-\textit{double \texorpdfstring{$\infty$}{infinity}-category}} \label{27.1}
Let 
\[
\C^\otimes \to \Ass
\]
be a monoidal \oo-category. We will construct the Morita double $\infty$-category 
\[
\Bmod_\A(\C)^\circledast
\]
as a simplicial subset of the Morita double $\infty$-category $\Bmod(\C)^\circledast$ \cite[cor. 4.4.3.2 and def. 4.4.3.10]{HA}:

\begin{mydef}\label{def:BMod_C_star}
For variable $[n] \in \Delta^\op$, the $\infty$-categories of $n$-multimodules in $\C$ arrange themselves into a map of simplicial sets 
\begin{equation}\label{eq:BmodC_map}
{\Bmod(\C)}^\circledast \to \Delta^\op
\end{equation}
defined by the following universal property:
\begin{adjustwidth}{1cm}{1cm}
For every map of simplicial sets $K \to \Delta^\op$,
we have a bijection
\[
\Hom_{(\Set_\Delta)/\Delta^\op} (K, {\Bmod(\C)}^\circledast)
\= \Alg_{\Tens_K / \Ass}(\C)_{0}
\]
between the set of simplicial maps $K \to \Bmod(\C)^{\circledast}$ over $\Delta^\op$ and the set\footnote{We emphasize that the right hand side is considered as a set by adding the subscript 0}
of $\Tens_K^\otimes$-algebras in $\C^\otimes$ over $\Ass$, functorial in $K$.
\end{adjustwidth}
We refer to ${\Bmod(\C)}^\circledast$ as the Morita pre-\textit{double $\infty$-category} of $\C$.\footnote{This is the same as~\cite[Definition 4.4.3.10]{HA}, except that we do not assume that the tensor product in $\C$ is compatible with geometric realizations (in the sense of~\cite[def. 3.1.1.18]{HA}).}
\end{mydef}

\begin{rem}
For ease of reference, let us note that ~\cref{def:BMod_C_star} relates dotted arrows as in the following diagram:
\[\begin{tikzcd}
	& {\mathrm{BMod}(\mathcal{C})^\circledast} \\
	K & {\Delta^\mathrm{op}}
	\arrow[from=1-2, to=2-2]
	\arrow[dotted, from=2-1, to=1-2]
	\arrow[from=2-1, to=2-2]
\end{tikzcd}\]
to dotted arrows as in the following diagram:
\[\begin{tikzcd}
	{\mathrm{Tens}_K^\otimes} & {\mathrm{Tens}^\otimes} & {\mathcal{C}^\otimes} \\
	{K \times \mathrm{Assoc}^\otimes} & {\Delta^\mathrm{op} \times \mathrm{Assoc}^\otimes} & {\mathrm{Assoc}^\otimes}.
	\arrow[from=1-1, to=1-2]
	\arrow[curve={height=-18pt}, dotted, from=1-1, to=1-3]
	\arrow[from=1-1, to=2-1]
	\arrow[from=1-2, to=2-2]
	\arrow[from=1-3, to=2-3]
	\arrow[from=2-1, to=2-2]
	\arrow[from=2-2, to=2-3]
\end{tikzcd}\]
The functoriality in $K$ implies that a
diagram of simplicial sets 
\begin{equation*}
\begin{tikzcd}
L \ar[d] \ar[r] & {\Bmod(\C)}^\circledast \ar[d] \\
K \ar[r] \ar[ur] & \Delta^\op
\end{tikzcd}
\end{equation*}
commutes if and only if the corresponding diagram 
\begin{equation*}
\begin{tikzcd}
\Tens_L^\otimes \ar[d] \ar[r] & \C^\otimes \ar[d] \\
\Tens_K^\otimes \ar[r] \ar[ur] & \Assoc
\end{tikzcd}
\end{equation*}
commutes. 
\end{rem}

In the following paragraphs we assemble a few first properties of $\rm{BMod}(\C)^\circledast$.
\begin{lem} \label{link2}
Let $q \colon \C^\otimes \to \Ass$ be a monoidal $\infty$-category, and
\[
p \colon {\Bmod(\C)}^\circledast \to \Delta^\op
\]
as above. Then $p$ is an inner fibration.
\end{lem}

\begin{proof}
We verify that $p$ has the right lifting property with respect to all inner anodyne maps.
By definition of ${\Bmod(\C)}^\circledast$ checking the lifting condition against an inner anodyne map $A \to B$ is equivalent to solving the following lifting problem
\begin{equation*}
\begin{tikzcd}
\Tens^\otimes_A \ar[d] \ar[r] & \C^\otimes \ar[d] \\
\Tens^\otimes_B \ar[r] \ar[ur, dashed] & \Ass. 
\end{tikzcd}
\end{equation*}
By~\cite[thm. 4.4.3.1]{HA} (flatness of $\Tens^\otimes \to \Delta^\op$) and~\cite[prop. B.3.14]{HA} the monomorphism $\Tens^\otimes_A \to \Tens^\otimes_B$ is a trivial cofibration in the Joyal model structure (\cite[thm. 2.2.5.1]{HTT}).
It follows that the solution to the above lifting problem exists, as the right vertical morphism $\C^\otimes \to \Ass$ is a categorical fibration.
\end{proof}

\begin{mydef}
\label{666a}
Let $\C^\otimes$ be a monoidal $\infty$-category and let $\A \subset \Alg(\C)$ be a class of algebras closed under equivalences. 
We say that 
\[
{\Bmod(\C)}^\circledast \to \Delta^\op
\]
\emph{admits coCartesian lifts from $\A$-multimodules} if whenever $\al:s \to s'$ is an edge in $\Delta^\op$ and $F_0$ is an $\A$-multimodule,
there exists a $p$-coCartesian lift of $\al$ extending $F_0$.
\end{mydef}

\begin{prop} \label{2}
Let $q \colon \C^\otimes \to \Ass$ be a monoidal $\infty$-category.
Let $\al: s \to s'$ be an edge in $\Delta^\op$ and let $F_0$ be a vertex of $\Alg_{\Tens_s/\Ass}(\C)$.
The vertex $F_0$ corresponds to an element of\footnote{Here we consider $s$ as a map of simplicial sets
$s \colon \set{s}=\Delta^0 \xto{0} \Delta^1 \xto{\al} \Delta^\op$.}
\[
\Hom_{{\Set_\Delta}/\Delta^\op} (\set{s}, {\Bmod(\C)}^\circledast),
\]
and hence to a vertex of ${\Bmod(\C)}^\circledast$ lying over $s$.
Let $F$ be a vertex of $\Alg_{\Tens_\al/\Ass}(\C)$ extending $F_0$; equivalently, $F$ is an edge of ${\Bmod(\C)}^\circledast$ lying over $\al$ with source $F_0$.
If $F$, regarded as
\[
\Tens^\otimes_\al \to \C^\otimes,
\]
is an operadic $q$-left Kan extension of $F_0$, then $F$ regarded as an edge of 
$
{\Bmod(\C)}^\circledast
$
is $p$-coCartesian. 
In particular, if $\C$ admits operadic left Kan extensions from $\A$-multimodules (\cref{27.2}), then $\Bmod(\C)^\circledast$ admits coCartesian lifts from $\A$-multimodules (\cref{666a}). 
\end{prop}
This is essentially Corollary 4.4.3.2 of \cite{HA}. We nevertheless repeat the proof here for the convenience of the reader, since we found the argument in loc. cit. a bit condensed. 

\begin{proof}
It's enough to solve the lifting problem of~\cite[rem. 2.4.1.4]{HTT} (dualized):
\begin{equation}
\tag{*}
\text{For $n \ge 2$:} \quad \quad
\begin{tikzcd}
\Delta^{\{0,1\}} 
\ar[dr,"F"]
\ar[d, hookrightarrow] \\
\Lambda^n_0 \ar[r] \ar[d, hookrightarrow] &
\Bmod(\C)^\circledast \ar[d,"p"] \\
\Delta^n \ar[r,"\sigma"] \ar[ur,dashed] & \Delta^\op.
\end{tikzcd}
\end{equation}
Base-changing $\Tens^\otimes$ along $\sigma$, we have a $\Delta^n$-family of $\infty$-operads 
\[
r \colon \Tens^\otimes_\sigma \to \Fin_* \times \Delta^n. 
\]
Let $\gamma$ be the composite map
\[
\Tens^\otimes_{\Delta^n} = \Tens_\sigma^\otimes \to \Tens^\otimes \to \Ass. 
\]
Diagram (*) gives rise to a commuting diagram of $\infty$-operad family maps (solid-arrow diagram below).
\begin{equation*}
\tag{**}
\begin{tikzcd}
\Tens^\otimes_{\Delta^{\{0\}}} 
\ar[ddr,"F_0"] \ar[d] \\
\Tens^\otimes_{\Delta^{\{0,1\}}} 
\ar[dr,"F"] \ar[d]
\\
\Tens^\otimes_{\Lambda^n_0}
\ar[r] \ar[d]
&
\C^\otimes \ar[d,"p"]
\\
\Tens^\otimes_{\Delta^n} 
\ar[r,"\gamma"] \ar[ur,dashed, "f"]
&
\Ass
\end{tikzcd}
\end{equation*}
By assumption, the functor $F$ appearing in diagram (**) is an operadic $q$-left Kan extension of $F_0$.
By~\cite[thm. 3.1.2.3 (B)]{HA} there exists a map $f$ of generalized operads making diagram (**) commute.
This solves the lifting problem (*). 
\end{proof}

\subsection{Review of the bar construction}
We review~\cite[Construction 4.4.2.7]{HA}.

\begin{mydef}
\label{AB12}
Let
\[
\Cut: \Delta^\op \to \Ass
\]
be the functor sending the object $[n]$ to $\langle n \rangle$ and the morphism
$
\al: [n] \to [m]
$
in $\Delta$ to the morphism
\[
\Cut(\al): \langle m \rangle \to \langle n \rangle
\]
which sends $i \in \langle m \rangle^\circ$ to the least $j \in \langle n \rangle^\circ$ such that $i \le \al(j)$, or to $*$ if no such $j$ exists (i.e. if $i> \al(j)$ for all $j \in \langle n \rangle^\circ$). Let
\[
\ArDopconv \subset \Fun([1], \Delta)^\op
\]
(denoted $\opnm{Step}$ in \cite[Definition 4.4.1.12]{HA}) denote the full subcategory spanned by morphisms with convex image. Define a functor 
\[
\Cleave: \ArDopconv \to \Tens^\otimes
\]
as follows. Let $f: [n] \to [k]$ be a morphism in $\Delta$ with convex image. Then 
\[
\Cleave(f) := \big(
\langle n \rangle, [k], c_\pm: \langle n \rangle^\circ
\rightrightarrows [k]
\big)
\]
with $c_+$, $c_-$ given by $c_+(i) = f(i)$ and $c_-(i) = f(i-1)$. Now let
\[
\begin{tikzcd}
{[n]} \ar[r,"f "] \ar[d,"\al_0"] & {[k]} \ar[d,"\al_1"]
\\
{[n']} \ar[r,"f'"] & {[k']}
\end{tikzcd}
\]
be a commuting square in $\Delta$ in which $f$, $f'$ both have convex image. Viewing $\al = (\al_0, \al_1)$ as a morphism $f' \to f$ in $\ArDopconv$, we set 
\[
\Cleave(\al) := \big(\Cut(\al_0), \al_1 \big).
\]
\end{mydef}

\begin{mydef}
\label{AB13}
We define a functor 
\[
\Prebar^ +: (\Delta_+)^\op \to \ArDopconv
\]
(denoted `$u_+$' in \cite[Notation 4.4.2.4]{HA}). On the level of objects, $\Prebar^+$ is given as follows. If $n \ge 0$, then $\Prebar^+([n])$ is the morphism with convex image (in fact surjective)
\[
[n+2] \to [2]
\]
sending $0 \mapsto 0$, $n+2 \mapsto 2$ and everything in between $\mapsto 1$. For the remaining case $n = -1$, we set $\Prebar^+([-1])$ equal to the identity map $[1] \to [1]$. 

We turn our attention to morphisms. If $f:[n] \to [m]$ is a morphism in $\Delta$, we define
\[
\Prebar^+(f): \Prebar^+([m]) \to ([n])
\]
using the evident commuting square in $\Delta$:
\[
\begin{tikzcd}
{[n+2]} \ar[r] \ar[d] & {[2]} \ar[d, "\sim"]
\\
{[m+2]} \ar[r] & {[2]}.
\end{tikzcd}
\]
If, instead, $f$ is a morphism $[-1] \to [n]$ in $\Delta_+$, we define $\Prebar^+(f)$ using the commuting square in $\Delta$:
\[
\begin{tikzcd}
{[1]} \ar[d, "id"] \ar[r,"\sim"] &
{\{0, n+2\}} \ar[r]  & {[n+2]} \ar[d, "{\Prebar^+([n])}"]
\\
{[1]} \ar[r,"\sim"]
&
{\{0,2\}} \ar[r] & {[2]}.
\end{tikzcd}
\]
\end{mydef}

\begin{mydef}
\label{AB13}
The composition
\[
(\Delta_+)^\op \xto{\Prebar^+}
\ArDopconv \xto{\Cleave}
\Tens^\otimes
\]
lands inside $ \Tens^\otimes_{\succ}$ (~\cref{bbb}), hence defines a functor 
\[
\rm{bar}^+:  {(\Delta_+)}^\op \to \Tens^\otimes_{\succ}.
\]
Its restriction to $\Delta^\op$ defines a functor
\[
\rm{bar} \colon \Delta^\op \to  \Tens^\otimes_{[2]}.
\]
\end{mydef}

\begin{mydef}\label{def:bar_const}
Let $q \colon \C^\otimes \to \Ass$ be a monoidal $\infty$-category. 
Let $F_0 = {_{A}M_{B}N_{C}}$ be an $(A,B,C)$-module in $\C$.
The \emph{operadic bar construction on $(M,N)$ over $B$} is given by
\[
{\Bar_B^\otimes(M,N)}_\bullet \colon 
\Delta^\op \xto{\rm{bar}} \Tens^\otimes_{[2]} \xto{F_0} \C^\otimes.
\]
The commuting diagram of simplicial sets
\begin{equation*}
\begin{tikzcd}
\Delta^\op \ar[r,"\rm{bar}"] \ar[d, hookrightarrow]
& \Tens^\otimes_{[2]} \ar[d] \ar[r,"F_0"]
& \C^\otimes \ar[d,"q"] \\
\Delta_+^\op \ar[r,"\rm{bar}^+"]
& \Tens^\otimes_{\succ} \ar[r]
& \Ass.
\end{tikzcd}
\end{equation*}
gives rise to a commuting square of simplicial sets (solid arrow diagram) 
\begin{equation*}
\begin{tikzcd}
\Delta^{\set{0}} \ar[rr,"{\Bar_B^\otimes(M,N)_\bullet}"]
\ar[d,hookrightarrow]
& & \Fun(\Delta^\op, \C^\otimes )
\ar[d,"q^{\Delta^\op}"] \\
\Delta^1 \ar[rr,"\sigma"] \ar[urr,dashed,"\beta"]
& & \Fun(\Delta^\op, \Ass).
\end{tikzcd}
\end{equation*}
By~\cite[prop. 3.1.3.1]{HTT}, $q^{\Delta^\op}$ is again a coCartesian fibration.
Hence the edge $\sigma$ in $\Fun(\Delta^\op, \Ass)$ admits a coCartesian lift $\beta$ as indicated, extending the vertex ${\Bar_B^\otimes(M,N)}_\bullet$ of $\Fun(\Delta^\op, \C^\otimes )$.
We define the \emph{(underlying) bar construction on $(M,N)$ over $B$} to be the simplicial object of the underlying category $\C$ given by 
\[
{\Bar_B(M,N)}_\bullet := \beta(1).
\]
The construction provides a natural transformation
\[
\beta \colon {\Bar_B^\otimes(M,N)}_\bullet 
\to 
{\Bar_B(M,N)}_\bullet
\]
of functors $\Delta^\op \to \C^\otimes$.
\end{mydef}

\begin{mydef} \label{def:admits_A_bar}
Let $A\in \Alg(\C)$ be an algebra.
We say that \emph{ $\C$ admits realizations of $A$-bar constructions $\otimes$-compatibly} if the following conditions hold:
\begin{enumerate}
\item
The $\infty$-category $\C$ admits realizations of the bar complexes $\Bar_A(M,N)_\bullet$. 
\item Fix $X,Y\in \C$. One always has
\[
\Bar_A(X \otimes M, N \otimes Y )_\bullet \simeq 
X \otimes \Bar_A( M, N )_\bullet \otimes Y
\]
and we require that the canonical map
\[
\left| \Bar_A(X \otimes M,N\otimes Y )_\bullet \right| \to X \otimes \left| \Bar_A(M,N)_\bullet \right| \otimes Y
\]
be an equivalence. 
\end{enumerate}

Let $\A$ be a class of algebras in $\C$.
We say that \emph{$\C$ admits realizations of $\A$-bar constructions $\otimes$-compatibly} 
if $\C$ admits realizations of $A$-bar constructions $\otimes$-compatibly for every $A\in \A$. When $\A$ is the class of all algebras, we simply say that $\C$ \emph{admits realizations of bar constructions $\otimes$-compatibly}.

Dually, let $A \in \coAlg(\C) = \Alg(\C^\op)^\op$ be a coalgebra. 
We say that \emph{$\C$ admits totalizations of $A$-cobar constructions $\otimes$-compatibly} if $\C^\op$ admits realizations of $A$-bar constructions $\otimes$-compatibly, and similarly for ``\emph{admits totalizations of $\A$-cobar constructions $\otimes$-compatibly}'' and ``\emph{admits totalizations of cobar constructions $\otimes$-compatibly}''.
\end{mydef}

\begin{prop}[Refinement of HA.4.4.2.8 (1)] \label{29.2}
Let
\[
q \colon \C^\otimes \to \Ass
\]
be a monoidal $\infty$-category (so in particular a coCartesian fibration).
Let $F_0 = {_AM_BN_C}$ be an $(A,B,C)$-module in $\C$.
Assume ${\Bar_B(M,N)}_\bullet$ admits a geometric realization compatibly with $\otimes$.
Suppose given a commutative diagram of generalized $\infty$-operads (solid arrow diagram below)
\begin{equation*}
\begin{tikzcd}
\Tens^\otimes_{[2]} \ar[r,"F_0"] \ar[d] & \C^\otimes \ar[d,"q"] \\
\Tens^\otimes_\succ \ar[r,"f"] \ar[ur,dashed,"F"] & \Ass.
\end{tikzcd}
\end{equation*}
Then there exists an operadic $q$-left Kan extension $F$ of $F_0$ (dotted arrow) making the diagram commute.
\end{prop}
In the situation of \cref{29.2} we say that $F$ exhibits $X = F|_{\Tens^\otimes_{[1]}}$ as a relative tensor product of $M$ and $N$ (see~\cite[def. 4.4.2.3]{HA}).
\begin{proof}
By assumption, the underlying bar construction ${\Bar_B(M,N)}_\bullet$ extends to a colimit diagram
\[
{\Bar_B(M,N)}_\bullet^+: 
\Delta^\op_+ \to \C.
\]
Moreover, since the colimit is assumed to be compatible with tensor product, the composite
\[
\Delta^\op_+ \xto{{\Bar_B(M,N)}_\bullet^+} \C
\subset \C^\otimes
\]
is in fact an operadic colimit diagram in $\C^\otimes$, see~\cite[prop. 3.1.1.16]{HTT}.
Denote the colimit by $P \in \C$. 
We now upgrade the \textit{augmented} underlying bar construction ${\Bar_B(M,N)}_\bullet^+$ to a diagram
\[
{\Bar_B^\otimes(M,N)}_\bullet^+: \Delta^\op_+ \to \C^\otimes 
\]
with cone point $P$, such that
\[
{\Bar_B^\otimes(M,N)}_\bullet^+|_{\Delta^\op} = 
{\Bar_B^\otimes(M,N)}_\bullet,
\]
along with a natural transformation
\[
\beta^+: {\Bar_B^\otimes(M,N)}_\bullet^+ 
\to {\Bar_B(M,N)}_\bullet^+
\]
extending the given natural transformation
\[
\beta \colon {\Bar_B^\otimes(M,N)}_\bullet 
\to {\Bar_B(M,N)}_\bullet.
\]
The natural transformation $\beta^+$ will be used to verify that ${\Bar_B^\otimes(M,N)}^+_\bullet$ too is an operadic colimit diagram.

Let $\a = \langle 1 \rangle$ denote the unique color of $\Ass$.
The map
\[
\Delta^\op_+ \xto{\text{bar}^+} \Tens^\otimes_{\succ} 
\to \Ass
\]
corresponds to a map
\[
\al^+_o \colon \Delta^\op \to \Ass_{/\a}. 
\]
Since $\a$ is terminal in the overcategory (and since $\Ass$ is the nerve of a 1-category), there's a unique natural transformation
\[
\al^+ \colon \Delta^\op \times \Delta^1 \to \Ass_{/\a}
\]
from $\al^+_o = \al^+|_{\Delta^\op \times \{0\}}$ to the constant diagram $\al^+|_{\Delta^\op \times \{1\}}$ with value $\a$.
Direct verification shows that the square solid-arrow diagram of simplicial sets 
\begin{equation*}
\begin{tikzcd}
\Delta^\op \times \set{1}
\ar[r,"{\Bar_B(M,N)}_\bullet^+"]
\ar[d,hookrightarrow,"\kappa"]
& \C^\otimes_{/P}
\ar[d,"\phi"] \\
\Delta^\op \times \Delta^1
\ar[r,"{(\beta,\al^+)}"]
\ar[ur,dashed,"\beta^+"]
& \C^\otimes \times_{\Ass} \Ass_{/\a}
\end{tikzcd}
\end{equation*}
commutes.
By~\cite[prop. 2.1.2.1]{HTT}, $\phi$ is a right fibration. 
By~\cite[cor. 2.1.2.7]{HTT} $\kappa$ is right anodyne. 
Hence there exists a dotted arrow $\beta^+$ as indicated, making both triangles of simplicial sets commute.
This completes the construction of the promised natural transformation $\beta^+: \Delta^\op_+ \times \Delta^1 \to \C^\otimes$.

We now verify that ${\Bar_B^\otimes(M,N)}_\bullet^+$ is an operadic colimit diagram.
By construction, for every vertex $x \in \Delta^\op$, the restriction $\beta^+|_{\{x\} \times \Delta^1} = \beta|_{\{x\} \times \Delta^1}$ is a $q$-coCartesian edge in $\C^\otimes$.
Moreover, $\beta^+|_{\{[-1]\} \times \Delta^1}$ is equal to the identity morphism of $P$, hence is again $q$-coCartesian.
Hence~\cite[prop. 3.1.1.15]{HA} applies to show that the \textit{augmented} operadic bar construction ${\Bar_B^\otimes(M,N)}_\bullet^+$ is an operadic $q$-colimit diagram.
The theorem follows by applying~\cite[prop. 4.4.2.5]{HA}. 
\end{proof}

\begin{prop}[Refinement of Lemma HA.4.4.3.9 (1)] \label{27.3}
Let $q: \C^\otimes \to \Ass$ be a monoidal $\infty$-category and let $\A$ be a class of algebra objects of $\C$ closed under equivalences.
Assume that $\C$ admits realizations of $\A$-bar constructions $\otimes$-compatibly.
Let 
\[
p \colon {\Bmod(\C)}^\circledast \to \Delta^\op
\]
be the associated Morita pre-\textit{double $\infty$-category} (\cref{def:BMod_C_star}). 
Then ${\Bmod(\C)}^\circledast$ admits coCartesian lifts from $\A$-multimodules (\cref{666a}).
\end{prop}
\begin{proof}
By \cref{2} it will suffice to show that for every map
\[
\al \colon [n]\to[m]
\]
in $\Delta$ and every $\A$-multimodule $F_0 \in \Alg_{\Tens_{[m]}/\Ass}(\C)$,
there exists an operadic $q$-left Kan extension $F \in \Alg_{\Tens_\al/\Ass}(\C)$ of $F_0$.
For each $r \ge 0$, we define 
\[
\al_i:[r] \to [r+1]
\]
to be the order-preserving map with image $[r+1] \setminus \{i\}$. 
We can factor $\al$ as a composition 
\[
[n] \xto{\beta} [k] \to [k+1] \to \cdots \to [m]
\]
where $\beta$ has convex image and the remaining maps are of the form 
\[
\al_i:[p-1] \to [p]
\]
for some $0<i<p$.
Using~\cite[thm. 4.4.3.1]{HA} (flatness of $\Tens^\otimes \to \Delta^\op$) and~\cite[thm. 3.1.4.1]{HA} (transitivity of operadic left Kan extensions) it's enough to show that if $\al$ is of either of these forms,
then there exists an operadic $q$-left Kan extension $F$ as above such that the associated output multimodule
\[
F_1: \Tens^\otimes_{[n]} \to \Tens_{\al}^\otimes \xto{F} \C^\otimes
\]
is again an $\A$-multimodule. 
When $\al$ has convex image this follows from~\cite[lem. 4.4.3.5.(1)]{HA}. 
Indeed, let $f$ denote the map
\[
\Tens^\otimes_{[n]} \to \Ass.  
\]
In the notation of that lemma, for every object $X \in \Tens_{[n]}$, there exists a map 
\[
F_0(v_\al X) \to C
\]
in $\C^\otimes$ lying over the map
\[
\tag{*}
f(v_\al X) \to f(X)
\]
given by an operadic $q$-colimit diagram $\Delta^1 \to \C^\otimes$;
this holds trivially since the map (*) is necessarily the identity map of the unique color of $\Ass$. 
Hence~\cite[lem. 4.4.3.5(1)]{HA} implies that there exists an operadic $q$-left Kan extension $F$ of $F_0$.
Moreover, by construction, if $F_0$ corresponds to an $A$-multimodule
\[
_{A_0}{M^{0,1}}_{A_{1}} M^{1,2}_{A_2}
\cdots
_{A_{n-1}} M^{n-1,n}_{A_n}
\]
then $F_1$ corresponds to a two-sided truncation of the above, and hence to an $\A$-multimodule. 

When $\al$ is of the form $\al_i$ we construct the desired operadic $q$-left Kan extension $F$ using~\cite[lem 4.4.3.8]{HA}.
Conditions (a), (b) and (c) hold trivially (since in the notation of that lemma the maps $f(a_{\al(j)}) \to f(b_j), f(m_{j-1,j}) \to f(n_{j-1,j}), f(m_{j,j+1}) \to f(n_{j,j+1})$ are all the identity of the unique object of $\Ass_{<1>}$.)
Condition (d) holds by \cref{29.2}.
Condition ($a'$) shows that $F_1$ is again an $\A$-multimodule.
\end{proof}

\begin{cor} \label{link4}
In the situation and the notation of \cref{27.3}, let
\[
{\Bmod_{\A}(\C)}^\circledast \subset {\Bmod(\C)}^\circledast
\]
be the full subcategory supported on $\A$-multimodules.
Assume that $\C$ admits realizations of $\A$-bar constructions $\otimes$-compatibly.
Then the restriction of $p$ to
\[
{\Bmod_\A(\C)}^{\circledast} \to \Delta^\op
\]
is a coCartesian fibration.
\end{cor}
\begin{proof}
The restriction 
\[
p' \colon {\Bmod_\A(\C)}^\circledast \to {\Bmod(\C)}^\circledast \xto{p} \Delta^\op
\]
is an inner fibration as a composition of inner fibrations.
Let $X \in  {\Bmod_\A(\C)}^\circledast$ and let $\al \colon x \to y$ be a map in $\Delta^\op$ with $x=p(X)$.
By \cref{27.3} there exists a $p$-coCartesian lift $f \colon X \to Y$ of $\al$ in ${\Bmod_\A(\C)}^\circledast$.
By~\cite[prop. 2.4.1.3(3)]{HTT} the edge $f$ is $p'$-coCartesian.
\end{proof}

\subsection{The Morita double \texorpdfstring{$\infty$-}{infinity}category of \texorpdfstring{$\A$}{A}-multimodules}
\begin{terminology} \label{05192}
We say that a morphism $\al \colon [n] \to [m]$ in $\Delta^\op$ is \emph{convex} if the associated map of totally ordered sets $[n] \from [m]$ has convex image.
Noting that identity morphisms are convex, and that the composite of two convex morphisms is again convex, we define the \emph{convex ordinal category} $\Delta^\op_\text{Convex} \subset \Delta^\op$ to be the subcategory whose objects are the same as those of $\Delta^\op$ and whose morphisms are the convex morphisms in $\Delta^\op$. 
\end{terminology}

\begin{prop} \label{05191}
Let
\[
q \colon \C^\otimes \to \Ass
\]
be a monoidal $\infty$-category and let 
\[
p \colon B^\circledast = {\Bmod(\C)}^\circledast \to \Delta^\op
\]
be the associated Morita pre-\textit{double $\infty$-category} (\cref{27.1}). 
Then the pullback $p_\text{Convex}$ of $p$ to $\Delta^\op_\text{Convex}$ as in the following diagram
\begin{equation*}
\begin{tikzcd}
B^\circledast_\text{Convex} \ar[r] \ar[d,"p_\text{Convex}"] & B^\circledast \ar[d,"p"] \\
\Delta^\op_\text{Convex} \ar[r] & \Delta^\op
\end{tikzcd}
\end{equation*}
is a coCartesian fibration. 
Moreover, for every $n>1$, the induced map
\[
B^\circledast_{[n]} \to
B^\circledast_{[1]} \times_{B^\circledast_{[0]}} 
\cdots \times_{B^\circledast_{[0]}} B^\circledast_{[1]} 
\]
is an equivalence of $\infty$-categories. 
\end{prop}

\begin{proof}
Let $\al:[m] \to [n]$ in $\Delta^\op$ be convex, and let $p(X_0) = [m]$.
Together, $\al$ and $X_0$ give rise to a commuting square (solid arrow diagram)
\begin{equation*}
\begin{tikzcd}
\Tens^\otimes_{[m]} \ar[r,"X_0"] \ar[d] & \C^\otimes \ar[d] \\
\Tens^\otimes_\al \ar[r] \ar[ur,dashed,"X"] & \Ass.
\end{tikzcd}
\end{equation*}
The condition of~\cite[lem. 4.4.3.5(1)]{HA} applies trivially to show that there exists an operadic $q$-\textit{left Kan extension} $X$ of $X_0$ as indicated.
By~\cref{2}, $X$ corresponds to a coCartesian edge in $B^\circledast$ with source $X_0$ lifting $\al$.
By~\cite[prop. 2.4.1.3(2)]{HTT} it follows that $p_\text{Convex}$ is a coCartesian fibration.
Using~\cite[prop. 4.4.1.11]{HA} (Segal condition for $\Tens^\otimes$), we see that for each $n \ge 0$, the inclusions $[1] \simeq \{i-1, i\} \cof [n]$ induce an equivalence of $\infty$-categories as stated. 
\end{proof}

\begin{terminology}
\label{capable}
We say a class $\A$ of algebras in a monoidal $\infty$-category $\C$ is \emph{capable of relative tensor products} if $\A$ is closed under equivalences and if $\C$ admits geometric realizations of $\A$-bar constructions $\otimes$-compatibly.
\end{terminology}

\begin{mydef}[{\cite[def. 2.1.4]{Rune_operads_via_SymSeq} }] \label{def:double_inf_cat}
A \emph{double $\infty$-category} is a coCartesian fibration $\M \to \Delta^{op}$ such that for all $n > 1$, the induced functor of $\infty$-categories 
\[
\M_{[n]} \xto{} \M_{[1]} \times_{\M_{[0]}} 
\cdots \times_{\M_{[0]}}  \M_{[1]}
\]
is an equivalence. 
\end{mydef}

\begin{cor} \label{link1}
Let $q \colon \C^\otimes \to \Ass$ be a monoidal $\infty$-category, and let $\A$ be a class of algebras in $\C$ which is capable of relative tensor products.  Let
\[
{\Bmod_\A(\C)}^\circledast \subset {\Bmod(\C)}^\circledast
\]
denote the full subcategory supported on the class of $\A$-multimodules.
Then
\[
p_\A: {\Bmod_\A(\C)}^\circledast \to \Delta^\op
\]
is a \emph{double $\infty$-category}.
\end{cor}
\begin{proof}
According to \cref{link4}, $p_\A$ is a coCartesian fibration.
According to \cref{05191}, $p_\A$ satisfies the Segal condition.
\end{proof}

In the situation of the above corollary, we refer to $ {\Bmod_\A(\C)}^\circledast$ as the \emph{Morita double $\infty$-category of $\C$ over $\A$}. We end this section with a remark which, we hope, will help to clarify its structure.  

\begin{rem}
Let $q:\Cc^\otimes \to \Ass$ be a monoidal $\infty$-category, and assume for simplicity that $\Cc$ is capable of relative tensor products so that 
\[
p: \Bmod(\Cc)^\circledast \to \Delta^\op
\]
is a double $\infty$-category. If
\[
\al:[m] \to [n]
\]
is a morphism in $\Delta^\op$, we have an induced functor
\[
\al_!: \Bmod(\Cc)^\circledast_{[m]}
\to
\Bmod(\Cc)^\circledast_{[n]}.
\]
Moreover, if the morphism $F_\to$ in the diagram
\begin{equation*}
\begin{tikzcd}
&
\Tens^\otimes_{[m]} \ar[r,"F"] \ar[d, "\iota_0"] & \C^\otimes \ar[d] \\
\Tens^\otimes_{[n]} \ar[r,"\iota_1"]
&
\Tens^\otimes_\al \ar[r, "f"] \ar[ur,dashed,"F_\to"] & \Ass
\end{tikzcd}
\end{equation*}
(in which $f$ is the composite $\Tens^\otimes_\al \to \Tens^\otimes \to \Ass$) is an operadic $q$-left Kan extension, then 
\[
\al_! F = F_\to \circ \iota_1.
\]
Two cases are of particular interest.
\begin{enumerate}

\item

Suppose the morphism $\al^\op: [m] \from [n]$ in $\Delta$ has convex image. Then $\iota_0$ admits a natural retraction $V_\al$  and 
\[
F_\to = F \circ V_\al
\]
is an operadic $q$-left Kan extension (compare Example 4.4.3.6 of Higher Algebra \cite{HA} or the proof of ~\cref{20510e}). Moreover,
\[
V_\al \circ \iota_1 = v_\al
\]
is given by composition with $\al$. It follows that the $n$-module $\al_! F$ is a projection of the $m$-module $F$ onto those components singled out by $\al$. 

\item

Suppose $n = m-1$ and $\al^\op: [m] \from [n]$ is an inner face map. Then the operadic $q$-left Kan extension $F_\to$ is constructed via an associated operadic $q$-left Kan extension along
\[
\Tens^\otimes_{[2]} \to \Tens^\otimes_\succ
\]
which \textit{defines} the relative tensor product. In this case, the $n$-module $\al_! F$ is given by a relative tensor product of neighboring bimodules in the $m$-module $F$ as indicated by $\al$. 

\end{enumerate}

To assemble these two building blocks into a natural \textit{formula} for general $\al$, it's useful to introduce the following notation: if $F: \Tens^\otimes_{[k]} \to \Cc^\otimes$ is a $k$-module in the monoidal $\infty$-category $\Cc$, we denote the associated algebras and bimodules by
\[
F_{ij} = F(\mM_{ij})
\quad \mbox{and} \quad
F_i = F(\aA_i).
\]
Thus, for each $i = 1, \dots, k$, $F_{i-1, i}$ is an $F_{i-1}$-$F_{i}$-bimodule. Also, if
\[
\al:[m] \to [n]
\]
is a morphism in $\Delta^\op$ as above, and $i \in [n] = \{0, \dots, n\}$, we denote $\al^\op(i)$ by $\al_i$. In these notations, we have for $i = 0, \dots, m$, 
\[
(\al_! F)_{i} = F_{\al_i},
\]
and for $0 \le i < i+1 = j \le m$, 
\[
(\al_! F)_{ij} = 
F_{\al_i, \al_i+1}
\otimes_{F_{\al_{i}+1}}
F_{\al_i+1, \al_i+2}
\otimes_{F_{\al_{i}+2}}
\cdots
\otimes_{F_{\al_j-1}}
F_{\al_j-1, \al_j}.
\]

\end{rem}

\section{Action of endomorphism categories on morphism categories in a double \texorpdfstring{$\infty$}{infinity}-category} \label{sec:BLMod}


In this section we construct the $\LM$-monoidal $\infty$-category $\BLMod_A(\C)$ (see \cref{def:BLMod_A}), which exhibits the $\infty$-category of left $A$-modules as left tensored over $A$-bimodules. We assume only that $\C$ admits bar constructions over $A$ which are compatible with tensor product.
In particular, we do not assume that $\C$ admits geometric realizations in general. As explained in the introduction, this level of generality is needed in order to handle relative tensor products in the twisted arrow category (see \cref{sec:Twisted_arrow_bar_construction}).

In fact, our construction takes place in the general setting of a double $\infty$-category. If $\M$ is a double $\infty$-category and $A$, $B$ are objects, we define the \emph{$\infty$-category $\Map_\M^h(B,A)$ of horizontal morphisms from $B$ to $A$ in $\M$} in \cref{lbl4}.
We then construct an $\LM$-monoidal category 
$
{\opnm{LMap}^h_\M(B,A)}^\otimes
$
which exhibits the $\infty$-category $\Map_\M^h(A,B)$ as left-tensored over the monoidal $\infty$-category ${\Map^h_\M(A,A)}^\otimes$ (\cref{lb9}).
We obtain the desired action of bimodules on left modules by applying this construction to the double $\infty$-category 
\[
\M = {\Bmod_\A(\C)}^\circledast \to \Delta^\op
\]
of \cref{link2}.

\subsection{The $\infty$-category of horizontal morphisms}
\begin{mydef}\label{lbl2}
For $\Yy$ an $\infty$-category we let ${(\Set^+_\Delta)}_{/\Yy}$ denote the simplicial model category of marked simplicial sets over $\Yy$ with coCartesian model structure (\cite[prop. 3.1.3.7, rem. 3.1.3.9, cor. 3.1.4.4]{HTT}). 
We let $\Cat_{\mathrm{coCart}/\Yy}$ be the simplicial nerve $N \big({(\Set^+_\Delta)}_{/\Yy}^\circ \big)$ of the full subcategory ${(\Set^+_\Delta)}_{/\Yy}^\circ$ of fibrant-cofibrant objects. 
\end{mydef}

\begin{rem}\label{lbl3}
There is a natural bijection between vertices of $\Cat_{\mathrm{coCart}/\Yy}$ and coCartesian fibrations of simplicial sets $\Xx \to \Yy$. 
Given two such, $\Xx$, $\Xx'$, the set of edges $\Xx \to \Xx'$ in $\Cat_{\mathrm{coCart}/\Yy}$ is in canonical bijection with the set of 
morphisms of simplicial sets over $\Yy$
which send coCartesian edges to coCartesian edges. 
\end{rem}

\begin{rem}\label{lbl3.1}
The proof of Theorem 1.2 of \cite{GepnerHauRu} may be dualized to show that the inclusion 
\[
\Cat_{\mathrm{coCart}/\Yy} \to (\Cat_\infty)_{/\Yy}
\] 
admits a left adjoint, hence preserves all limits. The further functor
\[
(\Cat_\infty)_{/\Yy} \to \Cat_\infty
\]
preserves pullbacks.

\end{rem}

\begin{mydef}\label{lbl4}
Let $\M$ be a double $\infty$-category (\cref{def:double_inf_cat}).
Given objects $A,B \in \M_{[0]}$, we define the \emph{$\infty$-cat\-egory of horizontal morphisms $B \to A$ in $\M$},
\(
\Map_{\M}^h(B,A),
\)
by the pullback diagram
\[
\begin{tikzcd}
\Map_{\M}^h(B,A) \ar[r] \ar[d] \pbcorner &
\M_{[1]} \ar[d,"\delta_0 \times \delta_1"]
\\
\set{*} \ar[r,"{(B,A)}"] &
\M_{[0]} \times \M_{[0]}
\end{tikzcd}
\]
in $\Cat_\infty$, where the $\delta_i$ are induced by the two face maps $[0] \to [1]$.
\end{mydef}

\begin{ex}\label{lbl4.5}
Let $\C^\otimes \to \Ass$ be a monoidal $\infty$-category and let $\A$ be a class of algebras in $\Cc$ capable of relative tensor products (\cref{capable}); we denote the associated full subcategory of $\Alg \C$ again by $\A$.
Then
\[
{\Bmod(\C)}^{\circledast}_{[0]} = \A
\] 
and given $A,B \in \A$ we have equivalences of $\infty$-categories
\[
\Map^h_{\Bmod(\C)}(B,A) \simeq {\BMod{A}{B}(\C)}.
\]
\end{ex}

\subsection{Left actions of horizontal endomorphisms} 
Let $\M$ be a double \oo-category and let $A,B \in \M_{[0]}$ be objects. In \cref{lb9} we will endow $\Map^h_\M(B,A)$ with a left-action of $\Map^h_\M(A,A)$.
The monoidal structure of $\Map^h_\M(A,A)$ is constructed in~\cite[def. 3.4.2 and proof of prop. 3.4.8]{Rune_operads_via_SymSeq}.
We recall its construction after fixing some relevant notation.
\begin{notation} \label{bl1}
Let $\Xx_\bullet$ be a simplicial object in an $\infty$-category $\C$ admitting limits.
Recall that the \emph{$0$-coskeleton} $\cosk_0 \Xx_\bullet$ is the simplicial object given by the right Kan extension
\[
\begin{tikzcd}
\set{*} \ar[r, "\Xx_0"] \ar[d, phantom, "\bigcap" description] & \C \\
\Delta^\op \ar[ur,"\cosk_0 \Xx_\bullet"'] &
\end{tikzcd}
\]
(see, e.g.,\cite[not. 6.5.3.1]{HTT}). 
In particular this applies to a diagram of $\infty$-categories $\Xx_\bullet \colon \Delta^\op \to \Cat_\infty$.

We now use straightening/unstraightening (\cite[thm. 3.2.0.1]{HTT}) to define the $0$-coskeleton of a coCartesian fibration over $\Delta^\op$.
Given a coCartesian fibration
\[
\Xx \to \Delta^\op
\]
we let $\Xx_\bullet \colon \Delta^\op \to \Cat_\infty$ denote its straightening.
Define the coCartesian fibration 
\[
\cosk_0 \Xx \to \Delta^\op
\]
as the un-straightening of $\cosk_0 \Xx_\bullet \colon \Delta^\op \to \Cat_\infty$.
Note that 
\[
{(\cosk_0\Xx)}_{[n]} \simeq \Xx_{[0]} \times \cdots \times \Xx_{[0]}
\]
($n+1$ copies).
\end{notation}

\begin{construction}[{\cite[def. 3.4.2]{Rune_operads_via_SymSeq}}] \label{lbl5}
Let $\M \to \Delta^\op$ be a double $\infty$-category (\cref{def:double_inf_cat}).
By construction, $\cosk_0\M$ comes equipped with a canonical morphism of coCartesian fibrations over $\Delta^\op$
\[
\ka \colon \M \to \cosk_0 \M
\]
in which the map of fibers over $[n] \in \Delta^\op$ is the map
\[
\M_{[n]} \to  \M_{[0]} \times \cdots \times \M_{[0]}
\]
induced by the $n+1$ face maps $[n] \from [0]$ in $\Delta$.
Given $A \in \M_{[0]}$, we construct a section 
$A^* \colon \Delta^\op \to \cosk_0 \M$ 
which sends 
\[ 
\Delta^\op \ni [n] \mapsto (A, \dots, A) \in {(\cosk_0 \M)}_{[n]}
\] 
as follows.
Let $\opnm{Const}(\ast)$ denote the constant simplicial $\infty$-category
\[
\opnm{Const}(\ast) \colon  \Delta^\op \to \Cat_\infty
\]
with value $\ast \in \Cat_\infty$, and let $\iota_{[0]}$ denote the functor $\ast \to \Delta^\op$ classified by $[0] \in \Delta^\op$.
The inclusion $A \in \M_{[0]}$ gives rise to a natural transformation
\[
\opnm{Const}(\ast) \circ \iota_{[0]} \to \M_{[0]}
\]
of functors $\ast \to \Cat_\infty$, hence to a natural transformation 
\[
\opnm{Const}(*) \to \cosk_0 \M_\bullet
\]
of functors $\Delta^\op \to \Cat_\infty$, since the right Kan extension $\cosk_0$ is right adjoint to 
\[
\iota_{[0]}^* \colon \Fun(\ast,\Cat_\infty) \to \Fun(\Delta^\op,\Cat_\infty).
\]
Unstraightening defines a section $A^*$ as hoped.

We define ${\End^h_\M(A)}^\circledast$ by the pullback diagram
\[
\tag{*}
\begin{tikzcd}
\End^h_\M(A)^\circledast \ar[r] \ar[d,"{\ka_A}"] \pbcorner &
\M \ar[d,"{\ka}"]
\\
\Delta^\op \ar[r,"{A^*}"'] &
\cosk_0 \M.
\end{tikzcd}
\]
in the category $\Cat_{\mathrm{coCart}/\Delta^\op}$ of \cref{lbl2}.
The associated simplicial $\infty$-category
\[
{\End^h_\M(A)}^\circledast_\bullet \colon \Delta^\op \to \Cat_\infty
\]
is a monoid object in the sense of~\cite[def. 4.1.2.5]{HA}.
Hence by~\cite[rem. 4.1.2.6, prop. 2.4.2.5 and rem. 2.4.2.6]{HA} we obtain a monoidal $\infty$-category 
\[
{\Map^h_\M(A,A)}^\otimes \to \Ass
\]
with underlying $\infty$-category equivalent to $\Map^{h}_\M(A,A)$.
\end{construction}

\begin{ex}\label{lb5.5}
Let $\C^\otimes \to \Ass$ be a monoidal $\infty$-category, let $\A$ be a class of algebras in $\Cc$ capable of relative tensor products (\cref{capable}), and let $\M = {\Bmod_\A(\C)}^\circledast$ be the Morita double $\infty$-category of $\A$-multimodules constructed in \cref{link1}.

The map $\ka$ of \cref{lbl5} diagram (*) corresponds via straightening to a morphism of coherently commutative simplicial $\infty$-categories given on the level of objects and 1-morphisms by a diagram like so, in which $\A$ is regarded as a full subcategory of $\Alg \C$ (we omit degeneracy maps):  
\[
\begin{tikzcd}
\vdots 
\ar[d, shift left=3]\ar[d, shift left=1] \ar[d, shift right=1] \ar[d, shift right=3]
&
\vdots
\ar[d, shift left=3]\ar[d, shift left=1] \ar[d, shift right=1] \ar[d, shift right=3]
\\
\Bmod_\A(\C) \times_{\A} \Bmod_\A(\C)
\ar[d, shift left=2] \ar[d] \ar[d, shift right=2]
\ar[r]
&
\A \times \A \times \A
\ar[d, shift left=2] \ar[d] \ar[d, shift right=2]
\\
\Bmod_\A(\C) 
\ar[d, shift left=1] \ar[d, shift right=1]
\ar[r]
&
\A \times \A
\ar[d, shift left=1] \ar[d, shift right=1]
\\
\A
\ar[r]
&
\A.
\end{tikzcd}
\]

Given $A \in \A$ then,  $\End^h_{{\Bmod_\A(\C)}^\circledast}(A)$ corresponds to a coherently commutative simplicial $\infty$-category given on the level of objects and $1$-morphisms by
\[
\begin{tikzcd}
\vdots 
\ar[d, shift left=3]\ar[d, shift left=1] \ar[d, shift right=1] \ar[d, shift right=3]
\\
\BMod{A}{A}(\C) \times \BMod{A}{A}(\C)
\ar[d, shift left=2] \ar[d] \ar[d, shift right=2]
\\
\BMod{A}{A}(\C) 
\ar[d, shift left=1] \ar[d, shift right=1]
\\
\set{*}.
\end{tikzcd}
\]
As a special case of \cref{lbl5}, this endows the category
\[
\BMod{A}{A}(\C)
\simeq
\Map^h_{
{\Bmod_\A(\C)}^\circledast
}
(A,A)
\]
of $A$-$A$-bimodules in $\Cc$ with a monoidal structure.
\end{ex}

\begin{rem}
\textit{Higher Algebra} does not explicitly mention a monoidal structure on the particular model $\BMod{A}{A}(\C)$ for the \oo-category of bimodules. Instead, in \cite[4.4.3.12]{HA}, Lurie endows a different model $\opnm{Mod}_A^\mathrm{Assoc}(\Cc)$ with a monoidal structure. He also obtains an equivalence of underlying \oo-categories 
\[
\opnm{Mod}_A^\mathrm{Assoc}(\Cc) \simeq 
\BMod{A}{A}(\C)
\]
and states that the tensor product on $\opnm{Mod}_A^\mathrm{Assoc}(\Cc)$ corresponds to the relative tensor product of bimodules. 

However, as the proof of \cite[4.4.3.12]{HA} shows, Lurie's construction of a monoidal structure on $\opnm{Mod}_A^\mathrm{Assoc}(\Cc)$ essentially includes our construction of a monoidal structure on $\BMod{A}{A}(\C)$ followed by a further equivalence of \oo-operads
\[
\opnm{Mod}_A^\mathrm{Assoc}(\Cc)^\otimes \simeq 
\BMod{A}{A}(\C)^\otimes
\]
over $\Assoc$ furnished by \cite[Theorem 4.4.1.28]{HA}. In particular, \cite[4.4.3.12]{HA} holds under our (weaker) assumptions. 
\end{rem}

We now turn to the construction of the action of the $\infty$-category
\(
\Map^{h}_\M(A,A)
\)
of horizontal endomorphisms of $A$ in $\M$ on the $\infty$-category $\Map^{h}_\M(B,A)$ of horizontal morphisms from $B$ to $A$. 
We first import the classical construction of the d\'ecalage of a simplicial set to the context of coCartesian fibrations over $\Delta^\op$.

\begin{construction}\label{lb6}
Let $+^\op$ denote the functor $\Delta \to \Delta$ defined on objects by
\[
S \mapsto S_+ := S \coprod \{\infty\}
\]
and on morphisms in the obvious way.
Let $+$ denote the induced functor $\Delta^\op \to \Delta^\op$. 
Suppose $\Yy \to \Delta^\op$ is a coCartesian fibration.
Then the \emph{d\'ecalage $\opnm{Dec} \Yy \to \Delta^\op$ of $\Yy$} is the coCartesian fibration obtained from $\Yy$ via pullback along the functor $+$. 

We let 
$\nato := \Delta^\op \times \Delta^1$.
Consider the two functors of 1-categories 
\[
\begin{tikzcd}
\Delta^\op \ar[r, bend left, "{+}"] \ar[r, bend right, "{\Id}"']
& \Delta^\op.
\end{tikzcd}
\]
The natural inclusions $S \subset S_+$ give rise to a natural transformation 
\(
+ \Rightarrow \Id,
\)
hence to a functor
\[
\delta \colon  \nato \to \Delta^\op.
\]
Similarly, let $\ep^\op$ denote the composite functor
\[
\Delta \to \{*\} \to \Delta
\]
sending all objects to a one-element totally ordered set, and let $\ep$ denote the induced functor on opposite categories. 
The maps of sets $\{*\} \to S_+$ sending $* \mapsto \infty$ give rise to a natural transformation $+ \Rightarrow \ep$, hence to a functor 
\[
\pi \colon \nato = \Delta^\op \times \Delta^1
\to \Delta^\op.
\]

Returning to the coCartesian fibration $\Yy \to \Delta^\op$, the pullbacks
\[
\delta^* \Yy \to \nato,
\quad 
\pi^* \Yy \to \nato
\]
correspond to morphisms 
\[
\begin{tikzcd}
\Dec \Yy \ar[r] \ar[d] & \Yy \\
\Yy_{[0]} \times \Delta^\op
\end{tikzcd}
\]
of coCartesian fibrations over $\Delta^\op$.
\end{construction}

\begin{prop}\label{lb7}
Let $\Yy \to \Delta^\op$ be a coCartesian fibration of simplicial sets.
Then the d\'ecalage of $\cosk_0 \Yy$ decomposes as a product
\[
\opnm{Dec} \cosk_0 \Yy \simeq
(\cosk_0 \Yy) \times_{\Delta^\op} (\Yy_{[0]} \times \Delta^\op)
\]
of coCartesian fibrations over $\Delta^\op$, where $\Yy_{[0]} \times \Delta^\op$ denotes the constant fibration with fiber $\Yy_{[0]}$.
\end{prop}
\begin{proof}
Applying \cref{lb6} to $\cosk_0 \Yy$, and noting that
\[
{(\cosk_0\Yy)}_{[0]} = \Yy_{[0]},
\]
we obtain a map $\psi$ from left to right. 
Additionally, taking fibers over $[n] \in \Delta$, we find that both sides are equivalent to ${(\Yy_{[0]})}^{\times (n+2)}$ compatibly with $\psi$.
It follows that $\psi$ is an equivalence. 
\end{proof}

\begin{ex}\label{lb8}
We apply \cref{lb6} and \cref{lb7} to the Morita double $\infty$-category ${\Bmod_\A(\C)}^\circledast \to \Delta^\op$ of \cref{link1}. 
The morphism of coCartesian fibrations 
\[
\Dec {\Bmod_\A(\C)}^\circledast \to {\Bmod_\A(\C)}^\circledast 
\]
corresponds via straightening to a diagram of $\infty$-categories
\[
\tag{D}
\begin{tikzcd}
\vdots 
\ar[d, shift left=3]\ar[d, shift left=1] \ar[d, shift right=1] \ar[d, shift right=3]
&
\vdots
\ar[d, shift left=3]\ar[d, shift left=1] \ar[d, shift right=1] \ar[d, shift right=3]
\\
\Bmod_\A(\C) \times_{\A} \Bmod_\A(\C)
\times_{\A} \Bmod_\A(\C)
\ar[d, shift left=2] \ar[d] \ar[d, shift right=2]
\ar[r]
&
\Bmod_\A(\C) \times_{\A} \Bmod_\A(\C)
\ar[d, shift left=2] \ar[d] \ar[d, shift right=2]
\\
\Bmod_\A(\C) \times_{\A} \Bmod_\A(\C) 
\ar[d, shift left=1] \ar[d, shift right=1]
\ar[r]
&
\Bmod_\A(\C)
\ar[d, shift left=1] \ar[d, shift right=1]
\\
\Bmod_\A(\C)
\ar[r]
&
\A
\end{tikzcd}
\]
in which the horizontal maps forget the rightmost algebra as well as the rightmost bimodule. 
The morphism of coCartesian fibrations 
\[
\Dec \cosk_0 {\Bmod_\A(\C)}^\circledast
\to \cosk_0 {\Bmod_\A(\C)}^\circledast 
\]
corresponds via straightening to a diagram of $\infty$-categories
\[
\tag{DC}
\begin{tikzcd}
\vdots 
\ar[d, shift left=3]\ar[d, shift left=1] \ar[d, shift right=1] \ar[d, shift right=3]
&
\vdots
\ar[d, shift left=3]\ar[d, shift left=1] \ar[d, shift right=1] \ar[d, shift right=3]
\\
\A \times \A \times \A \times \A
\ar[d, shift left=2] \ar[d] \ar[d, shift right=2]
\ar[r]
&
\A \times \A \times \A 
\ar[d, shift left=2] \ar[d] \ar[d, shift right=2]
\\
\A \times \A \times \A  
\ar[d, shift left=1] \ar[d, shift right=1]
\ar[r]
&
\A \times \A
\ar[d, shift left=1] \ar[d, shift right=1]
\\
\A \times \A
\ar[r]
&
\A.
\end{tikzcd}
\]
The proof of \cref{lb7} shows that, regarded as a morphism of simplicial objects, this is just projection from a product. 
\end{ex}

The left action of $ \Map^h_\M (A,A)$ on $\Map^h_\M(B,A)$ will be given in terms of a \textit{left action object in $\Cat_\infty$}; we recall how these correspond to $\LM$-monoidal categories.

\begin{rem}
\label{AC13}
The construction of \cite[prop. 4.2.2.9]{HA} provides us with an equivalence of $\infty$-categories
\[
\Mon_{\LM}(\Cat_\infty) \to \mathrm{LMon}(\Cat_\infty)
\]
from \emph{$\LM$-monoid objects} in $\Cat_\infty$ to left action objects in $\Cat_\infty$.
According to~\cite[prop. 2.4.2.5 and rem. 2.4.2.6]{HA}, $\LM$-monoid objects of $\Cat_\infty$ can be identified with $\LM$-algebra objects of $\Cat_\infty$, which, in turn, may be identified with $\LM$-monoidal categories.

More explicitly, suppose given a coCartesian fibration $\Xx \to \nato$ such that the associated natural transformation in $\Fun(\Delta^\op, \Cat_\infty)$ is a left action object. Then the associated $\LM$-monoidal category is given simply by pullback along a certain composition of functors
\[
\nato
\xto{\kappa}
\ArDopconv \times_{\Delta^\op} \{[1]\}
\xto{\Cleave_{[1]}}
\Tens^\otimes_{[1]} = \BM^\otimes
\]
whose image lies in $\LM^\otimes \subset \BM^\otimes$.  

These are constructed out of the category $\ArDopconv$ and the functors $\Cut$ and $\Cleave$ which we've assembled from various places in \textit{Higher Algebra} for the reader's convenience in \cref{AB12}.\footnote{For those fluent in the original, let us recall here that our category $\ArDopconv$ is denoted $\opnm{Step}$ in Higher Algebra, and our functor $\Cleave$ is denoted $\Phi$ in Higher Algebra.} There's a natural functor
\[
(s,t):\ArDopconv \to \Delta^\op \times \Delta^\op
\]
sending a map $f:[n] \to [k]$ in $\Delta$ with convex image to the pair
\[
(s,t)(f) := \big( [n], [k] \big).
\]
Evidently, the functors $\Cut$ and $\Cleave$ of \cref{AB12} fit together in a commuting square
\[
\begin{tikzcd}
\ArDopconv \ar[r,"\Cleave"] \ar[d,"{(s,t)}"] & 
\Tens^\otimes \ar[d]
\\
\Delta^\op \times \Delta^\op \ar[r,"\Cut \times id"] &
\Ass \times \Delta^\op.
\end{tikzcd}
\]
The functor $\Cleave_{[1]}$ is obtained from $\Cleave$ by pullback along the inclusion
\[
\Ass \times \{[1]\} \to \Ass \times \Delta^\op.
\]
The functor $\ka$ is given on objects as follows: we define $\ka([n], 1)$ to be the morphism $[n] \to [1]$ in $\Delta$ with convex image which sends all elements $i \in [n]$ to $0 \in [1]$; we define $\ka([n], 0)$ to be the morphism $[n+1] \to [1]$ which sends the elements $0, \dots, n$ to $0$ and $n+1 \mapsto 1$. The assignment at the level of morphisms is evident. 
\end{rem}

We now construct the left action of $ \Map^h_\M (A,A)$ on $\Map^h_\M(B,A)$.

\begin{construction}\label{lb9}
Fix $\M \to \Delta^\op$ a double $\infty$-category and two objects $A,B \in \M_{[0]}$ as in \cref{lbl4}.
Let $A^* \colon \Delta^\op \to \cosk_0 \M$ denote the section of 
\[
\cosk_0 \M \to \Delta^\op
\]
of \cref{lbl5}. 
Combining $A^*$ with the constant section of 
\[
\M_{[0]} \times \Delta^\op \to \Delta^\op
\]
determined by $B$, we obtain a section of 
\[
(\cosk_0 \M) \times_{\Delta^\op}
(\M_{[0]} \times \Delta^\op)
\to \Delta^\op.
\]
Composing with the equivalence of \cref{lb7}, we obtain a section $(A^*,B)$ of 
\[
\Dec \cosk_0\M \to \Delta^\op.
\]
Noting the Cartesian squares (solid arrow diagram)
\[
\begin{tikzcd}
\Dec \cosk_0 \M \ar[r] \ar[d] &
\delta^* \cosk_0 \M \ar[r] \ar[d] &
\cosk_0 \M \ar[d] \\
\Delta^\op \times \set{0}
\ar[r] \ar[rr, bend right=10, "+"'] \ar[u, dashed, bend right, "{(A^*,B)}"'] &
\Delta^\op \times \Delta^1
\ar[r] \ar[u, dashed, bend right, "{A^*B \ssearrow{A^*}}"'] &
\Delta^\op \ar[u, dashed, bend right, "{A^*}"']
\end{tikzcd}
\]
we obtain a section (denoted $A^*B \ssearrow A^*$) of the $\Delta^1$-family of coCartesian fibrations $\delta^* \cosk_0\M$ over $\Delta^\op$ with fibers $(A^*, B)$ over $\Delta^\op\times\{0\}$ (left-hand dotted arrow) and $A^*$ over $\Delta^\op \times \{1\}$ (right-hand dotted arrow = pullback of $A^*B \ssearrow A^*$ along the inclusion $\Delta^\op \times \{1\} \subset \Delta^\op \times \Delta^1 = \nato$). 
We define ${\opnm{LMap}^h_\M(B,A)}^\circledast$ by the pullback diagram in $\Cat_{\mathrm{coCart} / \nato}$ (\cref{lbl2})
\[
\tag{*}
\begin{tikzcd}
\operatorname{LMap}^h_\M (B,A)^\circledast \pbcorner
\ar[d] \ar[r] &
\delta^* \M
\ar[d, "{\delta^*\ka}"]
\\
\nato
\ar[r,"{A^*B \ssearrow A^*}"']
&
\delta^* \cosk_0 \M.
\end{tikzcd}
\]
The associated natural transformation in $\Fun(\Delta^\op, \Cat_\infty)$ is a left-action object of $\Cat_\infty$ in the sense of~\cite[def. 4.2.2.2]{HA}. 

Applying the constructions of \cref{AC13} to the left action object of $\Cat_\infty$ constructed above, we obtain an $\LM$-monoidal category 
\[
{\opnm{LMap}^h_\M(B,A)}^\otimes \to \LM^\otimes
\]
which exhibits $\Map_\M^h(B,A)$ as left-tensored over ${\Map^h_\M(A,A)}^\otimes$.
\end{construction}

\begin{ex}\label{lb10}
Applying \cref{lb9} to the double $\infty$-category ${\Bmod_\A(\C)}^\circledast$, we find that the straightening of ${\delta^*\Bmod_\A(\C)}^\circledast$ corresponds to a diagram of $\infty$-categories as in \cref{lb8}(D),
that $\delta^* {\cosk_0\Bmod_\A(\C)}^\circledast$ corresponds to a diagram of $\infty$-categories as in \cref{lb8} (DC), and so the map $\delta^*\ka$ corresponds to a morphism of diagrams
\[
\mathrm{\cref{lb8}(D)} \to \mathrm{\cref{lb8}(DC)}.
\]
Thus, if $\A$ is a class of algebras in $\Cc$ capable of relative tensor products and $A,B \in \A$
then ${\operatorname{LMap}^h_{\Bmod_\A(\C)^\circledast}(B,A)}^\circledast$ corresponds to a diagram like so:
\[
\begin{tikzcd}
\vdots 
\ar[d, shift left=3]\ar[d, shift left=1] \ar[d, shift right=1] \ar[d, shift right=3]
&
\vdots
\ar[d, shift left=3]\ar[d, shift left=1] \ar[d, shift right=1] \ar[d, shift right=3]
\\
\BMod{A}{A}(\C) \times  \BMod{A}{A}(\C)
\times  \BMod{A}{B}(\C)
\ar[d, shift left=2] \ar[d] \ar[d, shift right=2]
\ar[r]
&
\BMod{A}{A}(\C) \times  \BMod{A}{A}(\C)
\ar[d, shift left=2] \ar[d] \ar[d, shift right=2]
\\
\BMod{A}{A}(\C) \times  \BMod{A}{B}(\C) 
\ar[d, shift left=1] \ar[d, shift right=1]
\ar[r]
&
\BMod{A}{A}(\C)
\ar[d, shift left=1] \ar[d, shift right=1]
\\
\BMod{A}{B}(\C)
\ar[r]
&
\ast
\end{tikzcd}
\]
in which the horizontal maps are projections. 
\end{ex}

\begin{construction}\label{def:BLMod_A}
Let $\C^\otimes \to \Ass$ be a monoidal category and let
$A\in \Alg(\C)$ be an algebra object.
Assume $\C$ admits realizations of $A$-bar constructions $\otimes$-compatibly. 
Let $\A$ denote the class of those algebras which are equivalent to $A$ or to $\one$. 
Then $\A$ is capable of relative tensor products (\cref{capable}), so that by \cref{link1}, $\Bmod_{\A}(\C)^\circledast$ is a double $\infty$-category.
We define  
\[
{\BLMod_A(\C)}^\otimes \to \LM^\otimes
\]
to be the $\LM$-monoidal category ${\opnm{LMap}^h_\M(\one,A)}^\otimes$ of \cref{lb9} applied to $\M = {\Bmod_{\A}(\C)}^\circledast$.
\end{construction}

\begin{rem}\label{lb12}
In view of the equivalence between left $A$-modules and $A$-$\one$-bimodules of ~\cite[cor.4.3.2.8]{HA}, the $\LM$-monoidal category 
\[
{\BLMod_A(\C)}^\otimes \to \LM^\otimes 
\]
of \cref{def:BLMod_A}
exhibits $\LMod_A(\C)$ as left tensored over $ {\BMod{A}{A}(\C)}^\otimes$. Moreover, unwinding definitions, we find that the induced tensor product functor is given by the \emph{relative tensor product}
\begin{align*}
\otimes_A \colon  & \BMod{A}{A}(\C) \times \LMod_A(\C) \to \LMod_A(\C) , \\
& (M,X) \mapsto M \otimes_A X = \left| \Bar_A(M,X)_\bullet \right|.
\end{align*}
\end{rem}

\section{External relative tensor products in \texorpdfstring{$\BM^\otimes$}{BM}-monoidal categories}
\label{sec:External}


Let $\C^\otimes \to \LM^\otimes$ be an $\LM$-monoidal $\infty$-category exhibiting $\C_\m$ as left-tensored over $\C_\a$ and $A\in \Alg(\C_\a)$ an algebra.
Our goal in this section is to construct an $\LM$-monoidal $\infty$-category 
\[
\BLMod_A(\C)^\otimes \to \LM^\otimes
\]
exhibiting $\LMod_A(\C_\m)$ as left-tensored over $\BMod{A}{A}(\C_\a)$; see \cref{w9f2}.
In order to leverage the results of \cite[sec. 4.4.1]{HA} we consider the slightly more general setting of $\BM$-monoidal $\infty$-categories. \Cref{LM_to_BM_monoidal_construction} will allow us to apply our $\BM$-monoidal constructions to $\LM$-monoidal categories. The reader who wishes to move on to the central constructions of Koszul duality and is willing to accept the action of bimodules on left modules on faith, will only need \cref{m24j0} (compatibility of tensor product with $\A$-bar constructions) for the sequel.

\subsection[Promoting LM-monoidal infinity-categories to BM-monoidal infinity-categories]{Promoting $\LM$-monoidal $\infty$-categories to $\BM$-monoidal $\infty$-categories}

The following statement allows us to promote an $\LM$-monoidal $\infty$-category to a $\BM$-monoidal $\infty$-category. 
\begin{notation}
Let $\Oo^\otimes$ be an $\infty$-operad. 
Following~\cite[rem. 2.4.2.6]{HA}, write $\Cat_\infty^\Oo \subset (\operatorname{Op}_\infty)_{/\Oo^\otimes}$ for the subcategory spanned by $\Oo$-monoidal $\infty$-categories and $\Oo$-monoidal functors.
\end{notation}

\begin{rem}
The terminal $\infty$-operad $\Fin_* \xto{=} \Fin_*$ defines a symmetric monoidal structure on the terminal $\infty$-category $\ast \in \Cat_\infty$.
We can therefore consider it as an object of the $\infty$-category of monoidal $\infty$-categories $\ast \in \Cat_\infty^{\operatorname{Ass}}$.
\end{rem}

\begin{prop}\label{LM_to_BM_monoidal}
The forgetful functor induces an equivalence of $\infty$-categories
\[
\Cat_\infty^{\BM} \times_{ \Cat_\infty^{\operatorname{Ass}}} \set{\ast} \iso \Cat_\infty^{\LM}.
\]
\end{prop}
\begin{proof}
Endow $\Cat_\infty$ with the Cartesian monoidal structure of~\cite[prop. 2.4.1.5]{HA}.
The fiber product $ \Cat_\infty^\times \times_{\Fin_*} \BM^\otimes \to \BM^\otimes$ exhibits $\Cat_\infty$ as bitensored over itself, so we can consider $\Cat_\infty$ as a $\BM$-monoidal $\infty$-category.
Applying~\cite[cor.4.3.2.8]{HA} to the $\BM$-monoidal category $\M=\Cat_\infty$ and the trivial algebra $B=\ast \in \Alg(\Cat_\infty)$ shows that the forgetful functor 
\[
\Bmod(\Cat_\infty) \times_{\Alg(\Cat_\infty)} \set{\ast} \to \LMod(\Cat_\infty)
\]
is an equivalence.
The result follows from the equivalences
\[
\Cat_\infty^{\operatorname{Ass}} \simeq \Alg(\Cat_\infty),  \quad
\Cat_\infty^\LM \simeq \LMod(\Cat_\infty), \quad
\Cat_\infty^\BM \simeq \Bmod(\Cat_\infty),
\]
which are all cases of the equivalence $\Alg_\Oo (\Cat_\infty) \simeq \Cat_\infty^\Oo$ of~\cite[rem. 2.4.2.6]{HA}.
\end{proof}

\begin{ex} \label{LM_to_BM_monoidal_construction}
By choosing a homotopy inverse to the equivalence of \cref{LM_to_BM_monoidal} we can define a functor
\[
\Cat_\infty^\LM \iso \Cat_\infty^\BM \times_{\Cat_\infty^{\operatorname{Ass}}} \set{\ast} \to \Cat_\infty^\BM
\]
where the first functor is the chosen homotopy inverse and the second is the projection.
Informally this functor extends an $\LM$-monoidal $\infty$-category $\C$ to a $\BM$-monoidal $\infty$-category by defining a right action of the terminal category $\ast$ on $\C_\m$ in which the action map
\[
\C_\m \times \ast \to \C_\m
\]
is the canonical equivalence. 
\end{ex}

\subsection[The generalized infinity operad Tens_Nabla]{The generalized $\infty$-operad $\Tens^\otimes_\nato$}

Recall that $\nato= \Delta^\op \times \Delta^1$.
In this subsection we construct a $\nato$-family of $\infty$-operads
\[
\Tens^\otimes_\nato \to \nato \times \Fin_*
\]
(\cref{def:Tens_nato}) and a functor
\[
\varphi \colon \Tens^\otimes_\nato \to \BM^\otimes
\]
(\cref{const:Tens_S_to_BM}).
In section \ref{ZZ11}, these will be used to construct a certain coCartesian fibration
\[
{\BLMod_{\A}(\C)}^\circledast \to \nato
\]
associated to a $\BM$-monoidal $\infty$-category $\C$ and a suitable class $\A$ of algebra objects of $\C_-$ under certain assumptions.

We recall notation from \cref{lb6} associated to the d\'ecalage construction.

\begin{notation} \label{def:plus_and_delta}
Let 
$(-)_+ \colon \Delta \to \Delta$
be the functor sending a finite totally ordered set $I$ to 
$I_+=I \cup \set{\infty}$,
where $i< \infty$ for all $i\in I$. The natural inclusion $I \to I_+$ defines a natural transformation $id_{\Delta} \Rightarrow (-)_+$ of functors $\Delta \to \Delta$, or equivalently a natural transformation $(-)_+ \Rightarrow id_{\Delta^\op}$ of functors $\Delta^\op \to \Delta^\op$.
Let 
\[
\delta \colon \nato \to \Delta^\op, \quad (I,0) \mapsto I_+, \,(I,1) \mapsto I
\]
be the associated functor.

We often denote the objects of $\Delta^\op$ by $[k]= \set{0<1<\cdots<k}$.
The isomorphism $[k]_+ \= [k+1]$ identifies $\infty \in [k]_+ = \set{0<1<\cdots <k<\infty}$ with $k+1 \in [k+1]$.

Given $[k]\in \Delta$ we write 
\[
0_+ \colon [k]_+ \to [0]_+ \= [1] , \quad 0_+(i) = \begin{cases}
0 & i \neq \infty \\
1 & i = \infty 
\end{cases}
\]
for the map induced by applying $(-)_+$ to the zero map $[k] \to [0]$.
\end{notation}

\begin{rem}\label{delta_lambda_description}
A morphism in $\nato$ is of one of the forms 
\[ (\lambda, id_0), (\lambda, id_1), (\lambda, 0 \to 1) \]
where $\lambda \colon [k'] \to [k] \in \Delta$ is a map of linearly ordered nonempty finite sets.
The functor $\delta$ sends $([k],0)$ to $[k]_+ \= [k+1]$, sends $([k],1)$ to $[k]$, and acts on morphisms as follows.
\begin{enumerate}
\item $\delta$ sends $(\lambda,id_0)$ to $\lambda \colon [k'] \to [k] \in \Delta$.
\item $\delta$ sends $(\lambda,id_1)$ to 
\[
\lambda_+ \colon [k']_+ \to [k]_+ \in \Delta, 
\quad \lambda_+ (i) = \begin{cases}
\lambda(i) & i\neq \infty \\
\infty & i= \infty.
\end{cases}
\]
\item $\delta$ sends $(\lambda,0 \to 1)$ to 
\(
\lambda \colon [k'] \to [k] \subset [k]_+ \in \Delta.
\)
\end{enumerate}
\end{rem}

\begin{rem} \label{tau_ext_rel_tensor}
The category $\nato$ may be depicted as follows
\[
\tag{$\nato$}
\begin{tikzcd}
\vdots 
\ar[d, shift left=3, dashed]\ar[d, shift left=1, dashed] \ar[d, shift right=1, dashed] \ar[d, shift right=3, dashed]
&
\vdots
\ar[d, shift left=3]\ar[d, shift left=1] \ar[d, shift right=1] \ar[d, shift right=3]
\\
([2],0)
\ar[d, shift left=2, dashed] \ar[d, dashed] \ar[d, shift right=2, dashed]
\ar[r, rightsquigarrow]
&
([2],1)
\ar[d, shift left=2] \ar[d] \ar[d, shift right=2]
\\
([1],0)
\ar[d, shift left=1, dashed, "\tau"] \ar[d, shift right=1, dashed]
\ar[r, rightsquigarrow]
&
([1],1)
\ar[d, shift left=1] \ar[d, shift right=1]
\\
([0],0)
\ar[r, rightsquigarrow]
&
([0],1).
\end{tikzcd}
\]
The functor $\delta: \nato \to \Delta^\op$ restricts to the identity on the right column, identifies the left column with the full subcategory of $\Delta^\op$ corresponding to the totally ordered sets $[k]$, $k\ge1$, and morphisms in $\Delta$ which preserve maximal elements (dashed arrows below)
\[
\tag{$\Delta^\op$}
\begin{tikzcd}
\vdots 
\ar[d, shift left=3, dashed]\ar[d, shift left=1, dashed] \ar[d, shift right=1, dashed] \ar[d, shift right=3, rightsquigarrow]
\\
\left[ 2\right]
\ar[d, shift left=2, dashed] \ar[d, dashed] \ar[d, shift right=2, rightsquigarrow]
\\
\left[1\right]
\ar[d, shift left=1] \ar[d, shift right=1, rightsquigarrow]
\\
\left[0\right],
\end{tikzcd}
\]
and maps the horizontal arrows of $\nato$ to those face maps which do not preserve maximal elements (squiggly arrows).

The arrow $\tau$ in $\nato$, which corresponds to the morphisms
\[
\{0\} \to \{0,1\} \quad 0 \mapsto 0
\quad
\mbox{in}
\quad \Delta,
\quad
0 \xto{id} 0
\quad
\mbox{in}
\quad
\Delta^1,
\]
will induce the \emph{external relative tensor product}
(see \cref{def:external_relative_tensor_prod})
\[
\tau_* \colon \Bmod(\C_-) \times_{\Alg(\C_-)} \Bmod(\C_{\m}) \to \Bmod(\C_{\m})
\]
and hence deserves special attention. 
\end{rem}

\begin{notation} \label{def:Tens_nato}
We denote by
\[
\Tens^\otimes_\nato \to \nato
\]
the pullback of $\Tens^\otimes \to \Delta^\op$ along $\delta$ (\cref{def:plus_and_delta}). By construction, we can identify $\Tens^\otimes_\nato|_{\Delta^\op \times \set{1} }$ with $\Tens^\otimes \to \Delta^\op$,
and $\Tens^\otimes_\nato|_{\Delta^\op \times \set{0} }$ with the pullback of $\Tens^\otimes \to \Delta^\op$ along $(-)_+ \colon \Delta^\op \to \Delta^\op$.
\end{notation}

\begin{rem} \label{Tens_nato_description}
We can unwind the definition of the category $\Tens^\otimes_\nato$ as follows.
\begin{enumerate}
\item An object of $\Tens^\otimes_\nato$ is a tuple $(b,\angled{n},[k],c_-,c_+)$ where $b\in \Delta^1$, $\angled{n}\in \Fin_*$, $[k]\in \Delta^\op$ and
\[
c_-, c_+ \colon \angled{n}^\circ \to \delta([k],b) = \begin{cases} [k]_+ & b=0 \\ [k] & b=1 \end{cases}
\]
are maps
satisfying $c_-(i) \leq c_+(i) \leq c_-(i)+1$ for every $i \in \angled{n}^\circ$.
\item Let $(b,\angled{n},[k],c_-,c_+)$ and $(b',\angled{n'},[k'],c'_-,c'_+)$ be two objects of $\Tens^\otimes_\nato$.
A morphism from $(b,\angled{n},[k],c_-,c_+)$ to $(b',\angled{n'},[k'],c'_-,c'_+)$ is a tuple $(\beta,\al,\la)$ where 
\[
\beta \colon b \to b' \in \Delta^1, 
\quad \al \colon \angled{n} \to \angled{n'} \in \Ass,
\quad \la \colon [k'] \to [k] \in \Delta
\]
satisfy the following condition.
For every $j\in \angled{n'}^\circ$ with totally ordered fiber $\al^{-1}(j)= \set{i_0 \prec i_1 \prec \cdots \prec i_m}$ the following equalities hold:
\begin{align*}
& c_-(i_0) = \delta(\la,\beta)(c'_-(j)), \quad c_+(i_m) = \delta(\la,\beta)(c'_+(j)), \\
& c_-(i_0) \leq c_+(i_0) = c_-(i_1) \leq c_+(i_1) = \cdots = c_-(i_m) \leq c_+(i_m),
\end{align*}
where in the first two equalities we consider $\delta(\la,\beta)$ as defining a morphism of totally ordered sets $\delta([k'],b') \to \delta([k],b)$.
\end{enumerate}
\end{rem}

Since $\Tens^\otimes_\nato$ and $\BM^\otimes$ are nerves of $1$-categories we will be able to define a functor $\Tens^\otimes_\nato \to \BM^\otimes$ by its action on objects and morphisms. In its construction we use a certain natural transformation $\eta$ from $\delta$ to the constant functor $\nato \to \Delta^\op$ with value $[1]$. 

\begin{construction} \label{def:delta_to_[1]}
Let $\delta \colon \nato \to \Delta^\op$ be the map of \cref{def:plus_and_delta}.
The totally ordered set $\delta([k],b)$ admits a canonical map to $[1]= \set{0 <1}$, constructed as follows.
Note that $\delta$ carries the initial object $([0],0) \in \Delta^\op \times \Delta^1$ to $[0]_+ \= [1]$, and hence induces a map on the slice categories
\[
\Delta^\op \times \Delta^1 \= (\Delta^\op \times \Delta^1)_{([0],0)/} \xto{\delta} (\Delta^\op)_{[1]/} \= (\Delta_{/[1]})^\op.
\]
Given $([k], b) \in \nato$ on the left, we define
\[
\eta:\delta([k], b) \to [1]
\]  
to be the image of $([k], b)$ in $(\Delta_{/[1]})^\op$ on the right. Explicitly the map $\eta$ is given by
\begin{align*}
\eta \colon & \delta([k],0) = [k]_+ \to [1] , 
\quad \eta(i) = \begin{cases}
0 & i \neq \infty \\
1 & i = \infty
\end{cases} \\
\eta \colon & \delta([k],1) = [k] \to [1] , \quad \eta(i)=0.
\end{align*}
\end{construction}

\begin{construction} \label{const:Tens_S_to_BM}
We define a functor  
\[ 
\varphi \colon \Tens^\otimes_\nato \to \BM^\otimes 
\] 
by its action on objects and morphisms.
\begin{enumerate}
\item 
Given $(b,\angled{n},[k],c_-,c_+) \in \Tens^\otimes_\nato$ we define
\[
\varphi(b,\angled{n},[k],c_-,c_+):=(\angled{n}, \eta \circ c_-, \eta \circ c_+),
\] 
where the maps 
\[
\eta \circ c_- , \eta \circ c_+ \colon \angled{n}^\circ \to \delta([k],b) \xto{\eta} [1]
\]
are defined by composition with the canonical map $\eta$ of \cref{def:delta_to_[1]}.
\item
Given 
\[
(\beta,\al,\la) \colon (b,\angled{n},[k],c_-,c_+) \to (b',\angled{n'},[k'],c'_-,c'_+)
\]
a morphism in $\Tens^\otimes_\nato$, we set 
\[
\varphi(\beta,\al,\la):=\al \colon \angled{n} \to \angled{n'}.
\]
Let us verify that $\al$ indeed defines a morphism in $\BM^\otimes$. By the description of morphisms in \cref{Tens_nato_description} we have 
\[
c_-(i_0) = \delta(\la,\beta)(c'_-(j)), \quad c_+(i_m) = \delta(\la,\beta)(c'_+(j)),
\]
so by naturality of the canonical map $\eta$ we have 
\[
\eta \circ c_-(i_0) = \eta \circ c'_-(j), \quad \eta \circ c_+(i_m) = \eta \circ c'_+(j).
\]
By the description of the object $(b,\angled{n},c_-,c_+) \in \Tens^\otimes_\nato$ we also have
\[
c_-(i_0) \leq c_+(i_0) = c_-(i_1) \leq c_+(i_1) = \cdots = c_-(i_m) \leq c_+(i_m).
\]
The remaining conditions follow by applying the order preserving map
\[\eta \colon \delta([k],b) \to [1].\]
\end{enumerate}

\end{construction}

\subsection{Construction of $\BLMod_\A(\C)^\circledast$}
\label{ZZ11}

Let $\Cc^\otimes$ be a $\BM$-monoidal $\infty$-category which admits realizations of $\Aa$-bar constructions $\otimes$-compatibly for an appropriate class $\Aa$ of algebras in $\Cc_-$. Our goal in this subsection is to construct a certain coCartesian fibration
\[
p_\Aa: \BLMod_\Aa(\C)^\circledast
\to
\nato
\]
which will play a role similar to the role played by the Morita double $\infty$-category of \cref{lb5.5}. In \cref{extracting} we will fix algebras $A$ in $\Cc_-$ and $B\in \C_+$ and use $\BLMod_\Aa(\C)^\circledast$ to extract a left action of $A$-$A$-bimodules in $\Cc_-$ on $A$-$B$-bimodules in $\Cc_\m$. 

Throughout this subsection, we fix a $\BM$-monoidal category
\[
q \colon \Cc^\otimes \to \BM^\otimes.
\]

\begin{notation}
For any map of simplicial sets $f:K \to \nato$ we let
\[
\Tens^\otimes_f \to K
\]
denote the pullback of $\Tens^\otimes_\nato \to \nato$ along $f$. Equivalently, $\Tens^\otimes_f$ is the pullback of $\Tens^\otimes \to \Delta^\op$ along the composite
\[
K \to \nato \xto{\delta} \Delta^\op.
\]
We allow ourselves to use the notation 
\[
\Tens^\otimes_f =   \Tens^\otimes_K
\]
when we see no danger of confusion. 
We consider $\Tens^\otimes_K$ as a simplicial set over $\BM^\otimes$ via the map
\[
\Tens^\otimes_K = \Tens^\otimes_f \to \Tens^\otimes_\nato \xto{\varphi} \BM^\otimes
\]
where $\varphi$ is the functor of \cref{const:Tens_S_to_BM}.
\end{notation}

\begin{construction} \label{cons:BLMod_C_star}
In terms of our fixed $\BM$-monoidal category $q \colon \C^\otimes \to \BM^\otimes$ we define a map of simplicial sets
\[
p \colon \BLMod(\C)^\circledast \to \nato
\]
by the following universal property:
\begin{itemize}
\item[(*)] For every map of simplicial sets $K \to \nato$ there is a canonical bijection\footnote{The subscript `0' on the right refers to the \textit{set} of 0-simplices.}
\[
\Hom_{{\Set_\Delta}_{/\nato}}(K,\BLMod(\C)^\circledast) \= \Alg_{\Tens_K/\BM}(\C)_0.
\]
\end{itemize}
\end{construction}
In particular, a diagram of simplicial sets
\begin{equation*}
\begin{tikzcd}
L \ar[d] \ar[r] & {\BLMod(\C)}^\circledast \ar[d,"p"] \\
K \ar[r] \ar[ur] & \nato
\end{tikzcd}
\end{equation*}
commutes if and only if the corresponding diagram 
\begin{equation*}
\begin{tikzcd}
\Tens_L^\otimes \ar[d] \ar[r] & \C^\otimes \ar[d,"q"] \\
\Tens_K^\otimes \ar[r] \ar[ur] & \BM^\otimes
\end{tikzcd}
\end{equation*}
commutes. 

\begin{lem} \label{BTen_flat}
The 
map of simplicial sets of \cref{cons:BLMod_C_star}
\[
p \colon {\BLMod(\C)}^\circledast \to \nato
\]
is an inner fibration.
\end{lem}
\begin{proof}
We verify that $p$ has the right lifting property with respect to all inner anodyne maps.
By definition of ${\BLMod(\C)}^\circledast$ checking the lifting condition against an inner anodyne map $A \to B$ is equivalent to solving the following lifting problem
\begin{equation*}
\begin{tikzcd}
\Tens^\otimes_A \ar[d] \ar[r] & \C^\otimes \ar[d,"q"] \\
\Tens^\otimes_B \ar[r] \ar[ur, dashed] & \BM^\otimes. 
\end{tikzcd}
\end{equation*}
By~\cite[thm. 4.4.3.1]{HA} the functor $\Tens^\otimes \to \Delta^\op$ is a flat categorical fibration, and by~\cite[rem. B.3.12]{HA} so is its pullback $\Tens^\otimes_\nato \to \nato$.
By~\cite[prop. B.3.14]{HA} the monomorphism $\Tens^\otimes_A \to \Tens^\otimes_B$ is a trivial cofibration in the Joyal model structure (\cite[thm. 2.2.5.1]{HTT}).
It follows that the solution to the above lifting problem exists, as the right vertical morphism $\C^\otimes \to \BM^\otimes$ is a categorical fibration.
\end{proof}

Applying~\cite[cor. 4.4.3.2]{HA} to $\Oo^\otimes=\BM^\otimes$ and 
\(
S=\nato \xto{\varphi} \BM^\otimes
\)
we get the following proposition.

\begin{prop} \label{coCart_edges_in_BLMod}
Let $\al \colon s \to s'$ be an edge in $\nato=\Delta^\op \times \Delta^1$ and let $F_0$ be a vertex of $\Alg_{\Tens_s/\BM}(\C)$.
The vertex $F_0$ corresponds to an element of 
\[
\Hom_{{\Set_\Delta}_{/\nato}} (\set{s}, {\BLMod(\C)}^\circledast),
\]
and hence to a vertex of ${\BLMod(\C)}^\circledast$ lying over $s$.
Let $F$ be a vertex of 
$\Alg_{\Tens_\al/\BM}(\C)$
extending $F_0$; equivalently, $F$ is an edge of ${\BLMod(\C)}^\circledast$ lying over $\al$ with source $F_0$.
If $F$, regarded as a map of generalized $\infty$-operads
\[
\Tens^\otimes_\al \to \C^\otimes
\]
over $\BM^\otimes$,
is an operadic $q$-left Kan extension of $F_0$, then $F$ regarded as an edge of 
$
{\BLMod(\C)}^\circledast
$
is $q$-coCartesian. 
\end{prop}

\begin{proof}
Mutatis mutandis the same as \cref{2}.
\end{proof}

Our next goal is to introduce the notion of an $\A$-multimodule in a $\BM$-monoidal $\infty$-category (\cref{k1_A_multimodule}).

\begin{notation}
For $[k]\in \Delta^\op$ and $i=0,\ldots,k$ let 
\[
\epsilon_i \colon \Ass \to \Tens^\otimes_{[k]}
\]
be the map of~\cite[rem. 4.4.1.8]{HA}.
The map $\epsilon_i$ identifies $\Ass$ with the full subcategory of $\Tens^\otimes_{[k]}$ spanned by objects $(\angled{n},[k],c_-,c_+)$ where $c_-(j)=c_+(j)=i$ for every $j\in \angled{n}^\circ$.

Let $\mathrm{Ass}^\otimes_- \subset \BM^\otimes$ be the full subcategory of objects of the form $(\angled{n}, c_-,c_+)$ where $c_-(i)=c_+(i) = 0$ for every $i\in \angled{n}^\circ$.
The forgetful functor $\BM^\otimes \to \Ass$ restricts to an isomorphism $\mathrm{Ass}^\otimes_- \xto{\=} \Ass$, as in~\cite[rem. 4.3.1.10]{HA}.
\end{notation}

\begin{mydef}
\label{def:natomm}
Fix $s\in \nato = \Delta^\op \times \Delta^1$. An \emph{$s$-multimodule in $\C$} is a map of $\infty$-operads $\Tens^\otimes_s \to \C^\otimes$ over $\BM^\otimes$.
A \emph{multimodule in $\C$} is an $s$-multimodule in $\C$ for some $s\in \nato$.
\end{mydef}

\begin{rem}
\label{s23j1}
The $([k], i)$-multimodules of \cref{def:natomm} come in two varieties, depending on whether $i=0$ or $1$. Note that we have canonical equivalences
\[
\Tens^\otimes_{([k], 1)} \simeq 
\Tens^\otimes_{[k]}
\quad
\mbox{and}
\quad
\Tens^\otimes_{([k], 0)} \simeq 
\Tens^\otimes_{[k]_+}
\simeq
\Tens^\otimes_{[k+1]}.
\]
Starting with the case $i =1$, if $s=([k],1) \in \nato$ then
\[
\Tens^\otimes_s \to \Tens^\otimes_\nato \xto{\varphi} \BM^\otimes
\]
factors through $\mathrm{Ass}^\otimes_- \subset \BM^\otimes$. Therefore we can identify a $([k],1)$-multimodule in $\C$ with a $[k]$-multimodule in $\C_-$. 

Turning to the case $i=0$, a $([k], 0)$-multimodule 
\[
M: \Tens^\otimes_{([k],0)} 
\to \Cc^\otimes
\]
in $\Cc$ corresponds to a $[k]$-multimodule
\[
\tag{*}
_{A_0}{M^{0,1}}_{A_{1}} M^{1,2}_{A_2}
\cdots
_{A_{k-1}} M^{k-1,k}_{A_k}
\]
in $\Cc_-$ and an $A_k$-$A_\infty$-bimodule $M^{k, \infty}$ in $\C_\m$, where $A_\infty$ is an algebra in $\C_+$. This correspondence is furnished as follows. The inclusion
\[
\al: [k] \subset [k]_+ = \{0 < \dots < k < \infty\}
\]
gives rise to a map of $\infty$-operads \cite[not. 4.4.1.10]{HA}
\[
v_\al: \Tens^\otimes_{[k]} \to 
\Tens^\otimes_{[k]_+}
\]
given by
\[
v_\al \big(
\langle n \rangle, [k], c_-, c_+
\big) 
= 
\big(
\langle n \rangle, [k]_+, \al \circ c_-, \al \circ c_+
\big).
\]
The composition
\[
\Tens^\otimes_{[k]} 
\xto{v_\al} 
\Tens^\otimes_{[k]_+} 
\simeq
\Tens^\otimes_{([k],0)} 
\to \Tens^\otimes_\nato 
\xto{\varphi} 
\BM^\otimes
\]
factors through $\mathrm{Ass}^\otimes_- \subset \BM^\otimes$.
Therefore, we can identify the composition
\[
\Tens^\otimes_{[k]} 
\xto{v_\al} 
\Tens^\otimes_{([k],0)} 
\xto{M} \C^\otimes
\]
with a $[k]$-multimodule in $\C^\otimes_-$ as in (*). On the other hand, the inclusion $\beta: [1] \simeq \{k < \infty\} \subset [k]_+$ gives rise to a map of $\infty$-operads
\[
v_\be: \BM^\otimes \to \Tens^\otimes_{[k]_+}
\]
whose composition with $\phi$ is the identity map of $\BM^\otimes$. Thus, the composition
\[
\BM^\otimes
\xto{v_\be}
\Tens^\otimes_{([k],0)}
\xto{M}
\Cc^\otimes
\]
may be identified with a bimodule in $\Cc_\m$ over algebras $A'$ in $\C_-$ and $A_\infty$ in $\C_+$. A straightforward verification shows that $A' = A_k$.

\end{rem}

\begin{mydef}\label{k1_A_multimodule}
Let $\A \subseteq \Alg(\C_-)$ be a full subcategory, closed under equivalences. We say that a $([k],1)$-multimodule in $\C$ is an \emph{$\A$-multimodule} if the corresponding $[k]$-multimodule in $\C_-$ is an $\A$-multimodule.
\label{k0_A_multimodule}
We say that a $([k],0)$-multimodule $F\in \Alg_{\Tens_{([k],0)}/\BM}(\C)$ 
is an \emph{$\A$-multimodule} if the corresponding $[k]$-multimodule in $\C_-$,
\[
\Tens^\otimes_{[k]} 
\xto{v_\al} \Tens^\otimes_{([k],0)} \xto{F} \C^\otimes,
\]
is an $\A$-multimodule in $\C_-$.
\end{mydef}

\begin{mydef}
\label{aolkefamm}
Let $\A \subseteq \Alg(\C_-)$ be a full subcategory, stable under equivalences.
We say that \emph{$\C$ admits operadic left Kan extensions from $\A$-multimodules} 
if for every edge $\al \colon s \to s'$ of $\nato$ and $\A$-multimodule $F_0 \in \Alg_{\Tens_s/\BM}(\C)$ 
there exists a $q$-operadic left Kan extension $F \in \Alg_{\Tens_\al/\BM}(\C)$ of $F_0$ as in the following diagram
\begin{equation*}
\begin{tikzcd}
\Tens^\otimes_s \ar[d] \ar[r,"F_0"] & \C^\otimes \ar[d,"q"] \\
\Tens^\otimes_\al \ar[r] \ar[ur,dashed,"F"] & \BM^\otimes.
\end{tikzcd}
\end{equation*}
We also say that \emph{$\C$ admits relative tensor products of $\A$-multimodules}, generalizing~\cite[def. 4.4.2.3]{HA}.
\end{mydef}

In the terminology of \cref{coCart_edges_in_BLMod}, we immediately deduce the following statement from \cref{coCart_edges_in_BLMod}.

\begin{cor}
\label{acclfamm}
Let 
\[
p \colon \BLMod(\C)^\circledast \to \nato
\]
be the map of \cref{cons:BLMod_C_star}. If $\C$ admits operadic left Kan extensions from left $\A$-multimodules, then $\BLMod(\C)^\circledast \to \nato$ admits coCartesian lifts from left $\A$-multimodules:
for any edge
\[
\al \colon s \to s'
\]
of $\Delta^\op \times \Delta^1$ and $\A$-multimodule $F_0 \in \Alg_{\Tens_s/\BM}(\C)$ over $\al$, there exists a lift $F$ of $\al$ extending $F_0$.
\end{cor}

\begin{mydef}
Let us temporarily use the alternative notation $\C_0 = \C_-$ and $\C_1 = \C_+$ for the two monoidal categories associated to our fixed $\BM$-monoidal category $\Cc$. If $A_{i_0}$, $A_{i_1}$ are algebras in $\Cc_{i_0}$, $\Cc_{i_1}$ with 
$
0 \le i_0 \le i_1 \le 1
$, 
then we may speak unambiguously of \emph{
$A_{i_0}$-$A_{i_1}$-bimodules in 
$\Cc$
}; we denote the 
$\infty$-category of such by
$_{A_{i_0}}\BMod{}{}_{A_{i_1}}(\Cc)$
or simply by  $_{A_{i_0}}\BMod{}{}_{A_{i_1}}$. 
More generally, given algebras
\[
A_{i_0}, A_{i_1}, \ldots, A_{i_n},  
\quad
\mbox{with}
\quad
0 \le i_0 \le i_1 \le \cdots \le i_n \le 1,
\]
we refer to a tuple $(M_{01}, M_{12}, \ldots, M_{n-1,n})$ where $M_{k,k+1}$ is an $A_{i_k}$-$A_{i_{k+1}}$-bimodule as an \emph{$(A_{i_0}, A_{i_1}, \ldots, A_{i_n})$-module}. 
\end{mydef}

\begin{ex} \label{s23j2}
By \cref{s23j1}, a $([2],1)$-module consists of algebras $A$, $B$, $C$ in $\Cc_-$, an $A$-$B$-bimodule $M$ and a $B$-$C$-bimodule $N$, hence, in our terminology and notation from~\cref{LM_to_BM_monoidal_construction}, an $(A,B,C)$-module $_A M_B N_C$ in $\Cc_-$. 
On the other hand, a $([1],0)$-module consists of algebras $A$, $B$ in $\Cc_-$, an algebra $C$ in $\Cc_+$, an $A$-$B$-bimodule $M$ in $\Cc_-$, and a $B$-$C$-bimodule $N$ in $\Cc_\mM$. 
For fixed $A$, $B$, $C$, we refer to such an object again as an $(A,B,C)$-module and we again use the notation $_A M_B N_C$. 
\end{ex}

\begin{notation} 
\label{phi_001_[2]_algebra}
Let $\tau: \Delta^1 \to \nato$ denote the functor associated to the morphism
\[
([1],0)\to ([0],0)
\]
denoted by the same symbol $\tau$ above (\cref{tau_ext_rel_tensor}). 
Note that the composition $\Delta^1 \xto{\tau} \nato \xto{\delta} \Delta^\op$ corresponds to the morphism 
$[1] \to [2]$ in $\Delta$ used to define $\Tens_{\succ}$ in~\cite[not. 4.4.2.1]{HA}.
Let $\tau$ also denote the induced morphism
\[
\Tens^\otimes_\succ \to \Tens^\otimes_\nato.
\]
Let $\phi_\succ^{001}$ and $\phi_{[2]}^{001}$ denote the composites
\[
\phi_\succ^{001} \colon 
\Tens^\otimes_\succ \xto{\tau}
\Tens^\otimes_\nato
\xto{\varphi}
\BM^\otimes, \quad
\phi_{[2]}^{001} \colon 
\Tens^\otimes_{[2]}
\to
\Tens^\otimes_\succ \xto{\tau}
\Tens^\otimes_\nato
\xto{\varphi}
\BM^\otimes.
\]
We define a \emph{$\phi_{[2]}^{001}$-algebra}
to be a morphism of $\infty$-operads
\[
F: \Tens^\otimes_{[2]} \to \Cc^\otimes
\]
over $\BM^\otimes$, or equivalently a $([1],1)$-module, i.e.
an $(A,B,C)$-module where $A, B$ are algebras in $\C_-$ and $C$ is an algebra in $\C_+$. 
\end{notation}

\begin{rem}
Similarly to \cref{phi_001_[2]_algebra}, we may define a \emph{$\phi_\succ^{001}$-algebra} to be a morphism of $\infty$-operads $\Tens^\otimes_\succ \to \Cc^\otimes$ over $\BM^\otimes$. By an analysis similar to \cref{LabelChaos}, a $\phi_\succ^{001}$-algebra includes the data of
\begin{enumerate}
\item
algebras and bimodules 
\[
A^M B^N C, \quad \text{and} \quad {A'} ^{L} C'
\]
where $A$, $B$, $A'$ are algebras in $\Cc_-$, $C$, $C'$ are algebras in $\Cc_+$, $M$ is an $A$-$B$-bimodule in $\Cc_-$, $N$ is a $B$-$C$-bimodule in $\C_\m$, and $L$ is an $A'$-$C'$-bimodule in $\C_\m$, 
\item
homomorphisms of algebras $f: A \to A'$ and $g: C \to C'$, and
\item 
a morphism $M \otimes N \to L$ in $\C_\m$ which is bilinear up to coherent higher homotopies.
\end{enumerate}
\end{rem}

\begin{mydef} \label{def:ext_bar_const}
Let $F \colon \Tens^\otimes_{[2]} \to \Cc^\otimes$  be a $\phi^{001}_{[2]}$-algebra in $\Cc^\otimes$ with associated algebras and bimodules
\[
_AM _BN _C.
\]
We define the \emph{operadic bar construction $\Bar^\otimes(F)_\bullet = \Bar_B^\otimes(M,N)_\bullet$} to be the composite functor
\[
\Delta^\op \xto{\rm{bar}} \Tens^\otimes_{[2]} 
\xto{F} \Cc^\otimes.
\]
The associated \emph{underlying bar construction}
\[
\Bar(F)_\bullet = \Bar_B(M,N)_\bullet:
\Delta^\op \to \Cc_\mM
\]
and the natural transformation 
\[
\be: \Bar^\otimes(F)_\bullet \to \Bar(F)_\bullet
\]
are then constructed in the same way as for a monoidal category; see \cref{def:bar_const}.
\end{mydef}

\begin{mydef}
\label{m24j0}
Let $A \in \Alg(\Cc_-)$ be an algebra. 
We say that $\Cc^\otimes \to \BM^\otimes$ \emph{admits realizations of $A$-bar constructions $\otimes$-compatibly} if (in the terminology of \cref{s23j2}) for any $(B,A,C)$-module $F: \Tens^\otimes_{[2]} \to \Cc^\otimes$ with $B \in \Alg(\Cc_-)$ and $C \in \Alg(\Cc_i)$ ($i = -$ or $+$) the following conditions hold:
\begin{enumerate}
\item
The appropriate $\infty$-category ($\Cc_-$ if $i = -$ and $\Cc_\mM$ if $i = +$) admits realizations of the bar constructions $\Bar_A(M,N)_\bullet$. 
\item For every $X \in \C_-$ and every  $Y \in \C_i$, the canonical map
\[
\left| X \otimes \Bar_A( M,N )_\bullet \otimes Y \right| \to X \otimes \left| \Bar_A(M,N)_\bullet \right| \otimes Y
\]
is an equivalence.
\end{enumerate}
Let $\Aa$ be a class of algebras in $\Cc_-$.
We say that \emph{$\Cc$ admits realizations of $\Aa$-bar constructions $\otimes$-compatibly} 
if $\Cc$ admits realizations of $A$-bar constructions $\otimes$-compatibly for every $A\in \Aa$. 

Dually, let $A \in \coAlg(\Cc_-) = \Alg(\Cc_-^\op)^\op$ be a coalgebra. 
We say that \emph{$\Cc$ admits totalizations of $A$-cobar constructions $\otimes$-compatibly} if $\Cc^\op$ admits realizations of $A$-bar constructions $\otimes$-compatibly, and similarly for ``\emph{admits totalizations of $\Aa$-cobar constructions $\otimes$-compatibly}''.
\end{mydef}

\begin{prop}(Generalization of \cref{29.2} and \cite[prop. 4.4.2.8(1)]{HA}.)
\label{s23j11}
Let $F_0 = {_AM_BN_C}$ be an $(A,B,C)$-module in $\Cc$ with $A,B \in \Alg \Cc_-$ and $C \in \Alg \Cc_+$.
Assume the underlying bar construction ${\Bar_B(M,N)}_\bullet$ admits a geometric realization compatibly with $\otimes$.
Suppose given a commutative diagram of generalized $\infty$-operads (solid arrow diagram below)
\begin{equation*}
\begin{tikzcd}
\Tens^\otimes_{[2]} \ar[r,"F_0"] \ar[d] & \C^\otimes \ar[d,"q"] \\
\Tens^\otimes_\succ \ar[r,"\phi^{001}_\succ"] \ar[ur,dashed,"F"] & \BM^\otimes.
\end{tikzcd}
\end{equation*}
Then there exists an operadic $q$-left Kan extension $F$ of $F_0$ (dotted arrow) making the diagram commute.
\end{prop}
\begin{proof}
In view of the generality in which \cite[prop. 4.4.2.5]{HA} is stated, 
this is mutatis mutandis the same as \cref{29.2}.
\end{proof}
\begin{mydef} \label{def:external_relative_tensor_prod}
In the situation of \cref{s23j11} we say that $F$ exhibits $F_1 := F|_{\Tens^\otimes_{[1]}}$ as an \emph{external relative tensor product of $M$ and $N$ over $A$}.
\end{mydef}

\begin{mydef}
\label{NN11}
We define the \emph{inert subcategory} $\nato_\mathrm{inert}$ of $\nato$ to be the faithful subcategory with the same objects, and morphisms given in the notation of \cref{delta_lambda_description} by
\begin{enumerate}
\item
$(\la, id_1)$ with $\la$ convex, 
\item
$(\la, id_0)$ where the image of $\la$ is convex and includes the maximal element, 
\item
$(\la, 0\to 1)$ with $\la$ convex.
\end{enumerate}
\end{mydef}

\begin{prop}[Segal conditions for ${\BLMod(\C)}^\circledast$] \label{w26j3}
Let 
\[
p \colon B^\circledast = {\Bmod(\C)}^\circledast \to \nato
\]
be the map of simplicial sets of \cref{cons:BLMod_C_star}. 
Then the pullback $p_\text{inert}$ of $p$ to $\nato_\text{inert}$ as in the following diagram
\begin{equation*}
\begin{tikzcd}
B^\circledast_\text{inert} \ar[r] \ar[d,"p_\text{inert}"] & B^\circledast \ar[d,"p"] \\
\nato_\text{inert} \ar[r] & \nato
\end{tikzcd}
\end{equation*}
is a coCartesian fibration. 
Moreover, for every $n>1$, the induced maps
\[
B^\circledast_{( [n] , 1)}  \to
B^\circledast_{( [1] , 1)} \times_{B^\circledast_{( [0], 1) } } 
\cdots \times_{B^\circledast_{ ( [0] , 1)  } } B^\circledast_{ ( [1], 1) }
\simeq
\Bmod(\Cc_-) \times_{\Alg (\Cc_-)} 
\cdots 
\times_{\Alg (\Cc_-)} 
\Bmod(\Cc_-)
\]
and
\begin{align*}
B^\circledast_{([n],0)} 
\to
B^\circledast_{([n],1)} 
&
\times_{B^\circledast_{([0],0)}} 
B^\circledast_{([0],0)} 
\\
& \simeq
\Bmod(\Cc_-) \times_{\Alg (\Cc_-)} 
\cdots 
\times_{\Alg (\Cc_-)} 
\Bmod(\Cc_-)
\times_{\Alg (\Cc_-)} 
\Bmod(\Cc_\m)
\end{align*}
are equivalences of $\infty$-categories.
\end{prop}
\begin{proof}
Let $\al: \Delta^1 \to \nato$ denote the map of simplicial sets associated to the arrow $([k], 0\to1)$ in $\nato$, and consider a square of generalized $\infty$-operads
(solid arrow diagram)
\begin{equation*}
\begin{tikzcd}
\Tens^\otimes_{( [k ],0 ) } \ar[r,"X_0"] \ar[d] & \C^\otimes \ar[d] \\
\Tens^\otimes_\al \ar[r] \ar[ur,dashed,"X"] & \BM^\otimes.
\end{tikzcd}
\end{equation*}
An argument similar to \cite[4.4.3.5]{HA} shows that there exists an operadic $q$-left Kan extension as in the diagram. The remainder of the proof is similar to \ref{05191}.
\end{proof}

\begin{prop}\label{m24j1}
Let $\A$ be a class of algebra objects of $\C_-$ closed under equivalences. Assume that $\C$ admits realizations of $\A$-bar construction $\otimes$-compatibly (\cref{m24j0}). 
Let 
\[
p \colon {\BLMod(\C)}^\circledast \to \nato
\]
be the map of \cref{cons:BLMod_C_star}.
Then ${\BLMod(\C)}^\circledast$ admits coCartesian lifts from $\A$-multimodules (\cref{acclfamm}).
\end{prop}

\begin{proof}
In view of \cref{w26j3}, this is similar to \cref{27.3}.
\end{proof}

\begin{thm} \label{w26j1}
In the situation and the notation of \cref{m24j1}, let
\[
{\BLMod_{\A}(\C)}^\circledast \subset {\BLMod(\C)}^\circledast
\]
be the full subcategory supported on $\A$-multimodules.
Assume that $\C$ admits realizations of $\A$-bar constructions $\otimes$-compatibly.
Then the restriction of $p$ to
\[
p_\Aa: {\BLMod_\A(\C)}^\circledast \to \nato
\]
is a coCartesian fibration. 
Moreover, $p_\Aa$ obeys the Segal conditions of \cref{w26j3}.
\end{thm}
\begin{proof}
Mutatis mutandis the same as \cref{link4}.
\end{proof}

In the situation of \cref{w26j1} with $\A =\Alg(\C_-)$ for simplicity, the coCartesian fibration
\[
p \colon \BLMod(\Cc)^\circledast \to \nato
\]
corresponds to a $\nato$-shaped diagram of $\infty$-categories which includes the vertices and arrows:
\[
\tag{D}
\begin{tikzcd}[column sep=0.5em]
\vdots 
\ar[d, shift left=3]\ar[d, shift left=1] \ar[d, shift right=1] \ar[d, shift right=3]
&
\vdots
\ar[d, shift left=3]\ar[d, shift left=1] \ar[d, shift right=1] \ar[d, shift right=3]
\\
\Bmod(\C_-) \times_{\Alg(\Cc_-)} \Bmod(\C_-)
\times_{\Alg(\Cc_-)} \Bmod(\C_\mM)
\ar[d, shift left=2] \ar[d] \ar[d, shift right=2]
\ar[r]
&
\Bmod(\C_-) \times_{\Alg(\Cc_-)} \Bmod(\C_-)
\ar[d, shift left=2] \ar[d] \ar[d, shift right=2]
\\
\Bmod(\C_-) \times_{\Alg(\Cc_-)} \Bmod(\C_\m) 
\ar[d, shift left=1] \ar[d, shift right=1]
\ar[r]
&
\Bmod(\C_-)
\ar[d, shift left=1] \ar[d, shift right=1]
\\
\Bmod(\C_\m)
\ar[r]
&
\Alg(\Cc_-).
\end{tikzcd}
\]

\subsection{Extracting the $\LM$-monoidal $\infty$-category $\BLMod_A(\C)$} \label{extracting}
We construct the $\LM$-monoidal $\infty$-category $\BLMod_A(\C)$ (\cref{w9f2}).
We will make use of a certain sub-family of $\infty$-operads $\MAss^\otimes \subset \Tens^\otimes$ which forgets those colors of $\Tens^\otimes$ which index bimodules.

\begin{mydef}
\label{20510a}
Let $\MAss^\otimes$ denote the full subcategory of $\Tens^\otimes$ whose objects
\[
c_\pm: \langle n \rangle^\circ \rightrightarrows [k]
\]
satisfy $c_- = c_+$. Thus, an object is a triple $c=(\langle n \rangle, [k], c)$ with $\lan n \ran \in \Ass$, $[k] \in \Delta$, and $c: \lan n \ran^\circ \to [k]$ a map of sets, and a morphism $c \to c'$ consists of a pair $(\al, \la)$ where $\al$ is a morphism $\lan n \ran \to \lan n' \ran$ in $\Ass$, and $\la$ is a morphism $[k] \from [k']$ in $\Delta$. These are subject to the following condition: if $\al^\circ$ denotes the partially defined function associated to $\al$, then the square
\[
\begin{tikzcd}
\lan n \ran^\circ \ar[d,"c"] \ar[r,"\al^\circ", dotted] 
& 
\lan n' \ran^\circ \ar[d,"c' "] 
\\
{[k]}   & {[k']} \ar[l,"\la"']
\end{tikzcd}
\]
commutes, in the sense that whenever $i \in \lan n \ran^\circ$ satisfies $\al(i) \neq \ast$, we have
\[
c(i) = \la c' \al(i).
\]
\end{mydef}

Our results concerning the generalized $\infty$-operad $\MAss^\otimes$ and the associated double $\infty$-categories (\cref{20510e} below) amount to an essentially trivial portion of the results concerning $\Tens^\otimes$ and the associated double $\infty$-categories $\Bmod(\Cc)^\circledast$ obtained in \cite[\S4.4.3]{HA}; we limit ourselves here to an outline interspersed with references to loc. cit. 

\begin{prop} \label{20510b}
The inclusion
\[
\ep: \MAss^\otimes \subset \Tens^\otimes
\]
is a morphism of $\Delta^\op$-families of $\infty$-operads. The composite 
\[
\MAss^\otimes 
\xto{\ep} 
\Tens^\otimes 
\xto{\pi} \Delta^\op
\]
is a Cartesian fibration, hence in particular flat (\cite[B.3.8]{HA}).
\end{prop}

\begin{proof}
The first statement is clear. For the second statement, since $\pi \ep$ is a functor between 1-categories, it's automatically an inner fibration. Suppose given a morphism $\la:[k'] \from [k]$ in $\Delta$ and an object $c: \angled{n}^\circ \to [k]$ of $\MAss^\otimes$ lying over $[k] \in \Delta^\op$. Then the square diagram
\[
\begin{tikzcd}
\lan n \ran^\circ \ar[d,"\la c", dotted] \ar[r,"id", dotted] 
& 
\lan n \ran^\circ \ar[d,"c "] \\
k   
& 
k \ar[l,"\la"']
\end{tikzcd}
\]
describes a $\pi\ep$-Cartesian edge in $\MAss^\otimes$.
\end{proof}

Given an object $I \in \Delta^\op$, we denote by $\MAss^\otimes_I$ the fiber of $\MAss^\otimes \to \Delta^\op$ over $I$. For any $k \ge 0$, we have an equivalence of $\infty$-operads
\[
\MAss^\otimes_{[k]} \simeq (\Ass)^{\coprod (k+1)}.
\]
Indeed, this follows from the proof of \cite[Proposition 4.4.1.11]{HA} (Segal condition for $\Tens^\otimes$) and can also easily be checked directly. 

We denote by
\[
\ep_\nato: \MAss^\otimes_{\nato} 
\to
\Tens^\otimes_\nato
\]
the pullback of the inclusion $\ep: \MAss^\otimes \subset \Tens^\otimes$ along the map
\[
\delta: \nato \to \Delta^\op 
\]
of \cref{def:plus_and_delta}. Given a map of simplicial sets $f: K \to \nato$ we denote by
\[
\MAss^\otimes_f = \MAss^\otimes_K
\]
the pullback of $\MAss^\otimes_\nato$ along $f$.

\begin{mydef} \label{20510d}
Let $q: \Cc^\otimes \to \BM^\otimes$ be a $\BM$-monoidal $\infty$-category. 
We define a map of simplicial sets
\[
q_a \colon
\mathrm{LAlg}(\Cc)^\circledast 
\to
\nato
\]
by the following universal property:
\begin{itemize}
\item[(*)] For every map of simplicial sets $K \to \nato$ there is a canonical bijection
\[
\Hom_{{\Set_\Delta}_{/\nato}}(K,\rm{LAlg}(\C)^\circledast) \= \Alg_{\MAss_K/\BM}(\C)_0.
\]
Here the subscript `0' on the right refers to the collection of 0-simplices, and the subscript ``$\MAss_K/\BM$''
refers to the composite morphism of generalized $\infty$-operads
\[
\MAss^\otimes_{\nato} 
\xto{\epsilon_\nato}
\Tens^\otimes_\nato
\xto{\varphi}
\BM^\otimes
\]
where $\varphi$ is the map constructed in \cref{const:Tens_S_to_BM}.
\end{itemize}
\end{mydef}

\begin{prop}
\label{20510e}
Let $q: \Cc^\otimes \to \BM^\otimes$ be a $\BM$-monoidal category. 
Then the map $q_a \colon \mathrm{LAlg}(\Cc)^\circledast \to \nato$ of \cref{20510d} is a coCartesian fibration. 
\end{prop}
\begin{proof}
The map $\MAss^\otimes_\nato \to \nato$ is flat since it is a pullback of the flat map
\[
\MAss^\otimes \to \Delta^\op,
\]
see \cref{20510b}.
As in \cref{BTen_flat}, it follows that $q_a$ is an inner fibration. 
Fixing a square (solid arrow diagram below)
\[
\tag{*}
\begin{tikzcd}
\Delta^0 
\ar[r, "F_0"] \ar[d, "\{0\}", swap] &
\LAlg(\Cc)^\circledast
\ar[d, "q_a"]
\\
\Delta^1
\ar[ur, "F", dashed]
\ar[r, "\la", swap]
          &
          \nato
\end{tikzcd}
\]
we must establish the existence of a coCartesian lift $F$ of $\la$ extending $F_0$ as indicated. 
If we regard $\la$ also as a morphism $\la:s \to s'$ in $\nato$, the square diagram (*) corresponds to a diagram
\[
\tag{**}
\begin{tikzcd}
\MAss^\otimes_s
\ar[r, "F_0"] \ar[d, "\iota_0", swap] &
\Cc^\otimes
\ar[d, "q"]
\\
\MAss^\otimes_\la
\ar[ur, "F", dashed]
\ar[r]
            &
            \BM^\otimes. 
\end{tikzcd}
\]
In view of the flatness of $\MAss^\otimes_\nato \to \nato$, the proof of \cref{coCart_edges_in_BLMod} applies with little change to show that it suffices to establish the existence of an operadic $q$-\textit{left Kan extension} $F$ as in diagram (**). 

If $\mu:[k]\to [k']$ is any morphism in $\Delta$ then $\MAss^\otimes_\mu$ regarded a correspondence of $\infty$-operads from $\MAss^\otimes_{[k']}$ to $\MAss^\otimes_{[k]}$ is actually the Cartesian fibration of $\infty$-operads associated with a map of $\infty$-operads
\[
v_\mu: \MAss^\otimes_{[k']} \from 
\MAss^\otimes_{[k]}
\]
given by composition with $\mu$ (compare \cite[Notation 4.4.1.10]{HA}, where the analogous correspondences are nontrivial, except when $\mu$ has convex image). 
Consequently, the map $\iota_0$ of diagram (**) admits a retraction $V_\la$ and (similarly to \cite[Example 4.4.3.6]{HA}), 
\[
F:= F_0 \circ V_\la
\]
is an operadic $q$-\textit{left Kan extension}. 

Indeed, since $V_\la$ is a functor between 1-categories, it may be constructed in an elementary way as follows. Suppose 
\[\begin{tikzcd}
	\Delta^1 & \nato & \Delta^\op
	\arrow["\lambda"', from=1-1, to=1-2]
	\arrow["\mu", curve={height=-20pt}, from=1-1, to=1-3]
	\arrow["\delta"', from=1-2, to=1-3]
\end{tikzcd}\]
corresponds to a morphism $\mu: [k] \from [k']$ in $\Delta$. An object of $\MAss^\otimes_\la$ consists of a map of sets $c: \lan n \ran^\circ \to [l]$ where $l \in \{k, k'\}$, and a morphism consists of data
\[
\begin{tikzcd}
\tag{$\al, \nu$}
\lan n \ran^\circ \ar[d,"c"] \ar[r,"\al", dotted] 
& 
\lan n' \ran^\circ \ar[d,"c' "] 
\\
{[l]}   & {[l']} \ar[l,"\nu"']
\end{tikzcd}
\]
as in ~\cref{20510a} in which $\nu$ may be the identity of $[k]$, the identity of $[k']$ or the morphism $\mu$. We define $V_\la$ on objects by
\[
V_\la(c) := \begin{cases}
c & \text{if } l = k
\\
v_\mu(c) = \mu \circ c & \text{if } l = k',
\end{cases}
\]
and on morphisms by
\[
V_\la(\al, \nu) = \begin{cases}
(\al, \nu) & \text{if } \nu = id_{[k]}
\\
(\al, id_{[k]}) & \text{if } \nu = id_{[k']}
\\
(\al, id_{[k]}) & \text{if } \nu = \mu.
\end{cases}
\]

Subsequently, the operadic $q$-colimits that make $F$ into an operadic $q$-\textit{left Kan extension} may be established in two steps, in analogy with paragraphs 4.4.3.4 -- 4.4.3.6 of \cite{HA}, and we hope the reader will be content with an outline. Since the projection map
\[
\MAss^\otimes_\la \to \Delta^1
\]
is a Cartesian fibration, for every object $X \in \MAss^\otimes_{[k']}$ in the fiber above $1 \in \Delta^1$, the \oo-category
\[
K :=\MAss^\otimes_{[k]} \times_{\MAss^\otimes_\la} (\MAss^\otimes_\la)^{\rm{act}}_{/X}
\]
contains $v_\mu(X)$ as a final object. Hence the morphism $v_\mu(X) \to \infty$ to the cone point in the cocone $K^\rhd$ is cofinal. Since $V_\la$ sends the Cartesian edge $v_\mu(X) \to X$ of $\MAss^\otimes_\la$ to an identity morphism in $\MAss^\otimes_{[k]}$, it follows that
\[
F\big(v_\mu(X)\big) \to F(X)
\]
is an identity morphism in $\Cc^\otimes$, and hence an operadic $q$-colimit diagram. In view of the cofinality noted above, it follows that $F$ is an operadic $q$-\textit{left Kan extension} as claimed.
\end{proof}

We remark that the diagram of $\infty$-categories associated to the coCartesian fibration $q_a \colon \mathrm{LAlg}(\Cc)^\circledast \to \nato$ of \cref{20510e} includes the vertices and arrows
\[
\tag{D}
\begin{tikzcd}
\vdots 
\ar[d, shift left=3]\ar[d, shift left=1] \ar[d, shift right=1] \ar[d, shift right=3]
&
\vdots
\ar[d, shift left=3]\ar[d, shift left=1] \ar[d, shift right=1] \ar[d, shift right=3]
\\
\Alg(\Cc_-)^3 \times \Alg(\Cc_+)
\ar[d, shift left=2] \ar[d] \ar[d, shift right=2]
\ar[r]
&
\Alg(\Cc_-)^3
\ar[d, shift left=2] \ar[d] \ar[d, shift right=2]
\\
\Alg(\Cc_-)^2 \times \Alg(\Cc_+) 
\ar[d, shift left=1] \ar[d, shift right=1]
\ar[r]
&
\Alg(\Cc_-)^2
\ar[d, shift left=1] \ar[d, shift right=1]
\\
\Alg(\Cc_-) \times \Alg(\Cc_+)
\ar[r]
&
\Alg(\Cc_-)
\end{tikzcd}
\]
The map $\ep_\nato:\MAss^\otimes_{\nato} 
\to
\Tens^\otimes_\nato$ gives rise to a map $\rho$ of simplicial sets over $\nato$
\begin{equation*}
\begin{tikzcd}
\BLMod(\Cc)^\circledast 
\ar[dr] \ar[rr, "\rho"]
&&
\rm{LAlg}(\Cc)^\circledast
\ar[dl, "q_a"]
\\
& \nato .
\end{tikzcd}
\end{equation*}
The map
\[
\varphi: \Tens^\otimes_\nato \to \BM^\otimes
\]
restricts to a map
\[
\varphi: \MAss^\otimes_\nato \to (\Ass)^{\coprod 2}
\subset \BM^\otimes.
\]
If $A$ is an algebra in $\Cc_-$ and $B$ is an algebra in $\Cc_+$, we let $A^*B$ be the section of $q_a$ associated to the object
\[
\MAss^\otimes_\nato 
\xto{\varphi}
(\Ass)^{\coprod 2}
\xto{(A, B)}
\C_-^\otimes \coprod \C_+^\otimes 
\subset
\Cc^\otimes
\]
of $\Alg_{\MAss^\otimes_\nato/\BM^\otimes}(\C)$.

\begin{construction} \label{w9f2}
Let $\C^\otimes \to \BM^\otimes$ be a $\BM$-monoidal $\infty$-category and let  $A \in \Alg(\C_-)$, $B \in \Alg(\Cc_+)$ be algebra objects. 
Let $\Aa \subset \Alg(\C_-)$ be the full subcategory spanned by all algebras equivalent to $A$ and assume $\Cc^\otimes$ admits realizations of $\Aa$-bar constructions $\otimes$-compatibly, so that by \cref{w26j1} we have an associated coCartesian fibration 
\[
p_\Aa: \BLMod_\Aa(\Cc)^\circledast \to \nato
\]
which obeys the Segal conditions of \cref{w26j3}. 
We define the coCartesian fibration
\[
_A \BLMod _B(\Cc)^\circledast
\to
\nato
\]
by the pullback
\begin{equation*}
\begin{tikzcd}
_A \BLMod_B(\C)^\circledast \pbcorner 
\ar[r] \ar[d, "\rho_{A,B}"'] 
& 
\BLMod_\A(\C)^\circledast 
\ar[d,"\rho"] 
\\
\nato \ar[r,"A^*B"] 
&  
\rm{LAlg}(\C)^\circledast.
\end{tikzcd}
\end{equation*}
in $\Cat^{coCart}_\nato$. Applying \cref{AC13}, we obtain an $\LM$-monoidal category
\[
_A\BLMod_B(\Cc)^\otimes \to \LM^\otimes
\]
witnessing $\BMod{A}{B}(\C_\m)$ as left-tensored over $\BMod{A}{A}(\Cc_-)^\otimes$.

Now suppose given additional algebras $A' \in \Alg(\C_-)$, $B' \in \Alg(\C_+)$ and morphisms of algebras 
\[
f: A' \to A
\quad 
\mbox{and}
\quad
g:B' \to B.
\]
Then then there's an evident natural transformation
\[
f^*g: \Delta^1 \times \nato \to \rm{LAlg}(\Cc)^\circledast
\]
from $(A')^*B'$ to $A^*B$. Consider the pullback 
\begin{equation*}
\begin{tikzcd}
_f \BLMod_g(\C)^\circledast \pbcorner 
\ar[r] \ar[d, "\rho_{f,g}"'] 
&
\BLMod_\A(\C)^\circledast 
\ar[d,"\rho"] 
\\
\Delta^1 \times \nato \ar[r,"f^*g"] 
&
\rm{LAlg}(\C)^\circledast.
\end{tikzcd}
\end{equation*}
Then $\rho_{f,g}$ is a coCartesian fibration with fibers $_{A'} \BLMod_{B'}(\C)^\circledast$ over $0 \in \Delta^1$ and $_{A} \BLMod_{B}(\C)^\circledast$ over $1 \in \Delta^1$. (CoCartesian lifts of the arrow $(0 \to 1, id_x)$ for $x \in \nato$ are given by extension of scalars along $f$ and $g$.) Hence the composition
\[
\tag{*}
{_f \BLMod_g(\C)^\circledast }
\to \Delta^1 \times \nato 
\to \Delta^1
\]
is a coCartesian fibration. Additionally, the composition (*) is also a Cartesian fibration, and $\rho_{f,g}$ is a morphism of Cartesian-coCartesian fibrations over $\Delta^1$, hence gives rise to commuting triangles\footnote{Since this may cause confusion, let us emphasize that the diagonal maps $\rho_{?,?}$ are \textit{coCartesian fibartions}.}
\[
\begin{tikzcd}
_{A'} \BLMod_{B'}(\C)^\circledast
\ar[dr,"\rho_{A',B'}"']
\ar[rr, shift right = 2, "{(f,g)^*}"']
&&
_{A} \BLMod_{B}(\C)^\circledast
\ar[ll, "{(f,g)_*}"'] \ar[dl, "\rho_{A,B}"]
\\
&
\nato
\end{tikzcd}
\]
plus an adjunction with $(f,g)^*$ left adjoint to $(f,g)_*$. 
If $e$ is a $\rho_{A,B}$-coCartesian edge in $_{A} \BLMod_{B}(\C)^\circledast$ such that $\rho_{A,B}(e)$ is inert (\cref{NN11}), then $(f,g)_*(e)$ is $\rho_{A',B'}$-coCartesian. Hence $(f,g)_*$ gives rise to a lax $\LM$-monoidal functor 
\[
_{A'}\BLMod_{B'}(\Cc)^\otimes 
\from
_A\BLMod_B(\Cc)^\otimes 
\]
given informally by restriction of scalars along $f$ and $g$ (compare \cite[Definition 1.1.8]{DAGII}). 

More generally, we may replace the map $A^*B$ by the universal family of such maps
\[
\Alg(\C_-) \times \Alg(C_+) \times \nato
\to \opnm{LAlg}(\C)^\circledast
\]
to obtain a lax-monoidal functoriality in pairs $(A,B) \in \Alg(\C_-) \times \Alg(C_+)$.

\end{construction}

\section{\texorpdfstring{$\LM$}{LM}-algebras and \texorpdfstring{$\LM$}{LM}-monoidal pairings}
\label{sec:LM_monoidal_pairings}


In this section we collect a few preliminaries in preparation for \cref{SectionKoszLM}.

\subsection{\texorpdfstring{$\LM$}{LM}-monoidal $\infty$-categories and \texorpdfstring{$\LM$}{LM}-algebras}


Let $\C^\otimes$ be an $\LM$-monoidal $\infty$-category and $A\in \Alg(\C_\a)$ an algebra. Under the assumption that $\C$ admits realizations of $A$-bar constructions $\otimes$-compatibly, the following proposition provides a concrete interpretation of an $\LM$-algebra in the $\LM$-monoidal category $\BLMod_A( \C)^\otimes$.

\begin{prop} \label{LModBLMod_as_PB}
Let $\C^\otimes$ be an $\LM$-monoidal $\infty$-category and $A\in \Alg(\C_\aA)$ an algebra. 
Assume $\C$ admits realizations of $A$-bar constructions $\otimes$-compatibly.
Then we have a pullback square of $\infty$-categories 
\[
\begin{tikzcd}
\LMod( \BLMod_A( \C) ) 
\ar{d}{\mathrm{fgt}} \ar{r} \pbcorner 
&
\Alg( \BMod{A}{A}(\C_\aA) ) 
\ar{d}{\mathrm{fgt}} 
\\
\LMod(\C) \ar{r} & \Alg(\C_\aA).
\end{tikzcd}
\]
\end{prop}

\begin{proof}

Consider first a general $\LM$-monoidal $\infty$-category $\Ee^\otimes \to \LM^\otimes$  exhibiting $\E_\m$ as left-tensored over $\E_\a$ \cite[Definition 4.2.1.19]{HA}. The forgetful functor
\[
f:\LMod(\Ee) \to \Alg(\E_\a)
\]
is a Cartesian fibration. Indeed, under the weaker assumptions of \cite[4.2.1.13]{HA}, $f$ is already a categorical fibration; here, the map of left modules
\[
(A,M) \to (B,M)
\]
given by a morphism of algebras $\phi:A \to B$ and the identity map of $M$ provides a Cartesian lift of $\phi$ along $f$. Applying this to $\Ee^\otimes = \Cc^\otimes$ and to $\Ee^\otimes = \BLMod_A(\C)^\otimes$,
we find that it's enough to check that for every
\[
B \in \Alg(\BMod{A}{A}(\Cc_\aA))
\]
the induced functor between the fibers 
\begin{align}
\label{eq:fiber_funct}
U:\LMod_B ( \BLMod_A( \C)_\m ) \to \LMod_B(\C_\m)
\end{align}
is an equivalence (we abuse notation by writing $B$ again for the projection of $B$ to $\Alg(\C_\aA)$ along the forgetful functor). 

A candidate inverse 
\begin{align}
\label{eq:fiber_funct}
V:\LMod_B ( \BLMod_A( \C)_\m ) \from \LMod_B(\C_\m)
\end{align}
may be constructed as follows. Restriction of scalars (\cref{w9f2}) along the unit $A \to B$ of $B$ provides a functor 
\begin{align}
\label{provfunc}
\LMod(\BLMod_B(\C)) \to \LMod(\BLMod_A(\C)).
\end{align}
Using the structure of $B$ as a trivial algebra in $B$-$B$-bimodules, as a special case of the equivalence 
$
\LMod_\one(\Ee_\m) \simeq \Ee_\mM
$
in a general $\LM$-monoidal category we have an equivalence
\[
\LMod_B(\C_\m)
\simeq
\LMod_B(\BLMod_B(\C)).
\]
Combining with \ref{provfunc} we obtain a functor 
\[
\LMod_B(\C_\m)
\simeq
\LMod_B(\BLMod_B(\C)) \to \LMod_B(\BLMod_A(\C)).
\]
Using the functoriality of restriction of scalars, one can check that $U$ and $V$ are indeed inverse to one another.
\end{proof}

\begin{rem}
\label{UC23}
In the setting of \cref{LModBLMod_as_PB}, there's a further equivalence
\[
\Alg( \BMod{A}{A}(\C_\aA) ) 
\simeq
\Alg(\C_\aA)_{A/}.
\]
This may be extracted from \textit{Higher Algebra} as follows. Assuming only that $\Cc_\aA \to \Assoc$ is a fibration of $\infty$-operads, Lurie introduces in Definition 3.3.3.8 and Theorem 3.3.3.9 of \cite{HA} a fibration of $\infty$-operads
\[
\opnm{Mod}_A^\mathrm{Assoc}(\Cc_\aA)^\otimes
\to
\Assoc,
\]
and constructs in Corollary 3.4.1.7 a categorical equivalence
\[
\Alg \big(\opnm{Mod}_A^\mathrm{Assoc}(\Cc_\aA) \big)
\xto{\sim}
\Alg(\Cc)^{A/}
\simeq
\Alg(\Cc)_{A/}
\]
(proposition 4.2.1.5 of \cite{HTT} for this last equivalence). Assuming $\Cc^\otimes \to \Assoc$ is a monoidal $\infty$-category and that the class of algebras equivalent to $A$ is capable of relative tensor products, we construct in our Construction \ref{lbl5} and Example \ref{lb5.5} the monoidal $\infty$-category 
\[
{_A \opnm{BMod}_A}(\Cc)^\otimes \to \Assoc.
\]
Theorem 4.4.1.28 of \cite{HA} gives rise to an equivalence of $\infty$-operads 
\[
\opnm{Mod}_A^\mathrm{Assoc}(\Cc_\aA)^\otimes
\simeq
{_A \opnm{BMod}_A}(\Cc)^\otimes 
\]
over $\Assoc$.

\end{rem}

\subsection{LM-monoidal pairings}\label{DualityFromPairing}
We review the notion of a left representable pairing of $\infty$-categories from \cite[sec. 5.2.1]{HA}.

A \emph{pairing} of $\infty$-categories is a right fibration 
\[
\la \colon \M \to \C \times \D.
\]
Let $\chi \colon \C^{op} \times \D^{op} \to \Ss$ be the associated functor. 
We can view $\chi$ also as a contravariant functor from $\C$ to presheaves on $\D$, 
\[
\hat{\Ddual}_\lambda \colon \C^{op} \to \Psh(\D).
\]
An object $M\in \M$ over $(X,Y) \in \C \times \D$ is \emph{left universal} if it is a terminal object of the fiber $\M \times_{\C} \set{X}$ (see \cite[def. 5.2.1.8]{HA}).
We say that $M\in \M$ is a left universal lift of $X \in \C$.
We say that the pairing $\lambda$ is \emph{left representable} if every $X\in \C$ has a left universal lift.
If the pairing $\lambda$ is left representable then $\hat{\Ddual}_\lambda$ factors through the Yoneda embedding, and hence defines a functor
\[
\Ddual_\lambda \colon \C^{\op} \to \D
\]
(loc. cit.). We refer to $\D_\lambda \colon \C^{op} \to \D$ as the \emph{left duality functor} associated to $\lambda$.

Similarly, the pairing $\lambda$ defines a functor \( \hat{\Ddual}'_\lambda \colon \D^{op} \to \Psh(\C) \).
An object $M\in \M$ over $(X,Y)\in \C \times \D$ is \emph{right universal} if it is terminal in $\M \times_\D \set{Y}$.
A pairing is \emph{right representable} if every $Y\in \D$ has a right universal lift.
If the pairing $\lambda$ is right representable then $\hat{\Ddual}'_\lambda$ factors through the Yoneda embedding $\C \to \Psh(\C)$, and defines a \emph{right duality functor}
\[
\Ddual'_\lambda \colon \D^{op} \to \C.
\]
Moreover, if the pairing $\lambda$ is both left and right representable then the duality functors associated to $\lambda$ define an adjunction
\begin{align*}
\Ddual_\lambda^{op} \colon & \C \adj \D^{op} \noloc \Ddual'_\lambda , \\
\Map_{\D} ( Y, \Ddual_\lambda(X) ) \simeq & \chi(X,Y) \simeq \M \times_{\C \times \D} \set{(X,Y)} \simeq \Map_{\C} ( X, \Ddual'_\lambda(Y)).
\end{align*}

Our definition of ``$\LM$-monoidal pairing'' is a special case of~\cite[def. 5.2.2.20]{HA}
\begin{mydef}\label{monpair}
A pairing of $\LM$-monoidal $\infty$-categories, or simply an $\LM$-monoidal pairing, 
is given by three $\LM$-monoidal categories $\C^\otimes,\D^\otimes,\M^\otimes$ and an $\LM$-monoidal functor 
\[
\la \colon \M^\otimes \to \C^\otimes \times_{\LM^\otimes} \D^\otimes 
\]
which is a categorical fibration (best understood using~\cite[cor. 2.4.6.5]{HTT}) 
and which induces right fibrations
\[
\la_\a \colon M_\a \to \C_\a \times \D_\a, \quad \la_\m \colon M_\m \to \C_\m \times \D_\m
\]
after taking the fibers over $\a, \m \in \LM^\otimes_{\angled{1}}$.
We define the \emph{associated monoidal pairing} 
\[
\la_\a^\otimes \colon
\M_\a^\otimes \to \C_\a^\otimes 
\times_{\Ass} \D_\a^\otimes 
\]
by the pullback of $\la$ along $\Ass \to \LM^\otimes$. If the underlying pairings $\la_\a$, $\la_\m$ are both left representable, we say the monoidal pairing $\la$ is \emph{left representable}. 
\end{mydef}

\begin{rem}
If the pairing $\la_\a$ is left representable, then the duality functor 
\[
\Ddual_{\la_a} \colon \C_\a^\op \to \D_\a
\]
is lax monoidal.
We can therefore view $\D_\m$ as lax left tensored over $\C_\a^\op$ via $\Ddual_{\la_\a}$.
If in addition the pairing $\la_\m$ is left representable (so that in the terminology of \cref{monpair} $\la$ itself is left representable),
then the duality functor
\[
\Ddual_{\la_\m} \colon \C_\m^\op \to \D_\m
\]
is ``lax left linear over $\C_\a^\op$'':
for every $x\in \C_\a, y\in \C_\m$ there exists a natural morphism in $\D_\m$
\[
\Ddual_{\la_\a}(x) \otimes \Ddual_{\la_\m}(y) \to \Ddual_{\la_\m} ( x \otimes y),
\]
and these morphisms are coherently compatible.
\end{rem}

Let
\[
\la \colon \M^\otimes \to \C^\otimes \times_{\LM^\otimes} \D^\otimes
\]
be an $\LM$-monoidal pairing.
By~\cite[rem. 5.2.2.26]{HA} $\la$ induces a pairing of $\infty$-categories 
\begin{align*}
\LMod(\lambda)  \colon \LMod(\M) \to \LMod(\C) \times \LMod(\D).
\end{align*}
The goal of \cref{SectionKoszLM} is to show that $\LMod(\la)$ is left representable (under suitable assumptions), and hence defines a duality functor 
\[
\Ddual_{\LMod(\la)} \colon \LMod(\C)^\op \to \LMod(\D).
\]
We will produce left universal lifts by studying $\LM$-monoidal pairings of modules.


\subsection{$\LM$-monoidal pairings of bimodules}
Let
\[
\lambda^\otimes \colon \M^\otimes \to \C^\otimes \times_{\LM^\otimes} \D^\otimes
\]
be an $\LM$-monoidal pairing,
and let $M\in \Alg(\M_\aA)$ be an object over $(A,B)\in \Alg(\C_\aA)\times \Alg(\D_\aA)$. 
Assume $\M$ admits geometric realizations of $M$-bar constructions $\otimes$-compatibly, and similarly for $A$ and $B$. 
By its functoriality, construction \ref{w9f2} gives rise to a map of $\LM$-monoidal $\infty$-categories
\[
\la_M^\otimes:
{\BLMod_{M}(\M)}^\otimes
\to
{\BLMod_{A}(\C)}^\otimes 
\times_{\LM^\otimes} {\BLMod_{B}(\D)}^\otimes.
\]
However, in that construction we lose track of individual simplices. 
Consequently, we are unable to tell if  $ \la_M^\otimes$ is a categorical fibration. 
Thus, a direct analog of~\cite[lem. 5.2.2.29]{HA} is lacking.
Instead, in \cref{BLMod_construction} we take a fibrant replacement of $\la^\otimes_M$ in the model category of preoperads.

\begin{prop} \label{BLMod_construction}
In the situation and the notation described above, there exists an $\LM$-monoidal pairing 
\[
\tilde\lambda_M^\otimes \colon 
\widetilde {\BLMod_{M}(\M)}^\otimes
\to
{\BLMod_{A}(\C)}^\otimes 
\times_{\LM^\otimes} {\BLMod_{B}(\D)}^\otimes, 
\]
and an equivalence of $\LM$-monoidal categories 
\[
\tau \colon {\BLMod_{M}(\M)}^\otimes \iso
\widetilde {\BLMod_{M}(\M)}^\otimes
\]
such that $ \tilde\lambda_M^\otimes \circ \tau =  \lambda_M^\otimes$.
\end{prop}
\begin{proof}
We let $\POp$ denote the category of $\infty$-preoperads of~\cite[def. 2.1.4.2]{HA} endowed with the model structure of~\cite[prop. 2.1.4.6]{HA}. 
Given an $\infty$-operad $\Oo^\otimes$ we let $\Oo^{\otimes, \natural}$ denote the associated $\infty$-preoperad of~\cite[not. 2.1.4.5]{HA}. 
We recall that the overcategory 
${\POp}_{/\LM^{\otimes, \natural}}$
inherits a model structure in which weak equivalences, fibrations, and cofibrations are precisely the morphisms that become such after forgetting the morphism to  $\LM^{\otimes, \natural}$. 
Proposition 2.1.4.6 of \cite{HA} shows that the object
\begin{align*}
{\BLMod_A(\C)}^{\otimes, \natural}
& \times_{\LM^{\otimes, \natural}}
{\BLMod_B(\D)}^{\otimes, \natural} \\
&=
\left(
{\BLMod_A(\C)}^{\otimes}
\times_{\LM^{\otimes}}
{\BLMod_B(\D)}^{\otimes} 
\right)^\natural
\to 
\LM^{\otimes, \natural}
\end{align*}
of 
${\POp}_{/\LM^{\otimes, \natural}}$
is fibrant. 
We factor the induced map $\lambda_M^{\otimes, \natural}$ as
\[
{\BLMod_M(\M)}^{\otimes, \natural} 
\xto{\tau}
X
\xto{\mu}
{\BLMod_A(\C)}^{\otimes, \natural}
\times_{\LM^{\otimes, \natural}}
{\BLMod_B(\D)}^{\otimes, \natural} 
\]
where $\tau$ is a trivial cofibration and $\mu$ is a fibration.
Referring again to~\cite[prop. 2.1.4.6]{HA}, we find that there exists a fibration of $\infty$-operads $\tilde \lambda_M^\otimes$ as in the proposition such that $\mu$ is induced by a weak equivalence 
\[
X \simeq {\widetilde\BLMod_M(\M)}^{\otimes, \natural}  
\]
in ${\POp}_{/\LM^{\otimes, \natural}}$. 
In fact, ${\widetilde\BLMod_M(\M)}^{\otimes}$ is obtained from $X$ simply by forgetting the marked edges.
We will not distinguish notationaly between 
\[
\tau \colon {\BLMod_{M}(\M)}^\otimes \iso
\widetilde {\BLMod_{M}(\M)}^\otimes
\]
and the induced map on marked simplicial sets. Composing with $\tau$ we obtain a weak equivalence 
\[
{\BLMod_M(\M)}^{\otimes, \natural}
\simeq
{\widetilde\BLMod_M(\M)}^{\otimes, \natural}.
\]
By~\cite[prop. 2.1.4.6]{HA}, the map $\tilde \lambda_M^\otimes$ is a categorical fibration. 

We claim that $\tau$ induces an equivalence of $\LM^\otimes$-monoidal categories as in the lemma.
By~\cite[prop. 2.1.4.6]{HA} a fibrant object in ${(\POp)}_{/\LM^{\otimes,\natural}}$ is given by an $\infty$-operad $\Oo^\otimes$ and a categorical fibration 
\[
\Oo^{\otimes, \natural} \fib \LM^{\otimes,\natural}.
\]
If $\Oo^\otimes$, $\Oo'^{ \otimes}$ are $\infty$-operads over $\LM^\otimes$ then the Kan complex 
\[
\Map_{{(\POp)}_{/\LM^{\otimes,\natural}}}
(\Oo^{ \otimes, \natural},\Oo'^{ \otimes, \natural})
\]
is a model for the mapping space 
$
\Map_{\mathrm{Op}_{\infty_{/\LM^\otimes}}} 
(\Oo^{ \otimes},\Oo'^{ \otimes}).
$
Therefore, by the Yoneda lemma, a weak equivalence 
\[
f \colon \Oo^{\otimes, \natural} \to \Oo'^{\otimes, \natural}
\]
in ${(\POp)}_{/\LM^{\otimes, \natural}}$ is induced by an equivalence of  $\infty$-operads $f \colon \Oo^\otimes \iso \Oo'^\otimes$ over $\LM^\otimes$.
If $\Oo^\otimes$ is an $\LM$-monoidal category (i.e $\Oo^\otimes \fib \LM^\otimes$ is a coCartesian fibration), then by the abstract characterization of coCartesian edges~\cite[prop. 2.4.4.3]{HTT} we see that $\Oo'^\otimes \fib \LM^\otimes$ is also an $\LM$-monoidal category.
Our claim follows by taking $f=\tau$.

In particular, $\tau$ induces equivalences of underlying $\infty$-categories
\begin{align*}
\BMod{M}{M}(\M_\a) = {\BLMod_{M}(\M)}^\otimes_{\a} \iso
\widetilde {\BLMod_{M}(\M)}^\otimes_{\a}, \\
\LMod_{M}(\M_\m) = {\BLMod_{M}(\M)}^\otimes_{\m} \iso
\widetilde {\BLMod_{M}(\M)}^\otimes_{\m}.
\end{align*}

Taking fibers over the two colors of $\LM^\otimes$, it remains to show that the associated maps
\begin{small}
\begin{align*}
\tilde \lambda_{M, \a}^\otimes & \colon 
{\widetilde\BLMod_{M}(\M)}^\otimes_{\a} \to 
{\BLMod_{A}(\C)}^\otimes_{\a} \times {\BLMod_{B}(\D)}^\otimes_{\a} = \BMod{A}{A}(\C_\a) \times \BMod{B}{B}(\D_\a), \\
\tilde \lambda_{M, \m}^\otimes & \colon 
{\widetilde\BLMod_{M}(\M)}^\otimes_{\m} \to 
{\BLMod_{A}(\C)}^\otimes_{\m} \times {\BLMod_{B}(\D)}^\otimes_{\m} = \LMod_{A}(\C_\m) \times \LMod_{B}(\D_\m)  
\end{align*}
\end{small}
are right fibrations. 

It suffices to show that the compositions 
\begin{align*}
\BMod{M}{M}(\M_\a)
& \iso
{\widetilde\BLMod_{M}(\M)}^\otimes_{\a} \xto{\tilde \lambda_{M, \a}^\otimes}
\BMod{A}{A}(\C_\a) \times \BMod{B}{B}(\D_\a), \\
\LMod_{M}(\M_\m)
& \iso
{\widetilde\BLMod_{M}(\M)}^\otimes_{\m} \xto{\tilde \lambda_{M, \m}^\otimes} 
\LMod_{A}(\C_\m) \times \LMod_{B}(\D_\m)  
\end{align*}
obtained using the factorization of $\lambda^{\otimes,\natural}_M$ above, are right fibrations. 
The first is a right fibration by the proof of~\cite[lem. 5.2.2.29]{HA}.
The second is a right fibration by the same argument \textit{mutatis mutandis}; we nevertheless give the details. 

The various categories of algebras and left modules at play form a diagram as follows,
\[
\begin{tikzcd}[ampersand replacement=\&]
\LMod_M(\M_\m) 
\ar[d]{}{} \ar[r]
\& 
\LMod(\M) 
\ar[dd] \ar[dr]
\\
\LMod_A(\C_\m) \times \LMod_B(\D_\m) 
\ar[d] \ar[rr]
\&\&
\LMod(\C) \times \LMod(\D)
\ar[d]
\\
\set{*}
\ar[r]
\&
\Alg(\M_\a)
\ar[r]
\&
\Alg(\C_\a) \times \Alg(\D_\a)
\end{tikzcd}
\]
in which the lower and left rectangles are Cartesian. 
Consequently, the induced square 
\[
\begin{tikzcd}[ampersand replacement=\&,column sep=1em]
\LMod_M(\M_\m) 
\ar[d]{}{} \ar[r] 
\& 
\LMod(\M) 
\ar[d]{}{\theta}
\\
\LMod_A(\C_\m) \times \LMod_B(\D_\m) 
\ar[r]
\&
\Alg(\M_\a)
\times_{\Alg(\C_\a)\times\Alg(\D_\a)}
\big(
\LMod(\C) \times \LMod(\D)
\big)
\end{tikzcd}
\]
is Cartesian. 
We will show that $\theta$ is a right fibration. 
Let $\theta'$ be the projection 
\[
\Alg(\M_\a)
\times_{\Alg(\C_\a)\times\Alg(\D_\a)}
\big(
\LMod(\C) \times \LMod(\D)
\big)
\to
\LMod(\C) \times \LMod(\D).
\]
Since $\theta$ is a categorical fibration, it will suffice to show that $\theta'$ and
\[
\theta \circ \theta' = \LMod{}{}(\lambda)
\colon
\LMod{}{}(\M) \to \LMod{}{}(\C)\times \LMod{}{}(\D)
\]
are right fibrations. 
The map $\theta'$ is a pullback of 
\[
\Alg(\lambda_\a)\colon
\Alg(\M_\a) \to \Alg(\C_\a)\times \Alg(\D_\a).
\]
We are therefore reduced to showing that $\Alg(\lambda_\a)$ and $\LMod{}{}(\lambda)$ are right fibrations. 
To fix ideas, let us focus on the latter. 
Let us denote 
\[
\E^\otimes :=  \C^\otimes \times \D^\otimes \quad \text{and} \quad q:= \LMod{}{}(\lambda).
\]
Since $\lambda_\a$ is by assumption a categorical fibration, it follows that $\Alg(\lambda_\a)$ is an inner fibration, and hence that $\theta'$ is an inner fibration. 
Since $\theta$ is an inner fibration as well, it follows that
\[
q: \LMod{}{}(\M) \to \LMod{}{}(\E)
\]
is an inner fibration.
According to~\cite[prop. 2.4.2.4]{HTT}, to show that $q$ is a right fibration we must check two properties:
\begin{itemize}
\item[(A)] for any morphism in $\LMod{}{}(\E)$, any lift of the target may be extended to a Cartesian lift of the morphism, 
\item[(B)] every morphism in $\LMod{}{}(\M)$ is $q$-Cartesian.
\end{itemize}
Recall the characterization of a Cartesian edge as relative limit from~\cite[ex. 4.3.1.4.]{HTT}.
We apply~\cite[cor. 3.2.2.3]{HA} with $\Oo^\otimes := \LM^\otimes$, 
\[
p := \lambda \colon \M^\otimes \to \E^\otimes,
\]
and $K := \Delta^0$ (noting that the objects $X \in \Oo$ considered there have as two possible values the two colors $\a, \m$ of the $\infty$-operad $\LM^\otimes$).
The condition of the Corollary holds because of our assumption that $\lambda_\a$ and $\lambda_\m$ are right-fibrations, and the conclusions (1) and (2) of the Corollary imply the two properties (A), (B) as required.
\end{proof}

\section{Koszul duality for \texorpdfstring{$\LM$}{LM}-monoidal pairings}
\label{SectionKoszLM}


Let
\[
\la \colon \M^\otimes \to \C^\otimes \times_{\LM^\otimes} \D^\otimes
\]
be a left-representable $\LM$-monoidal pairing (\cref{monpair}). 
Our goal in this section is to prove that under certain assumptions, the induced pairing
\[
\LMod(\la) \colon
\LMod(\M)
\to
\LMod(\C) \times \LMod(\D) 
\] 
is left representable (\cref{LM52227}) --- our analog of \cite[5.2.2.27]{HA}. An outline of the proof was already given in \S\ref{SectionOverview}. However, armed as we are now with the constructions of sections \ref{sec:External} and \ref{sec:LM_monoidal_pairings}, we can be more precise. 

The proof revolves around the commutative diagram\footnote{In fact, we will modify this diagram using \cref{BLMod_construction} to ensure that the vertical maps are categorical fibrations.}
\begin{equation*}
\begin{tikzcd}
\LMod(\BLMod_M(\M)) \ar[d] \ar[r] & \LMod(\M) \ar[d,"\LMod(\la)"] \\
\LMod(\BLMod_A(\C)) \times \LMod(\BLMod_\one(\D)) \ar[r] & \LMod(\C) \times \LMod(\D), 
\end{tikzcd}
\end{equation*}
where $A=X_\a \in \Alg(\C_\a) $ is the underlying algebra of a left-module $X = (X_\a, X_\m) \in \LMod(\C)$ and $M\in \Alg(\M_\a)$ a lift of $(A,\one)$, whose existence requires further assumptions on the pairing $\la_\a$.\footnote{for example, the assumption that the unit object $\one \in \D_\a$ is terminal.}
We start by lifting $X$ to a left-module
\[
X' \in \LMod(\BLMod_A(\C))
\]
using \cref{triv_BL_mod_lift}.
The advantage of $X'$ over $X$ is that its underlying algebra is the unit object of $\BMod{A}{A}(\C)$.
This makes it is easy to find a left universal lift $Z' \in \LMod(\BLMod_M(\M))$ of $X'$, provided that the underlying pairings of
\[
\BLMod_M(\M)^\otimes \to \BLMod_A(\C)^\otimes \times_{\LM^\otimes} \BLMod_\one(\D)^\otimes
\]
are left representable, see \cref{LM52228}.
Showing the left representability of these underlying pairings is the main step of the proof.

The left representability of 
\[
\BMod{M}{M}(\M) \to \BMod{A}{A}(\C) \times \BMod{\one}{\one}(\D)
\]
was proved in \cite[lem. 5.2.2.40]{HA} using free resolutions of bimodules.
In \cref{LM52240} we use free resolutions of left modules to show the left representability of 
\[
\LMod_{M}(\M) \to \LMod_{A}(\C) \times \LMod_{\one}(\D).
\]

To complete the proof we show that the forgetful functor
\[
\LMod(\BLMod_M(\M)) \to \LMod(\M)
\]
preserves left representable objects, see \cref{LM52230}.
Applying the forgetful functor to $Z'\in \LMod(\BLMod_M(\M))$ produces the desired left universal lift $Z\in \LMod(\M)$ of $X$.

\subsection{Lifting left modules to left modules over algebras in bimodules}
Let $\C^\otimes \to \LM^\otimes$ be an $\LM$-monoidal category and $A\in \Alg(\C_\aA)$ an algebra. 
Let $\BLMod_{A}(\C)$ be the $\LM$-monoidal category of \cref{w9f2}.
Consider the forgetful functor 
\begin{align*}
\LMod(\BLMod_{A}(\C)) \to \LMod(\C).
\end{align*}
\begin{prop}
\label{triv_BL_mod_lift}
Assume $\C^\otimes$ admits realizations of $A$-bar constructions $\otimes$-compatibly. 
Then every object $X\in \LMod(\C)$ over $A\in \Alg(\C_\aA)$ admits a lift to $\LMod(\BLMod_{A}(\C))$ over the trivial algebra $A\in \Alg(\BMod{A}{A}(\C_\aA))$. 
\end{prop}
\begin{proof}
This follows directly from \cref{LModBLMod_as_PB} using the universal property of the pullback:
\[
\begin{tikzcd}
\ast \ar["X"', bend right]{ddr} \ar["A", bend left=10]{drr} \ar[dr, dashed,"\exists!"] 
\\
& \LMod( \BLMod_A( \C) ) \ar{d}{\mathrm{fgt}} \ar{r} \pbcorner 
& \Alg( \BMod{A}{A}(\C_\aA) ) \ar{d}{\mathrm{fgt}}  
\\
& \LMod(\C) \ar{r} & \Alg(\C_\aA). \qedhere
\end{tikzcd} 
\]
\end{proof}

\subsection{Koszul duals of modules over trivial algebras.}
The next proposition, which is analogous to~\cite[prop. 5.2.2.28]{HA}, concerns the special case of our Koszul duality in which the algebra in question is trivial. In \cref{LM52227} we apply this to an arbitrary module $(A,X)$ by regarding $A$ as the trivial algebra in the category of $A$-$A$-bimodules. 
Thus, we will apply \cref{LM52228} to an $\LM$-monoidal pairing of the form
\[
\BLMod_M(\M)^\otimes \to 
\BLMod_A(\C)^\otimes 
\times_{\LM^\otimes} \BLMod_{\one}(\D)^\otimes
\simeq
\BLMod_A(\C)^\otimes
\times_{\LM^\otimes} 
\D^\otimes.
\]
\begin{prop} \label{LM52228}
Let 
\[
\la \colon \M^\otimes \to \C^\otimes \times_{\LM^\otimes} \D^\otimes
\] 
be an $\LM$-monoidal pairing.
Assume the underlying pairings
\[
\la_\a \colon \M_\a \to \C_\a \times \D_\a
\quad \mbox{and} \quad
\la_\m \colon \M_\m \to \C_\m \times \D_\m
\]
are left representable. 
Consider the induced pairing of $\infty$-categories
\[
\LMod(\lambda)  \colon \LMod(\M) \to \LMod(\C) \times \LMod(\D).
\]
Let $A\in \Alg(\C_\a)$ be a trivial algebra object of $\C_\a$.
Then:
\begin{enumerate}
\item Every left $A$-module 
$ X\in \LMod(\C) $
over
$ A \in \Alg(\C_\a) $
has a left universal lift, i.e a left universal object $Y\in \LMod(\M)$ over $X\in \LMod(\C)$. 
\item Let $X\in \LMod(\C)$ be a left module over $A\in \Alg(\C_\a)$ and 
let $Y\in \LMod(\M)$ be an object over $X$. 
Then $Y\in \LMod(\M)$ is a left universal lift of $X$ if and only if its images $Y_\m \in \M_\m, Y_\a \in \M_\a$ are left universal lifts of $X_\m \in C_\m,\, A\in \C_\a$ respectively.
\end{enumerate}
\end{prop}

\begin{proof}
We have to show that $\LMod(\M) \times_{\LMod(\C)} \set{X}$ has a terminal object, and that an object $Y\in \LMod(\M) \times_{\LMod(\C)} \set{X}$ is terminal if and only if the objects 
$Y_\m \in \M_\m \times_{\C_\m} \set{X_\m}$ and $ Y_\a \in \M_\a \times_{\C_\a} \set{A} $ are terminal. 

Consider $\M^\otimes, \C^\otimes$ and $\LM^\otimes$ as coCartesian fibrations over $\LM^\otimes$.
Since $A$ is a trivial algebra, the functor $X \colon \LM^\otimes \to \C^\otimes$ carries coCartesian edges to coCartesian edges.
Define $\N^\otimes \to \LM^\otimes$ as the pullback
\[
\begin{tikzcd}
\N^\otimes \pbcorner \ar[d] \ar[r] & \M^\otimes \ar[d] \\
\LM^\otimes \ar[r, "X"] & \C^\otimes
\end{tikzcd}
\]
in ${(\Cat_{\infty})}^{coCart}_{/\LM^\otimes}$ (\cref{lbl2}).
Verification of the Segal conditions, which is straightforward, shows that $\N^\otimes \to \LM^\otimes$ is an $\LM$-monoidal category.
Note that the $\infty$-category $\N^\otimes$ is equivalent to the fibered product $\M^\otimes \times_{\C^\otimes} \LM^\otimes$, since the forgetful functor
\[ 
{(\Cat_\infty)}^{coCart}_{/\LM^\otimes} \to \Cat_\infty
\]
preserves limits (see \cref{lbl3.1}).
Unwinding definitions, we find that we have equivalences 
\[
\LMod(\N) \= \LMod(\M) \times_{\LMod(\C)} \set{X} , \quad
\N_\a \= \M_\a \times_{\C_\a} \set{A}, \quad
\N_\m \= \M_\m \times_{\C_\m} \set{X_\m}.
\]
The categories $\N_\a \= \M_\a \times_{\C_\a} \set{A}$ and $\N_\m \= \M_\m \times_{\C_\m} \set{X_{\m}}$ each possess a terminal object by assumption.
Since the forgetful functor preserves and reflects limits \cite[cor. 3.2.2.5 applied with $K=\emptyset$]{HA}, $\LMod(\N)$ has a terminal object,
and, moreover, $Y\in \LMod(\N)$ is terminal if and only if
the underlying objects $Y_\m \in \M_\m \times_{\C_\m} \set{X_\m}$ and $ Y_\a \in \M_\a \times_{\C_\a} \set{A} $ are both terminal objects, as claimed.
\end{proof}

\subsection{Left representability for left modules}
\label{sec:left_rep_for_left_mod}
Let
\[
\la \colon
\M^\otimes \to
\Cc^\otimes \times_{\LM^\otimes}
\D^\otimes
\]
be an $\LM$-monoidal pairing and $M \in \Alg(\M_\a)$ be an algebra over
\[
(A,B) \in \Alg(\C_\a) \times \Alg(\D_\a).
\]
Consider the $\LM$-monoidal pairing 
\[
\tilde\la_M^\otimes \colon 
\widetilde{\BLMod}_M(\M)^\otimes \to \BLMod_A(\C)^\otimes \times \BLMod_B(\D)^\otimes
\]
of \cref{BLMod_construction}.
Our next goal is to show that the underlying pairings of $\tilde\la_M^\otimes$ are left representable (see \cref{LM52240}).
We will need the following statement, which is an addendum to~\cite[lem. 5.2.2.32]{HA}.

\begin{lem} \label{LM52232}
In the situation and the notation above, assume that $B$ is a trivial algebra, that $\M$ admits realizations of $M$-bar constructions $\otimes$-compatibly, and that $\C$ admits realizations of $A$-bar constructions $\otimes$-compatibly.\footnote{These assumptions ensure that \cref{w9f2} applies to make $\LMod_M (\M)$ left-tensored over $\operatorname{BMod}_M(\M)$ and similarly for $\LMod_A(\C)$. In keeping with the spirit of this article, we impose these assumptions generously, although they may not be strictly necessary for the present lemma.} Let
\[
F \colon  \M_\mM \to \LMod_M (\M), \quad  F \colon X \mapsto M \otimes X
\]
be a left adjoint to the forgetful functor. Then $F$ carries left universal objects of $\M_\mM$ with respect to the pairing of $\infty$-categories 
\[
\lambda_\mM \colon  \M_\mM \to \C_\mM \times \D_\mM
\] 
to left universal objects of $\LMod_M(\M)$ with respect to the pairing of $\infty$-categories
\[
{(\lambda_M)}_\m \colon \LMod_M(\M) \to \LMod_A(\C) \times \LMod_B(\D).
\]
\end{lem}
The proof will require the following categorical lemma.
\begin{lem} \label{PullbackAdjunctionLemma}
Let 
\[
\begin{tikzcd}
\overline{\C} \ar[r,"\overline{F}"] \ar[d,"p"] & \overline{\D} \ar[d,"q"] \\
\C \ar[r,"F"] & \D
\end{tikzcd} 
\]
be a commutative diagram of $\infty$-categories. 
Assume that $p,q$ are Cartesian fibrations, and that $F,\overline{F}$ have right adjoints $F \dashv U, \, \overline{F} \dashv \overline{U}$ such that $p\circ \overline{U} = U \circ q$. 
Let $X\in \C$ be an object, and denote the fiber of $p$ over $X$ by $\overline{\C}_{X}$
and the fiber of $q$ over $FX$ by $\overline{\D}_{FX}$.
Then the restriction of $\overline{F}$ to the fibers over $X$ and $FX$, 
\[
\overline{F}_X \colon  \overline{\C}_X \to \overline{\D}_{FX} ,
\]
has a right adjoint given by the composition 
\[
\overline{\D}_{FX} \xto{\overline{U}_{FX}} \overline{\C}_{UFX} \xto{\eta^*_X} \overline{\C}_X,
\]
where $\eta^*_X$ is the functor induced by the Cartesian fibration $p$ and the unit map $\eta_X \colon  X \to UFX$.
\end{lem}
\begin{proof}
With $\overline X \in \overline \C_X$ and $\overline Y \in \ol \D_{FX}$ fixed, an elementary calculation with adjoints and Cartesian fibrations yields a homotopy equivalence of mapping spaces 
\[
\tag{*}
\Map_{\overline \C_X} (\ol X, \eta_X^* \ol U_{FX} \ol Y)
\simeq
\Map_{\ol \D_{FX}}(F \ol X, \ol Y).
\]
A straightforward but tedious construction, based, for instance, on the twisted arrow category, upgrades (*) to an equivalence of left fibrations over $\ol \C_X^{op} \times \ol \D_{FX}$. 
Alternatively (by interpreting $F$, $\ol F$ as a morphism of diagrams of $\infty$-categories), this is a special case of~\cite[prop. 2.1.7]{Arad}.
\end{proof}
\begin{proof}[Proof of \cref{LM52232}]
Let
\[ 
F' \colon \C_\mM \to \LMod_A(\C), \quad
F'' \colon \D_\mM \to \LMod_B(\D)
\]
be left adjoints to the forgetful functors
\[
G' \colon  \LMod_A(\C) \to \C_\mM, \quad
G'' \colon  \LMod_B(\D) \to \D_\mM.
\]
Modifying $F$ by a homotopy if necessary, we may assume that the diagram 
\[
\begin{tikzcd}
\M_\mM \ar[r,"F"] \ar[d] & \LMod_M(\M) \ar[d] \\
\C_\mM \times \D_\mM \ar[r,"F' \times F''"] & \LMod_A(\C) \times \LMod_B(\D)
\end{tikzcd}
\]
is commutative.
For each  
$X \in \C_\mM$, $F$ induces a functor
\begin{align*}
f \colon \M_\mM \times_{\C_\mM} \set{X} 
\to 
\LMod_M(\M) \times_{\LMod_A(\C)} \set{F'(X)}, 
\quad & 
f \colon  Z \mapsto M \otimes Z.
\end{align*}
We have to show that $F$ preserves left universal objects over $X$, or equivalently that $f$ preserves terminal objects. 
We will show that $f$ is an equivalence of $\infty$-categories. 
By \cref{PullbackAdjunctionLemma}, the functor $f$ has a right adjoint $g$, given by the composition 
\[
g \colon  
\LMod_M(\M) \times_{\LMod_A(\C)} \set{F'(X)} 
\to 
\M_\mM \times_{\C_\mM} \set{G' F' X} 
\xto{\eta^*_X} 
\M_\mM \times_{\C_\mM} \set{X}
\]
of the forgetful functor and the functor $\eta^*_X$ induced by the unit map $\eta_X \colon X \to G' F' X$. 
We claim that the unit and counit of $f \dashv g$ are both natural equivalences.

Let $u \colon  id \to g \circ f$ be the unit map of $f \dashv g$.
For every object 
$
Z \in \M_\mM \times_{\C_\mM} \set{X}
$ 
with image $Y\in \D_\mM$, the unit map $u_Z \colon  Z \to (g\circ f)(Z)$ has image in $\D_\mM$ equivalent to the unit map 
\[
Y \to (G'' \circ F'') (Y) \simeq B \otimes Y.
\]
Since $B$ is a trivial algebra, the
latter is an equivalence.
Since $\M_\mM \times_{\C_\mM} \set{X} \to \D_\mM$ is a right fibration it follows from~\cite[prop. 2.1.1.5]{HTT} that $u_Z$ is an equivalence.

The same argument shows that the counit of $f \dashv g$ is an equivalence, proving our claim.
\end{proof}

With \cref{LM52232} at hand, we return to showing that the underlying pairings of $\tilde\la_M^\otimes$ are left representable.
This statement is an addendum to~\cite[lem. 5.2.2.40]{HA}.
\begin{prop} \label{LM52240}
Let $M\in \Alg(\M_\aA)$ be an object over $(A,B)\in \Alg(\C_\aA)\times \Alg(\D_\aA)$ and assume $\M$ admits realizations of $M$-bar constructions $\otimes$-compatibly, and similarly for $A$ and $B$.
Assume in addition that:
\begin{enumerate}
\item The object $B\in \Alg(\D_\aA)$ is a trivial algebra.
\item The underlying pairings $\la_\aA, \la_\mM$ are left representable.
\item The $\infty$-categories $\D_\aA, \D_\mM$ admit totalizations of cosimplicial objects. 
\end{enumerate}
Then the induced pairings of $\infty$-categories
\begin{align*}
{(\lambda_M)}_\a & \colon  \BMod{M}{M}(\M) \to \BMod{A}{A}(\C) \times \BMod{B}{B}(\D)  , \\
{(\lambda_M)}_\m & \colon  \LMod_{M}(\M) \to \LMod_{A}(\C) \times \LMod_{B}(\D)  
\end{align*}
are both left representable.

\end{prop}
\begin{proof} 
The right fibration ${(\lambda_M)}_\a$ is left representable by~\cite[prop. 5.2.2.40]{HA}.
We will prove that ${(\lambda_M)}_\m$ is left representable.
Our proof rests on the following facts:
\begin{enumerate}
\item~\cite[ex. 4.7.2.5]{HA}: Let 
$ U \colon \LMod_A(\C_\m) \to \C_\m $ 
be the forgetful functor and let $X_\bullet\in \LMod_A(\C)$ be a $U$-split simplicial object.
Then $X_\bullet$ admits a colimit in $\LMod_A(\C)$ which is preserved by $U$.
\item~\cite[cor. 4.7.2.11]{HA}: Let \( q \colon \overline{\C} \to \C \) be a right fibration of $\infty$-categories and $X_\bullet$ a simplicial object of $\overline{\C}$. 
If $q(X_\bullet)$ is split in $\C$ then $X_\bullet$ is split in $\overline{\C}$.
\end{enumerate}

Fix an object $X \in \LMod_A(\C_\mM)$. Let
\[
U_A  \colon  \LMod_A(\C_\m) \to \C_\mM
\]
denote the forgetful functor and $f$ its left adjoint. 
By Lurie-Barr-Beck~\cite[thm. 4.7.3.5]{HA}, the functor $U_A$ exhibits \( \LMod_A(\C) \) as monadic over $\C_\mM$. 
Invoking~\cite[prop. 4.7.3.14]{HA} we deduce that there exists a $U_A$-split simplicial object (in the sense of \cite[def 4.7.2.2]{HA}):
\begin{equation*}
\begin{tikzcd}[column sep=tiny]
\Delta_{+}^{op} \ar[d] \ar[r,equal,"\sim"] & (\Delta^{op})^\triangleright \ar[r,"{\overline{X}_\bullet}"] &[2em]  \LMod_A(\C) \ar[d,"U_A"] \\
\Delta_{-\infty}^{op} \ar[rr] & & \C_\m
\end{tikzcd}
\end{equation*}
such that $\overline{X}_{-\infty} = X$ and $\overline X_\bullet$ is a colimit diagram witnessing an equivalence $X \simeq \colim X_\bullet $ between $X$ and the colimit of $X_\bullet := \overline{X}|_{\Delta^{op}_+}$,
and such that each $X_n \in \LMod_A(\C), \, n \neq -\infty$, belongs to the essential image of the functor $f$ (i.e. $X_\bullet$ is a free resolution of $X\in \LMod_A(\C)$).

The pairing ${(\lambda_M)}_\m$ is classified by functors 
\begin{align*}
\chi  \colon  {\LMod_B(\D)}^{op} \to \PSh(\LMod_A(\C)) , \quad  \chi \colon  Y \mapsto \chi_Y \\
\chi'  \colon  {\LMod_A(\C)}^{op} \to \PSh(\LMod_B(\D)) , \quad  \chi' \colon  X \mapsto \chi'_X.
\end{align*}
In particular, $\chi_Y$ is classified by the right fibration 
\begin{align}\label{chi_Y_right_fib}
(\lambda_M)_Y \colon \LMod_M(\M) \times_{\LMod_B(\D)} \set{Y} \to \LMod_A(\C) \times \set{Y}.
\end{align}
We claim that the augmented cosimplicial diagram 
\begin{align} \label{Xop_ch'}
\Delta_{+} \xto{X_\bullet^{op}} {\LMod_A(\C)}^{op} \xto{\chi'} \PSh(\LMod_B(\D))  
\end{align}
is a limit diagram, witnessing an equivalence of presheaves
\begin{align}
\chi'_X \iso \invlim \chi'_{X_\bullet}.
\end{align}

Before establishing this claim, we explain how the proposition follows from the claim. We wish to prove that $\chi'_X$ is representable. Since $\D$ admits totalizations of cosimplicial objects, it is enough to show that each $\chi'_{X_n}$ is representable, which follows from condition (2) together with \cref{LM52232}.  

To establish the claim, it will suffice to check that 
for arbitrary $Y\in \LMod_B(\D)$ 
the induced map
\[
\chi'_X (Y) \to \invlim \chi'_{X_\bullet}(Y) 
\]
is a homotopy equivalence.
Since \( \chi_Y(X) \) and \(\chi'_X(Y) \) are both equivalent to the fiber of ${(\lambda_M)}_\m$ over $(X,Y)$, we  may equally show that
\begin{align}  \label{resolution_comparison_map}
\chi_Y(X) \to \invlim \chi_Y (X_\bullet) 
\end{align}
is a homotopy equivalence of spaces.
Therefore we fix $Y\in \LMod_B(\D)$ for the rest of the proof.
We will prove that \cref{resolution_comparison_map} is a homotopy equivalence using~\cite[prop 5.2.2.39]{HA}, which requires some preparation. 

Write $Y_\m\in \D_\m$ for the image of $Y$ under the forgetful functor $\LMod_B(\D) \to \D_\m$.
We first construct an $\LM$-monoidal $\infty$-category $\M_Y$ exhibiting $\M_\m \times_{\D_\m} \set{Y_\m}$ as left tensored over $\M_\a \times_{\D_\a} \set{B}$.
Let $\M^\otimes_Y$ 
be the pullback 
\begin{equation*}
\begin{tikzcd}
\M_Y^\otimes \ar[d] \ar[r] \pbcorner & \M^\otimes  \ar[d] \\
\LM^\otimes \ar[r,"Y"] & \D^\otimes.
\end{tikzcd}
\end{equation*}
Since $B$ is a trivial algebra we see that $\M^\otimes_Y  \to \LM^\otimes$ is a coCartesian fibration.
Also, note that the horizontal top morphism $\M_Y \to \M$ is an $\LM$-monoidal functor.

Next we express the source of \eqref{chi_Y_right_fib} as a category of left modules in $\M_Y$.
Regard $M$ as an algebra in $(\M_Y)_\a \simeq \M_\a \times_{\D_\a} \set{B}$, and consider the $\infty$-category 
\begin{align*}
\LMod_M(\M_Y)
\simeq 
\LMod_M(\M) \times_{\LMod_B(\D)} \set{Y}.
\end{align*}
We have a commuting square
\begin{align*}
\begin{tikzcd}[ampersand replacement=\&]
\LMod_M(\M_Y) \ar[d,"q"] \ar[r,"U_{\M_Y}"] 
\& 
\M_\m \times_{\D_\m} \set{Y_\m} \ar[d,"q'"] \\
\LMod_A(\C) \ar[r,"U_A"] 
\& 
\C_\m
\end{tikzcd}
\end{align*}
where
\begin{align*}
U_{\M_Y}  \colon  \LMod_M (\M_Y)  \to (\M_Y)_\m = \M_\m \times_{\D_\m} \set{Y_\m} 
\end{align*}
is the forgetful functor,
and $q', q$ are the compositions 
\begin{align*}
& q' \colon \M_\m \times_{\D_\m} \set{Y_\m} \to \M_\m \xto{\lambda_\m} \C_\m \times \D_\m \to \C_\m , \\
& q \colon \LMod_M(\M_Y) \to \LMod_M(\M) \xto{{(\lambda_M)_\m}} \LMod_A(\C) \times \LMod_B(\C) \to \LMod_A(\C).
\end{align*}
Moreover, the right fibration $q$ is equivalent to the right fibration $(\lambda_M)_Y$ of~\eqref{chi_Y_right_fib} and hence classifies the presheaf $\chi_Y$.
To see this, consider the following diagram:
\begin{equation*}
\begin{tikzcd}
\LMod_M(\M_Y) \ar[d,"(\lambda_M)_Y"] \ar[r] \pbcorner & \LMod_M(\M) \ar[d,"(\lambda_M)_\m"] \\
\LMod_A(\C) \times \set{Y} \pbcorner \ar[d] \ar[r] & \LMod_A(\C) \times \LMod_B(\D) \ar[d] \ar[r] & \LMod_A(\C) \\
\set{Y} \ar[r] & \LMod_B(\D).
\end{tikzcd}
\end{equation*}
and note that the composition of the middle row is the identity, and that $q$ is given by the composition of the top path from $\LMod_M(\M_Y)$ to $\LMod_A(\C)$.

We are therefore in the situation of~\cite[cor. 5.2.2.39]{HA}, i.e. the following conditions are equivalent:
\begin{enumerate}
\item The map~\eqref{resolution_comparison_map} is a homotopy equivalence.
\item For every simplicial object $N_\bullet$ in $\LMod_M(\M_Y)$ making the following solid diagram commute
\[
\begin{tikzcd}
\Delta^{op} \ar[d] \ar[r]{}{N_\bullet} & \LMod_M(\M_Y) \ar[d]{}{q} \\
(\Delta^{op})^{\triangleright} \ar[ru, dashed]{}{\exists N^+_\bullet} \ar[r]{}{\ol{X}_\bullet} & \LMod_A(\C)
\end{tikzcd}
\]
there exists an extension to a $q$-colimit diagram $N^+_\bullet$ as indicated.
Equivalently (since $\ol{X}_\bullet$ itself is a colimit diagram) there exists a geometric realization $|N_\bullet|$ which is preserved by the forgetful functor $q$.
\end{enumerate}
Observe that $U_{\M_Y}(N_\bullet)$ is a simplicial object in $\M_\m \times_{\C_\m} \set{Y_\m}$ and the functor $q' \colon \M_\m \times_{\C_\m} \set{Y_\m} \to \C_\m$ is a right fibration.
Since the image 
\[
q'(U_{\M_Y}(N_\bullet)) = U_A ( q(N_\bullet)) = U_A(X_\bullet|_{\Delta^{op}})
\]
is a split simplicial object and $q'$ is a right fibration, it follows from~\cite[cor. 4.7.2.11]{HA} that $U_{\M_Y}(N_\bullet)$ is a split simplicial object in $\M_\mM \times_{\C_\mM} \set{Y_\mM}$.

We therefore have a $U_{\M_Y}$-split simplicial object $N_\bullet$ of
\[
\LMod_M(\M_Y)= \LMod_M \left( \M \times_{\D} \set{Y} \right),
\]
so~\cite[ex. 4.7.2.5]{HA} implies that $N_\bullet$ admits a colimit in $\LMod_M(\M_Y)$ which is preserved by $U_{\M_Y}$.

We claim that the colimit of $\LMod_M(\M_Y)$ is preserved by $q$.
To see this, note that $U_{\M_Y}(N_\bullet)$ is split, so the colimit of $N_\bullet$ is preserved by $q' \circ U_{\M_Y} = U_A  \circ q$.
Applying~\cite[cor. 4.7.2.11]{HA} again, we conclude that the colimit of $N_\bullet$ is preserved by the right fibration $q$.    
It follows that the augmented cosimplicial diagram~\eqref{resolution_comparison_map} is a limit diagram, as claimed. This completes the proof of the proposition.  
\end{proof}

\subsection{Forgetting left module structures and left universal objects}
Let $\C^\otimes$ be an $\LM$-monoidal category and let $A$ be an algebra in $\C_\aA$. 
Assume $\C^\otimes$ admits realizations of $A$-bar constructions $\otimes$-compatibly. 
Then the unit map $ \one \to A $ gives rise to a lax $\LM$-monoidal map of $\LM$-monoidal categories
\[
\BLMod_A(\C)^\otimes \to 
\BLMod_\one(\C)^\otimes \simeq
\C^\otimes
\]
(\cref{w9f2}), hence to a functor 
\[
\LMod\left(\BLMod_A(\C)\right)
\to
\LMod(\C) 
\]
which corresponds under the equivalence of \cref{LModBLMod_as_PB} to forgetting $A$. 

We return to the situation of \cref{sec:left_rep_for_left_mod}
given by an $\LM$-monoidal pairing
\[
\lambda^\otimes \colon  \M^\otimes 
\to 
\C^\otimes \times_{\LM^\otimes} \D^\otimes,
\] 
and an algebra $M\in \Alg(\M_\aA)$ over algebras $(A,B)\in \Alg(\C_\aA)\times \Alg(\D_\aA)$. Assume $\M$ admits realizations of $M$-bar constructions $\otimes$-compatibly, and similarly for $A$. 
Applying functorial fibrant replacement in the model category  ${\POp}_{/\LM^{\otimes, \natural}}$ of $\infty$-preoperads over 
$\LM^{\otimes, \natural}$ 
as in the proof of \cref{BLMod_construction}, we obtain an equivalence of $\LM$-monoidal pairings
\[
\begin{tikzcd}
\M^\otimes \ar[rr, "\simeq"]
\ar[dr]
&& \widetilde \M^\otimes
\ar[dl]
\\
&\C^\otimes \times_{\LM} \D^\otimes
\end{tikzcd}
\]
(in the sense of~\cite[Construction 5.2.1.14]{HA}) and a functor 
\[
\widetilde 
{\BLMod_M(\M)}^\otimes 
\to 
\widetilde \M^\otimes, 
\]
giving rise to a morphism of pairings
\[
\begin{tikzcd}
\LMod \left(\widetilde{\BLMod_M(\M)} \right)
\ar[r] \ar[d]
&
\LMod (\widetilde{\M} ) 
\ar[d]
\\
\LMod \left( \BLMod_A(\C) \right)
\times
\LMod \left( \BLMod_B(\D) \right)
\ar[r]
&
\LMod(\C) 
\times 
\LMod(\D).
\end{tikzcd}
\]
The following proposition, concerning this map of pairings, mirrors~\cite[prop. 5.2.2.30]{HA}. 
\begin{prop} \label{LM52230}
Let $M \in \Alg(\M_\aA)$ be an object over $(A,B)\in \Alg(\C_\aA)\times \Alg(\D_\aA)$. 
Assume that $B\in \Alg(\D_\aA)$ is a trivial algebra. 
Assume $\M$ admits realizations of $M$-bar constructions $\otimes$-compatibly, and similarly for $A$. 
Assume that for every $C \in \C_\aA$, the Kan complex
\[
\lambda_\aA^{-1} \set{ (C, B) } \subseteq \M
\]
is contractible. 
Then (with tilde denoting functorial fibrant replacement in the model category of $\infty$-preoperads as above) the forgetful functor 
\[ 
\LMod \left(\widetilde{\BLMod_M(\M)} \right) \to \LMod (\widetilde{\M} ) 
\] 
carries left universal objects to left universal objects.
\end{prop}
\begin{proof}
In order to lighten notation, we drop the tildes throughout the proof. Let $X\in \LMod( \BLMod_A(\C))$ and let $X_0 \in \LMod(\C)$ be its image under the forgetful functor 
\(
\LMod \left( \BLMod_A(\C) \right) \to \LMod( \C ).
\) 
Consider the commutative diagram 
\[
\begin{tikzcd}
\LMod \left( \BLMod_M(\M) \right) \times_{\LMod(\BLMod_A(\C))} \set{X} \ar{r} \ar[d] &  \LMod ( \M ) \times_{\LMod(\C)} \set{X_0} \ar[d] \\
\LMod \left( \BLMod_B(\D) \right) \times \set{X} \ar[r] & \LMod ( \D ) \times \set{X_0}.
\end{tikzcd}
\]
We have to show that the top horizontal functor preserves terminal objects.
We will show it is an equivalence. 
Since $B\in \Alg(\D_\aA)$ is a trivial algebra, the bottom horizontal map is an equivalence. 
Since the vertical maps are right fibrations, it will suffice to show that for each $Y \in \LMod (\BLMod_B(\D))$ with image $Y_0 \in \LMod_B(\D)$, the induced map 
\begin{align}
\label{LM30_induced_map}
\begin{tikzcd} 
\LMod \left( \BLMod_M(\M) \right) \times_{\LMod(\BLMod_A(\C)) \times \LMod(\BLMod_B(\D))} \set{(X,Y)} \ar{d} \\
\LMod ( \M ) \times_{\LMod(\C) \times \LMod(\D)} \set{(X_0,Y_0)} 
\end{tikzcd}
\end{align}
is a homotopy equivalence. 
We leverage Lurie's proof of~\cite[prop. 5.2.2.30]{HA} to verify this claim.

By \cref{LModBLMod_as_PB} and \cref{UC23} we have an equivalence of $\infty$-categories 
\[
\LMod(\BLMod_A(\C)) \simeq 
\LMod(\C) \times_{\Alg(\C_\aA)} {\Alg(\C_\aA)}_{A/},
\]
and similarly for $\LMod(\BLMod_B(\D))$ and $\LMod(\BLMod_M(\M))$. 
Let
\[
A' \in {\Alg(\C_\aA)}_{A/}, \quad B' \in {\Alg(\D_\aA)}_{B/}
\]
be the images of 
\[
X \in \LMod(\C) \times_{\Alg(\C_\aA)} {\Alg(\C_\aA)}_{A/}, \quad Y \in \LMod(\D) \times_{\Alg(\D_\aA)} {\Alg(\D_\aA)}_{B/},
\]
respectively, and let \( A'_0 \in \Alg(\C_\aA), \, B'_0 \in \Alg(\D_\aA) \) be the images of $A',B'$, respectively. 
Thus for instance, 
\[
A' = (A \to A_0')
\]
is an $A$-algebra, and $X_0$ records the action of $A_0'$ on a left module. 
With this notation, we can identify the domain of~\eqref{LM30_induced_map} with the pullback of the diagram
\[
\begin{tikzcd}
& \LMod(\M) \times_{\Alg(\M_\aA)} \Alg(\M_\aA)_{M/} 
\ar{d}
\\
\set{(X',Y')}  \ar{r}
& 
\left(\LMod(\C) \times_{\Alg(\C_\aA)} \Alg(\C_\aA)_{A/} \right) \times \left( \LMod(\D) \times_{\Alg(\D_\aA)} \Alg(\D_\aA)_{B/} \right).
\end{tikzcd}
\]
The above pullback may be constructed from the following nine-object diagram 
\[
\begin{tikzcd}
\LMod(\M) \ar{d} \ar{r} 
& 
\Alg(\M_\aA) \ar{d} 
& 
\ar{l} \Alg(\M_\aA)_{M/} \ar{d} 
\\
\LMod(\C) \times \LMod(\D) \ar{r} 
& 
\Alg(\C_\aA) \times \Alg(\D_\aA) 
& 
\ar{l} \Alg(\C_\aA)_{A/} \times \Alg(\D_\aA)_{B/} 
\\
\ast \ar{r}{=} \ar{u}{(X_0, Y_0)} 
& 
\ast \ar{u}{(A'_0, B'_0)} 
& 
\ar[l, "="'] \ast \ar{u}{(A',B')} ,
\end{tikzcd}
\]
by first taking pullbacks along the rows. Reversing the order of the limits,
we obtain a homotopy
pullback square 
\begin{equation*}
\begin{tikzcd}[column sep=-10em]
\LMod \left( \BLMod_M(\M) \right) \underset{\LMod(\BLMod_A(\C)) \times \LMod(\BLMod_B(\D))}\times \set{(X,Y)} \ar{dd} \ar{dr} 
\\
& \Alg(\M_\aA)_{M/} \underset{\Alg(\C_\aA)_{A/} \times \Alg(\D_\aA)_{B/} }\times \set{(A',B')} \ar{dd} 
\\
\LMod ( \M ) \underset{\LMod(\C) \times \LMod(\D)}\times \set{(X_0,Y_0)} \ar{dr} 
\\
& \Alg(\M_\aA) \underset{\Alg(\C_\aA) \times \Alg(\D_\aA)}\times \set{(A',B')},
\end{tikzcd}
\end{equation*}
where the left vertical map is equivalent to~\eqref{LM30_induced_map}.
By the proof of~\cite[prop. 5.2.2.30]{HA}, the right vertical map is a homotopy equivalence. 
It follows that~\eqref{LM30_induced_map} is a homotopy equivalence, as claimed.
\end{proof}

\subsection{Koszul duality for $\LM$-monoidal pairings}
\label{subseckdlm}
The following proposition corresponds to~\cite[prop. 5.2.2.27]{HA}.
\begin{thm} \label{LM52227}
Let 
\[
\lambda^\otimes \colon  
\M^\otimes \to 
\C^\otimes \times_{\LM^\otimes} \D^\otimes
\]
be an $\LM$-monoidal pairing which 
satisfies the following properties:
\begin{enumerate}
\item 
If $\one$ denotes the unit object of $\D_\aA$, then the right fibration \( \M_\aA \times_{\D_\aA} \set{\one} \to \C_\aA \) is a categorical equivalence. 
\item 
The underlying pairings $\la_\aA$, $\la_\mM$ are left representable. 
\item 
The $\infty$-categories $\D_\aA, \D_\mM$ admit totalizations of cosimplicial objects. 
\item 
If $B\in \Alg(\D_\aA)$ is a trivial algebra and $M\in\Alg(\M_\aA)$ lies over $(A,B) \in \Alg(\C_\aA) \times \Alg(\D_\aA)$, then $\C$ admits realizations of $A$-bar constructions $\otimes$-compatibly and $\M$ admits realizations of $M$-bar constructions $\otimes$-compatibly.
\end{enumerate}
Then the induced pairing 
\[
\LMod(\lambda)  \colon  \LMod(\M) \to \LMod(\C) \times \LMod(\D)
\]
is left representable. 
\end{thm}

\begin{proof}
Fixing $X\in \LMod(\C)$ a left module over \( A \in \Alg(\C_\aA) \), our goal is to show that there exists a left universal object in $\LMod(\M)$ lying over $X$.   
By \cref{triv_BL_mod_lift}, we can lift $X\in \LMod(\C)$ to a left module $X' \in \LMod \left(\BLMod_A(\C)\right)$ over the trivial algebra $A\in \Alg \big(\BMod{A}{A}(\C_\aA) \big)$.

Let $B$ be a trivial algebra object of $\D_\aA$. Then the fiber \( \M_\aA \times_{\D_\aA} \set{B} \) has an induced monoidal structure and condition (1) implies that the right fibration
\[
\M_\aA \times_{\D_\aA} \set{B} \to \C_\aA,
\]
which upgrades naturally to a monoidal functor, is an equivalence of monoidal $\infty$-categories.
It follows that the pair $(A,B)$ can be lifted to an object $M \in \Alg(\M_\aA)$ in an essentially unique way. 
By \cref{w9f2} and \cref{BLMod_construction}, we then have an $\LM$-monoidal category $ \widetilde \BLMod_{M}(\M)^\otimes$ which witnesses the action of the monoidal $\infty$-category
\[
\widetilde \BLMod_{M}(\M)^\otimes_\a
\simeq 
{_M \mathrm{BMod}_M(\M)^\otimes}
\]
on the $\infty$-category
\[
\widetilde \BLMod_{M}(\M)_\m
\simeq
\mathrm{LMod}_M(\M)
\]
by relative tensor products, as well as an $\LM$-monoidal pairing 
\[
\lambda_{M}  \colon  
\widetilde \BLMod_{M}(\M)^\otimes
\to 
\BLMod_A(\C)^\otimes 
\times_{\LM^\otimes} 
\BLMod_B(\D)^\otimes.
\]
By \cref{LM52240}, the underlying pairings 
\begin{align*}
{(\lambda_{M})}_\a & \colon  \BMod{M}{M}(\M) \to \BMod{A}{A}(\C) \times \BMod{B}{B}(\D)  , \\
{(\lambda_{M})}_\m & \colon  \LMod_{M}(\M) \to \LMod_{A}(\C) \times \LMod_{B}(\D)  
\end{align*}
are both left-representable. 
Since $A$ is a trivial algebra of ${_A \mathrm{BMod}_A(\C)^\otimes}$, we may therefore apply \cref{LM52228} by substituting the $\LM$-monoidal pairing denoted $\la_M$ here for the $\LM$-monoidal pairing denoted $\la$ there. 
In this way, we obtain a left universal object 
\[
Z' \in \LMod \left( \widetilde {\BLMod_{M}(\M)} \right)
\]
lying over $X'$.

Let 
\[
\M^\otimes \to \widetilde{\M}^\otimes 
\to
\C^\otimes \times_{\LM^\otimes}  \D^\otimes
\]
denote the factorization induced by functorial fibrant replacement in the model category of $\infty$-operads as in the preamble to \cref{LM52230}. 
By \cref{LM52230}, the forgetful functor 
\[
\LMod \big( \widetilde {\BLMod_{M}(\M)} \big)
\to \LMod( \widetilde{\M}) 
\]
carries $Z'$ to a left universal object $\widetilde{Z} \in \LMod( \widetilde{\M})$ over $X \in \LMod(\C)$. 
Since the induced functor 
\[
\mathrm{LMod}(\widetilde{\M})
\to
\mathrm{LMod}({\M})
\]
over $\mathrm{LMod}(\C) \times \mathrm{LMod}(\D)$ is an equivalence of pairings, it follows that $\widetilde{Z}$ lifts to a left universal object $Z \in \mathrm{LMod}({\M})$ over $X$, as hoped. 
\end{proof}

\section{The case of the twisted arrow category}
\label{sec:Twisted_arrow_bar_construction}


\Cref{LM52227} applies in a straightforward way to the twisted arrow category of a suitable symmetric monoidal category in which the unit object is terminal. Our first goal in this section (\cref{BarCobarAdjunction}) is to compare the underlying objects of the resulting Koszul duals with the bar construction. Our second goal is to dispense with the assumption regarding the unit object. The final result is stated in \cref{BarCobarAdjunction-XXXversion}.

\subsection{The right duality functor}
As mentioned in \cref{SectionOverview}, we use the right duality functor to identify the underlying objects produced by the left duality functor. Our results concerning the right duality functor are best developed in the general setting of $\LM$-monoidal pairings. We thus continue with $\LM$-monoidal pairings in this subsection before specializing to the twisted arrow category in the next subsection. We begin with an $\LM$-monoidal addendum to~\cite[lem. 5.2.2.33]{HA}.

\begin{lem} \label{LM52233}
Let $\lambda^\otimes \colon \M^\otimes \to \C^\otimes \times_{\LM^\otimes} \D^\otimes$ be an $\LM$-monoidal pairing 
and let $M \in \Alg(\M_\a)$ be an algebra over $(A,B)\in \Alg(\C_\a) \times \Alg(\D_\a)$, where $B$ is a trivial algebra object of $\D_\a$.
Let $\one$ denote the unit object of $\M_\a$, and suppose we are given an augmentation $\epsilon \colon M \to \one$.
Then the induced map
\begin{align*}
\M_\m \simeq \LMod_\one(\M_\m)  \to  \LMod_M(\M_\m) 
\end{align*}
carries right universal objects of $\M_\m$ to right universal objects of $\LMod_M(\M_\m)$.
\end{lem}
\begin{proof}
Fix an object $D\in \D_\m$. 
Let $\M = \M_\a \coprod \M_\m$ denote the fiber of $ \M^\otimes \to \LM^\otimes \to \Fin_*$ over $\langle 1 \rangle \in \Fin_*$, and consider the full subcategory\foot{Here we identify the trivial algebra $B\in \Alg(\D_\a)$ with its underlying object $B \simeq \one_{\D_a}\in\D_\a$.}
\[ 
\M_D \subseteq \M, \quad \M_D := \M_\a \times_{\D_\a} \set{B} \coprod \M_\m \times_{\D_\m} \set{D}.
\]
By~\cite[Proposition 2.2.1.1]{HA} the $\infty$-category $\M_D$ is the underlying $\infty$-category of an $\LM$-monoidal subcategory $\M_D^\otimes \subseteq \M^\otimes$, which exhibits $\M_\m \times_{\D_\m} \set{D}$ as left tensored over $\M_\a \times_{\D_\a} \set{B}$.

We are interested in the restriction of the induced map $\M_\m \to \LMod_M(\M_\m)$ to the fibers over $D$.
Note that $M$ is an algebra object of $\M_\a \times_{\D_\a} \set{B}$, and we have an equivalence 
\[
\LMod_M (\M_D) \simeq \LMod_M( \M_\m) \times_{ \LMod_{B}(\D_\m)} \set{D}.
\]
Moreover, the following diagram commutes  
\[
\begin{tikzcd}[column sep=0.7em]
\M_\m \ar[no head, "\simeq"]{r} & \LMod_\one (\M_\m) \ar[r] & \LMod_M(\M_\m) \\
\M_\m \times_{\D_\m} \set{D} \ar[hook]{u} \ar[no head, "\simeq"]{r} & \LMod_\one (\M_\m) \times_{\LMod_{B}(\D_\m)} \set{D} \ar[hook]{u} \ar[r] & \LMod_M (\M_\m) \times_{\LMod_{B}(\D_\m)} \set{D} \ar[hook]{u} \\
                                          & \LMod_\one (\M_D) \ar[no head, "\simeq"]{u} \ar[r] & \LMod_M ( \M_D) \ar[no head, "\simeq"]{u}.
\end{tikzcd}
\]
The bottom map is the map induced by the augmentation $\epsilon \colon M \to \one$, considered as a map of associative algebras in the fiber $\M_\a \times_{\D_\a} \set{B}$.
By~\cite[cor. 4.2.3.3]{HA}, this map preserves all limits, and in particular terminal objects.
\end{proof}

Following~\cite[variant 5.2.1.16]{HA}, we say that a morphism of right representable pairings
\[
\begin{tikzcd}
\M \ar[r,"\gamma"] \ar[d, "\la"] & \M' \ar[d, "\la' "] 
\\
\C \times \D \ar[r] & \C' \times \D'
\end{tikzcd}
\]
is right representable, if $\gamma$ carries right universal objects to right universal objects. 
We recall from~\cite[Proposition 5.2.1.17]{HA} that under this assumption, the induced diagram of $\infty$-categories
\[
\begin{tikzcd}
\C \ar[d] & \D^\op \ar[d] \ar[l]
\\
\C' & \D'^{\op} \ar[l]
\end{tikzcd}
\]
commutes up to canonical homotopy.

The following proposition mirrors~\cite[cor. 5.2.2.34]{HA}.
\begin{cor}\label{lm52234}
In the situation and the notation of \cref{LM52233}, if
\[
\lambda_\m \colon \M_\m \to \C_\m \times \D_\m
\]
is right representable, then the induced pairing 
\[
\LMod(\lambda) \colon \LMod_M (\M_\m) \to \LMod_A(\C_\m) \times \LMod_B(\D_\m)
\]
is right representable, and the duality functor $\Ddual'_{\LMod(\lambda)}$ can be identified with the composition 
\[
{\LMod_B(\D_\m)}^\op \simeq \D_\m^\op \xto{\Ddual'_\lambda} \C_\m \to \LMod_A(\C_\m) ,
\]
where the last map is induced by the augmentation $\epsilon$.
\end{cor}
\begin{proof}
We begin with the first statement. 
Consider the morphism of pairings given by
\[
\tag{*}
\begin{tikzcd}
\M_\m \ar[d,"\lambda_\m"] \ar[r] & \LMod_M(\M_\m) \ar[d,"\LMod(\lambda)"] \\
\C_\m \times \D_\m \ar[r] & \LMod_A(\C_\m) \times \LMod_B(\D_\m) ,
\end{tikzcd}
\]
where the horizontal maps are induced by the augmentation $\epsilon$. 
Given arbitrary $D \in \D_\m$, there exists a right universal object $Z \in \M_\m$ lying over $D$ by assumption. 
By \cref{LM52233} the image of $Z$ is a right universal object in $\LMod_M(\M_\m)$. 
Since $B\in \Alg(\D_\a)$ is a trivial algebra, the map $\D_\m \to \LMod_B(\D_\m)$ is an equivalence. 
It follows that $\LMod(\lambda)$ is right representable, as stated.

Moreover, it follows that diagram (*) is a right representable morphism of right representable pairings, hence by~\cite[prop. 5.2.1.17]{HA} the diagram
\[
\begin{tikzcd}
\D_\m^{\op} \ar[d] \ar[r,"\Ddual'_\lambda"] & \C_\m \ar[d] \\
\LMod_B(\D_\m)^{\op} \ar[r,"\Ddual'_{\LMod(\lambda)}"] & \LMod_A(\C_\m)
\end{tikzcd}
\]
commutes up to canonical homotopy. 
Since $B\in \Alg(\D_\a)$ is a trivial algebra, the left vertical map is an equivalence, and the proof is complete. 
\end{proof}
%
%
\subsection{The \texorpdfstring{$\LM$}{LM}-monoidal twisted arrow category}


Let $\C$ be an $\infty$-category and $\TwArr(\C)$ the twisted arrow category of~\cite[sec. 5.2.1]{HA}.
In order to apply \cref{LM52227} we construct an $\LM$-monoidal structure on the twisted arrow category, extending~\cite[ex. 5.2.2.23]{HA}.
\begin{prop}
\label{lmtw}
Let $\C^\otimes \to \LM^\otimes$ be an $\LM$-monoidal $\infty$-category.
Then there exists an $\LM$-monoidal pairing 
\[ 
{\TwArr(\C)}^\otimes \to \C^\otimes \times_{\LM^\otimes} {(\C^\op)}^\otimes,
\]
where the $\LM$-monoidal category ${\TwArr(\C)}^\otimes$ witnesses $\TwArr(\C_\m)$ as left tensored over the mo\-noi\-dal category $\TwArr(\C_\a)$.
\end{prop}
\begin{proof}
Identify the $\LM$-monoidal category $\C^\otimes$ with a left action object in the $\infty$-category $\Cat_\infty$, exhibiting the left action of $\C_\a$ on $\C_\m$. 
As in~\cite[ex. 5.2.2.23]{HA}, we can apply the functor 
\[
\Cat_\infty \to CPair^\mathrm{perf}, \quad \E \mapsto \left( \TwArr(\E) \to \E \times \E^\op \right)
\]
of~\cite[rem. 5.2.1.20]{HA} to the underlying $\infty$-categories $\C_\a$ and $\C_\m$.
Following~\cite[prop. 5.2.1.10]{HA}, we see that the pairing $\TwArr(\C_\m) \to \C_\m \times \C_\m^\op$, considered as an object of $CPair^\mathrm{perf}$, admits a left action of $\TwArr(\C_\a) \to \C_\a \times \C_\a^\op$, considered as a monoid object in $CPair^\mathrm{perf}$.
We can now identify this left action object with the desired  $\LM$-monoidal pairing 
\(
{\TwArr(\C)}^\otimes \to \C^\otimes \times_{\LM^\otimes} {(\C^\op)}^\otimes. 
\)
\end{proof}
\begin{rem}
The coCartesian fibration \( {(\C^\op)}^\otimes \to \LM^\otimes \) appearing in the previous proposition can be identified with the ``fiberwise opposite'' of $\C^\otimes \to \LM^\otimes$.
Explicitly, the fiberwise opposite is given by applying the dualizing construction of~\cite{DualizingCarFibs} to the coCartesian fibration $\C^\otimes \to \LM^\otimes$, then taking the opposite functor.
\end{rem}

\subsection{Bar constructions in the twisted arrow category}
Denote by $s$ and $t$ the \emph{source} and \emph{target} maps
\[
s \colon   \TwArr(\C) \to \C \times \C^\op \to \C, \quad
t \colon \TwArr(\C) \to \C \times \C^\op   \to \C^\op
\]
of the twisted $\infty$-category.

\begin{prop}\label{prop:colimits_in_TwArr}
Let $\C$ be an $\infty$-category, and let $K$ be a weakly contractible simplicial set. 
\begin{enumerate}
\item 
Let  
\[
p \colon K \to \TwArr(\C) 
\]
be a diagram in the twisted arrow category. 
If $sp$ has a colimit in $\C$ and $tp$ is equivalent to a constant diagram, then $p$ admits a colimit in $\TwArr(\C)$.
\item
Let
\[
\overline p: K^\triangleright \to \TwArr(\C)
\]
be a diagram in the twisted arrow category. 
If $s \overline p$ is a colimit diagram and $t \overline p$ is equivalent to a constant diagram, then $\overline p$ is a colimit diagram. 
\end{enumerate}
\end{prop}
\begin{proof}
We begin with part (2). 
Without loss of generality, assume that $t\overline p$  is constant with value $E$. 
We can therefore consider 
\[
\ol p \colon K^\triangleright \to {\TwArr(\C)} \times_{\C^\op} \set{E}.
\]
as a diagram in the fiber of $t$ over $E$.
By~\cite[prop. 5.2.1.8]{HA} applied to $id_E \in \TwArr(\C)$ and~\cite[prop. 4.4.4.5]{HTT}, ${\TwArr(\C)} \times_{\C^\op} \set{E}$ is equivalent to the slice category $\C_{/E}$.
Therefore, by~\cite[prop. 1.2.13.8]{HTT}, $\ol p$ is a colimit diagram in the fiber ${\TwArr(\C)} \times_{\C^\op} \set{E}$.

By~\cite[prop. 5.2.1.3]{HA}, $t$ is a right fibration. 
So by~\cite[cor. 4.3.1.16]{HTT} $\ol{p}$ is a $t$-colimit diagram in $\TwArr(\C)$. 
Since $K$ is weakly contractible, the constant diagram $t \circ \ol{p}$ is a colimit diagram (see~\cite[cor. 4.4.4.10]{HTT}). 
Therefore $\ol{p}$ is a colimit diagram in $\TwArr(\C)$ by combining~\cite[prop. 4.3.1.5 (2)]{HTT} with~\cite[ex. 4.3.1.3]{HTT}.

We now show (1). 
The projection $s \circ p$ has a colimit in $\C$ by assumption, which by~\cite[prop. 1.2.13.8]{HTT} lifts to 
\[
\ol{p} \colon K^\triangleright \to \TwArr(\C) \times_{\C^\op} \set{E},
\]
a colimit of $p$ in $\C_{/E} \simeq {\TwArr(\C)} \times_{\C^\op} \set{E}$.
By (2), $\ol{p}$ is also a colimit diagram in $\TwArr(\C)$. 
\end{proof}

\begin{cor}\label{prop:Bar_const_in_TwArr}
Let $\C^\otimes \to \LM$ be an $\LM$-monoidal $\infty$-category, and let
\[
M \in \Alg(\TwArr(\C_\aA))
\]
be an algebra lying over $(A,B)\in \Alg(\C_\aA) \times \Alg(\C_\aA^\op)$. 
Assume $B$ is a trivial algebra and $\C^\otimes$ admits realizations of $A$-bar constructions $\otimes$-compatibly (\cref{m24j0} applied to $\C^\otimes$ regarded as a $\BM$-monoidal category as in \cref{LM_to_BM_monoidal_construction}).
Then $\TwArr(\C)^\otimes$ admits realizations of $M$-bar constructions $\otimes$-compatibly.
\end{cor}

\begin{proof}
We regard all $\LM$-monoidal categories as $\BM$-monoidal (\cref{LM_to_BM_monoidal_construction}). 
In order to simplify the ensuing division into cases, we introduce some notation. We define a partial composition law on the set $C=\{0,\mM, 1\}$ of colors of $\BM^\otimes$ by
\[
0^2 = 0
\quad
0 \mM = \mM
\quad
\mM 1 = \mM
\quad
1^2 = 1
\]
and a relation by $i <j$ if the product $ij$ is defined. 

We first verify condition (1) of \cref{m24j0} (existence of realizations). 
Let $(X,Y)$ be an $(L,M,R)$-module in $\TwArr(\Cc)^\otimes$ with $L \in \Alg \TwArr(\Cc_0)$ and $R \in \Alg \TwArr(\Cc_
r)$ ($r = 0$ or $1$). 
There's a unique color $i \in C$ with $0<i<r$. 
Moreover, $Y$ belongs to $\TwArr(\Cc_i)$ and $0i=i$ so that for all $n \ge 0$, $X \otimes M^{\otimes n} \otimes Y$ belongs again to $\TwArr(\Cc_i)$. 
Consider the diagrams
\[
\Delta^\op \xto{{\Bar_M(X,Y)}_\bullet} \TwArr(\Cc_i) \xto{s} \Cc_i,
\quad
\Delta^\op \xto{{\Bar_M(X,Y)}_\bullet} \TwArr(\Cc_i) \xto{t} \Cc_i^\op.
\]
The diagram $t \circ \Bar_M(X,Y)_\bullet$ is equivalent to a constant diagram since $B\in \Alg(\C_\aA^\op)$ is trivial. On the other hand, the diagram $s \circ \Bar_M(X,Y)_\bullet$ admits a colimit in $\Cc_i$ by assumption, since it is equivalent to an $A$-bar construction in $\Cc_i$. Therefore $\Bar_M(X,Y)_\bullet$ has a colimit in $\TwArr(\Cc_i)$ by \cref{prop:colimits_in_TwArr}

Next we verify condition (2) (realizations of $M$-bar constructions are compatible with tensor products). We continue to work with an $(L,M,R)$-module $(X,Y)$ with $R \in \Cc_r$ and $Y \in \Cc_i$ as above. Let $j>i$ and fix $X' \in \TwArr(\C_\aA)$, $Y' \in \TwArr(\C_j)$ so that the tensor product $Y \otimes Y'$ is defined and lies in $\Cc_{ij}$. Let 
\[
{\Bar_M(X,Y)}_\bullet^+ 
\colon 
\Delta^\op_+ 
\to 
\TwArr(C_i)
\]
be a colimit diagram extending the bar construction $\Bar_M(X,Y)_\bullet$. Consider the diagram 
\[
\ol p \colon 
\Delta^\op_+ \xto{{\Bar_M(X,Y)}_\bullet^+} \TwArr(\C_i) 
\xto{X' \otimes - \otimes Y'} 
\TwArr(C_{ij})
\] 
defined by applying the functor $X' \otimes - \otimes Y'$ to  ${\Bar_M(X,Y)}_\bullet^+$.
We claim that $\ol p$ is a colimit diagram. Since $B$ is a trivial algebra, the diagram 
\(
{\Bar_B(t(X),t(Y))}_\bullet^+
\)
is equivalent to a constant diagram with value $t(X) \otimes t(Y)$, hence $t \circ \ol{p}$ is equivalent to the constant diagram with value $t(X') \otimes t(X) \otimes t(Y) \otimes t(Y')$. Since realizations of $A$-bar constructions in $\C$ are by assumption $\otimes$-compatible, $s \circ \ol p$ is a colimit diagram. By \ref{prop:colimits_in_TwArr}(2), $\ol p$ is a colimit diagram, as desired.
\end{proof}

In the proof of \cref{BarCobarAdjunction} below, we compare the Koszul dual $\Ddual_\mathrm{Ksz}(A,X)$ with the relative tensor product $\one \otimes_A X$ by showing that (after forgetting the comodule structure on the former) both fulfill the same universal mapping property in $\C_\mM$. For this purpose, following \cite[Def. 5.2.2.1]{HA}, we define the ``bar object'' $\Bar(A,X)$ purely in terms of the universal mapping property. As noted by Lurie, the bar object can exist even in the absence of relative tensor products. 

\begin{mydef}\label{DefBar}
Let $\C^\otimes$ be an $\LM$-monoidal category with unit object $\one \in \C_\a$ and let
\[
\ep: A \to \one
\]
be a morphism in $\Alg(\C_\a)$. Then there are induced functors 
\[
\rho^\ep_\a: \C_\a \simeq
\BMod{\one}{\one}(\C_\a)
\to
\BMod{A}{A}(\C_\a),
\quad
\rho^\ep_\m:
\C_\m \simeq
\LMod_\one(\C_\m) \to \LMod_A(\C_\m).
\]
Let $C\in \C_\a$.
A morphism $f: A \to \rho^\ep_\a(C)$ in $\BMod{A}{A}(\C_\a)$ is said to \emph{exhibit $C$ as the bar object on $A\to \one$} (or simply \emph{on $A$}) if, for every object $D \in \C_\a$, composition with $f$ induces a homotopy equivalence 
\[
\Map_{\C_\a}(C,D) \xto{\sim} 
\Map_{\BMod{A}{A}(\C_\a)}(A, \rho^\ep_\a(D)).
\]
When such an object $C$ exists it is uniquely determined up to equivalence and denoted by $\Bar(A)$.

Similarly, if $X \in \LMod_A(\C_\m)$ and $Y\in \C_\m$,
we say that a morphism $g: X \to \rho^\ep_\m(Y)$ in $\LMod_A(\C_\m)$ \emph{exhibits $Y$ as the bar object on $(A,X)$} if, for every object $Z \in \C_\m$, composition with $g$ induces a homotopy equivalence 
\[
\Map_{\C_\m}(Y,Z) \xto{\sim} 
\Map_{\LMod_{A}(\C_\m)}(X, \rho^\ep_\m(Z)).
\]
When such a $Y$ exists it is uniquely determined up to equivalence and we denote it by $\Bar(A,X)$.
\end{mydef}

\begin{rem}
\label{Bar_def_as_Bar_construction}
Let $\C^\otimes$ be an  $\LM$-monoidal $\infty$-category which admits realizations of bar constructions $\otimes$-compatibly. 
It follows from~\cite[prop. 4.6.2.17]{HA} that if $\ep:A \to \one$ is an augmented associative algebra object of $\C_-$, then the forgetful functor 
\[
\rho^\ep_\m: \C_\mM \simeq \LMod_\one(\C)
\to 
\LMod_A(\C)
\]
admits a left adjoint, given by the relative tensor product
\[
X \mapsto \one \otimes_A X.
\]
It follows that the \textit{bar object} $\Bar(A,X)$ exists and is given by the relative tensor product 
\[
\Bar(A,X) = \one \otimes_A X.
\]
\end{rem}

The next theorem mirrors~\cite[thm. 5.2.2.17]{HA}.

\begin{thm} 
\label{BarCobarAdjunction}
Let $\C^\otimes \to \LM^\otimes$ be an $\LM$-monoidal $\infty$-category, and let
\[
{\lambda: \TwArr(\C)}^\otimes \to \C^\otimes \times_{\LM^\otimes} {(\C^\op)}^\otimes
\]
be the associated $\LM$-monoidal pairing of \cref{lmtw}.
\begin{enumerate}
\item
The induced map
\[
\LMod(\la) \colon
\LMod(\TwArr(\C)) \to
\LMod(\C) \times \LMod(\C^\op)
\]
is a pairing of $\infty$-categories. 
\item
Assume that the unit object $\one \in \C_\aA$ is final and that $\C_\aA, \C_\mM$ admit geometric realizations of simplicial objects. 
Assume, moreover, that $\C^\otimes$ admits realizations of bar constructions $\otimes$-compatibly (\cref{m24j0}). 
Then the pairing $\LMod(\la)$ is left representable, and therefore determines a functor 
\[
\Ddual^\op_\mathrm{Ksz}:=\Ddual_{\LMod(\la)}: {\LMod(\C)}^\op \to \LMod(\C^\op).
\]
Moreover, the composition with the forgetful functor 
\[
{\LMod(\C)}^\op \xto{ \Ddual^\op_\mathrm{Ksz} } \LMod(\C^\op) \to
\C^\op_\a \times \C^\op_\m 
\]
is given by the bar objects
\[
(A,X) \mapsto (\Bar(A), \Bar(A,X))
\]
as defined in \cref{DefBar}.
\item
Dually, assume that the unit object $\one \in \C_\aA$ is initial and that $\C$ admits totalizations of cosimplicial objects. 
Assume, moreover, that  $\C$ admits totalizations of cobar constructions $\otimes$-compatibly (\cref{m24j0}).
Then the pairing $\la$ is right representable, and therefore determines a functor 
\[
\Ddual^{'\op}_\mathrm{Ksz}:=\Ddual'_{\LMod(\la)}: {\LMod(\C^{\op})}^{\op} \to \LMod(\C).
\]
Moreover, the composition with the forgetful functor
\[
{\LMod(\C^{\op})}^{\op} \xto{ \Ddual^{'\op}_\mathrm{Ksz} } \LMod(\C) \to
\C_\a \times \C_\m  
\]
is given by
\[
(C,Y) \mapsto  ( \opnm{CoBar}(Y), \opnm{CoBar}(C,Y) ),
\]
where $\opnm{CoBar}(Y)$ and $\opnm{CoBar}(C,Y)$ denote  bar objects in $\C^{\op}$.
\end{enumerate}
\end{thm}

\begin{proof}[Proof of \cref{BarCobarAdjunction}]
(1) is a special case of~\cite[rem. 5.2.2.26]{HA}. 
We turn to (2). 
Let $\D = \C^\op$. We claim that the $\LM$-monoidal pairing 
\[
\la^\otimes \colon
{\TwArr(\C)}^\otimes \to
\C^\otimes \times_{\LM} {\D}^\otimes
\]
satisfies the hypotheses of \cref{LM52227}:
\begin{description}
\item[\ref{LM52227}(1)] 
The map \( \M_\aA \times_{\D_\aA} \set{\one} \to \C_\aA \) is an equivalence since $\one \in \C_\aA$ is terminal by assumption.
\item[\ref{LM52227}(2)] 
The underlying pairings $\la_\aA$, $\la_\mM$ are left representable by~\cite[prop. 5.2.1.10]{HA}.
\item[\ref{LM52227}(3)] 
The $\infty$-categories $\D_\aA, \D_\mM$ admit totalizations by assumption. 
\item[\ref{LM52227}(4)] 
Let $B \in \Alg(\D_\aA)$ be a trivial algebra and $M \in \Alg(\M_\aA)$ an algebra lying over $(A,B) \in \Alg(\C_\aA) \times \Alg(\D_\aA)$. 
By assumption, $\C$ admits realizations of $A$-bar constructions $\otimes$-compatibly. 
Since $B\in \Alg(\C_\aA^\op)$ is a trivial algebra, \cref{prop:colimits_in_TwArr} implies that $\TwArr(\C)$ admits realizations of $M$-bar constructions $\otimes$-compatibly. 
\end{description}
We thus have a duality functor $\Ddual_{\LMod \la}$ as above.

Let $(A,X)$ be a left module in $\C$ and let
\[
(B,Y) = \Ddual_{\LMod \la}(A,X) \in \LMod(\C^{\op}).
\]
By~\cite[thm. 5.2.2.17 (2)]{HA}, the image of $B$ in $\C_\a$ is equivalent to $\Bar(A)$. We must show that the image of $Y$ in $\C_\m$ is equivalent to $\Bar(A,X)$.
Since the unit object $\one \in \C_\a$ is final, there exists an essentially unique augmentation
\[
\ep:A \to \one.
\]
We can then identify $\ep$ with an algebra object of the monoidal $\infty$-subcategory
\[
{(\C_\a)}_{/\one} \simeq 
\TwArr(\C_\a) \times_{\C_\a^{\op}} \set{ \one }
\subset \TwArr(\C_\a).
\]
Let $M$ denote the associated algebra object in $\TwArr(\C_\a)$:
\[
\TwArr(\C_\a) \to \C_\a \times \C_\a^\op, \quad M \mapsto (A, \one).
\]
By \cref{LM52240} and \cref{lm52234}, the induced pairing $\LMod_M(\la)$:
\[
\LMod_M(\TwArr(\C_\m)) \to
\LMod_A(\C_\m) \times
\LMod_\one(\C^{\op}_\m) 
\simeq
\LMod_A(\C_\m) \times
\C^{\op}_\m 
\]
is both left and right representable, hence induces adjoint duality functors
\[
\Ddual^\op_{\LMod_M (\la)}:
\LMod_A(\C_\m) \leftrightarrows
{(\C_\m^{\op})}^{\op} \simeq \C_\m
:\Ddual'_{\LMod_M (\la)}
\]
as in \cref{DualityFromPairing}.
By \cref{lm52234}, $\Ddual'_{\LMod_M (\la)}$ is the functor induced by the augmentation of $A$. 
It follows that its left adjoint $\Ddual_{\LMod_M (\la)}$ is given by
\[
X' \mapsto  \Bar(A, X').
\]
By the construction of the $\LM$-Koszul duality functor $\Ddual_{\LMod(\la)}$ (outlined in the introduction and carried out in the proof of \cref{LM52227}), as an object of the $\infty$-category $\C^{\op}_\a \times \C_\m^{\op}$, we can identify $(B,Y)$ with 
\[
\big(\Ddual_{\BMod{M}{M}(\la)}(A), 
\Ddual_{\LMod_M(\la)}(X) \big), 
\]
hence with $(\Bar(A), \Bar(A,X))$ as claimed. 
Statement (3) follows by symmetry. 
\end{proof}

\subsection{Dispensing with the assumption that the unit object is terminal}

\begin{construction}
\label{cons:C_1}
Let $\C^\otimes$ be an $\LM$-monoidal category witnessing the $\infty$-category $\C_\m$ as left tensored over the monoidal $\infty$-category $\C_\a$, and let $\one \in \C_\a$ denote a unit object. We construct a new $\LM$-monoidal category $\C_{/\one}^\otimes$ which witnesses $\C_\m$ as left tensored over the over-category $\C_{/\one}^\otimes$ with its induced monoidal structure. Recall that $(\C_\a)_{/\one}^\otimes$ comes together with monoidal functor 
\[
\epsilon:(\C_\a)_{/\one} \to \C_\a
\]
\[
(X \to \one) \mapsto X.
\]
We may regard $\C^\otimes$ as an object of the $\infty$-category $\LMod(\Cat_\infty)$ of left modules in the $\infty$-category $\Cat_\infty$ of $\infty$-categories (with its cartesian symmetric monoidal structure), and we may regard $\ep$ as a morphism in the $\infty$-category $\Alg(\Cat_\infty)$ of algebras in $\Cat_\infty$. The Cartesian fibration
\[
\LMod(\Cat_\infty) \to \Alg(\Cat_\infty)
\]
\[
(\E_\a, \E_\m) \mapsto \E_\a
\]
gives rise to a pullback functor
\[
\epsilon^*: 
\LMod_{{(\C_\a)}_{/\one}}(\Cat_\infty)
\from
\LMod_{\C_\a}(\Cat_\infty).
\]
We let
\[
\Cc_{/\one}^\otimes := \ep^*\C^\otimes.
\]
We denote the $\infty$-category of $\LM$-algebras in $\C^\otimes_{/ \one}$ by  
\[
\LMod^\mathrm{aug}(\C): = \LMod( \C_{/\one}).
\]
Evidently, the unit object in ${(\C_\a)}_{/\one}$ is terminal.

Similarly, applying the construction above to (the opposite of) the symmetric monoidal functor  
\[
{(\C_\a)}_{\one // \one} :=  {\left( {(\C_\a)}_{/\one} \right)}_{\one /}  \to {(\C_\a)}_{/\one} 
\]
constructs an $\LM$-monoidal category $\C^\otimes_{\one / /\one}$ which exhibits $\C_\m$ as tensored over $\C_{ \one / / \one}$. We remark that
\[
\LMod(\C_{\one//\one}) \simeq \LMod^\mathrm{aug}(\C).
\]

\begin{mydef}
Define $\opnm{coLMod}^\mathrm{aug}(\C) := {\LMod^\mathrm{aug}(\C^{\op})}^{\op}$.
\end{mydef}
Informally, one can describe an object of $\operatorname{coLMod}^\mathrm{aug}(\C)$ 
as a pair $(C,Y)$ of an augmented coalgebra $C$ in $\C_\a$ and a left comodule $Y$ over $C$. 
\end{construction}

\begin{cor}
\label{BarCobarAdjunction-XXXversion}
Let $\C^\otimes \to \LM^\otimes$ be an $\LM$-monoidal $\infty$-category. There is an associated $\LM$-monoidal pairing
\[
\la\colon {\TwArr(\C_{\one//\one})}^\otimes 
\to
\C^\otimes_{\one//\one} 
\times_{\LM^\otimes} 
{(\C_{\one//\one}^\op)}^\otimes
\]
and an induced pairing of $\infty$-categories
\[
\LMod(\la) \colon
\LMod(\TwArr(\C_{\one//\one})) \to
\LMod(\C_{\one//\one}) \times \LMod(\C^\op_{\one//\one}).
\]
If $\Cc_\aA$, $\Cc_\mM$ admit geometric realizations of simplicial objects and $\C^\otimes$ admits geometric realizations of bar constructions $\otimes$-compatibly, then the pairing $\LMod(\la)$ is left representable, and therefore determines a functor 
\[
\Ddual_\mathrm{Ksz} =    \Ddual_{\LMod(\la)}^\op \colon \LMod^\mathrm{aug}(\C) \to \opnm{coLMod}^\mathrm{aug}(\C) ,
\]
whose composition with the forgetful functor 
\[
\opnm{coLMod}^\mathrm{aug}(\C) \to
\C_\a \times \C_\m
\]
is given by
\[
(A,X) \mapsto (\one \otimes_A \one, \one \otimes_A X).
\]

If $\C_\a, \C_\m$ admit both realizations of simplicial objects and totalization of cosimplicial objects, and if $\Cc^\otimes$ admits realizations of bar constructions and totalizations of cobar constructions $\otimes$-compatibly, then the induced duality functors form an adjunction
\[
\Ddual_\mathrm{Ksz} \colon \LMod^\mathrm{aug}(\C) \adj \operatorname{coLMod}^\mathrm{aug}(\C) \noloc \Ddual'_\mathrm{Ksz}.
\]
\end{cor}

\begin{proof}
If $\C_\a$, $\C_\m$ admit realizations of simplicial objects and $\C^\otimes$ admits geometric realizations of bar constructions $\otimes$-compatibly, then the same holds with $\C_\a$ replaced by $(\C_\a)_{\one//\one}$ and $\C^\otimes$ replaced by $\C^\otimes_{\one//\one}$. A similar remark applies to totalizations of cosimplicial objects / cobar constructions. The theorem follows by combining \cref{BarCobarAdjunction} and \cref{Bar_def_as_Bar_construction}.
\end{proof}

\bibliographystyle{plain}
\bibliography{references}

\vfill

\Small\textsc{
ID: Department of Mathematics, Ben-Gurion University of the Negev.
}

\texttt{ishaida@bgu.ac.il}

\Small\textsc{
AH: Work carried out at Ben Gurion University of the Negev and at Stockholm University. 
}

\texttt{asafklm@gmail.com}

\end{document}